\documentclass[11pt]{amsart}
\usepackage[utf8]{inputenc}
\usepackage[T1]{fontenc}
\usepackage{ae,aecompl}
\usepackage[margin=1in]{geometry}
\usepackage{graphicx}
\usepackage{amsmath, accents}
\usepackage{amsfonts}
\usepackage{amsthm}
\usepackage{tikz}
\usepackage{ytableau}
\usepackage{caption}
\usetikzlibrary{decorations.markings}
\usepackage{microtype}
\usepackage[bookmarks=false]{hyperref}
\usepackage{subcaption}

\usepackage[toc,page]{appendix}
\usepackage{minibox}

\makeatletter
\newcommand\xleftrightarrow[2][]{%
  \ext@arrow 9999{\longleftrightarrowfill@}{#1}{#2}}
\newcommand\longleftrightarrowfill@{%
  \arrowfill@\leftarrow\relbar\rightarrow}
\makeatother

\graphicspath{ {graphics/} }

\newcommand\myeq{\mathrel{\stackrel{\makebox[0pt]{\mbox{\normalfont\tiny def}}}{=}}}

\theoremstyle{definition}
\newtheorem{thm}{Theorem}[section]
\newtheorem{lem}[thm]{Lemma}
\newtheorem{df}[thm]{Definition}

\newtheorem{prop}[thm]{Proposition}
\newtheorem{cor}[thm]{Corollary}

\newtheorem{rem}[thm]{Remark}

\newcommand{\C}{\mathbb{C}}

\newcommand{\R}{\mathbb{R}}
\newcommand{\Z}{\mathbb{Z}}
\newcommand{\Zp}{\mathbb{Z}_+}
\newcommand{\N}{\mathbb{N}}
\newcommand{\ii}{\mathbf{i}}

\newcommand{\U}{\mathcal{U}}

\newcommand{\CC}{\mathcal{C}}

\newcommand{\TT}{\mathbb{T}}

\newcommand{\GT}{\mathbb{GT}}
\newcommand{\GTp}{\mathbb{GT}^+}

\newcommand{\bfw}{\mathbf{w}}
\newcommand{\bfx}{\mathbf{x}}
\newcommand{\bfy}{\mathbf{y}}

\numberwithin{equation}{section}

\begin{document}

\title{Pieri Integral Formula and Asymptotics of Jack Unitary Characters}

\author{Cesar Cuenca}

\address{Department of Mathematics, MIT, Cambridge, MA, USA}
\email{cuenca@mit.edu}

\date{}

\begin{abstract}
We introduce Jack (unitary) characters and prove two kinds of formulas that are suitable for their asymptotics, as the lengths of the signatures that parametrize them go to infinity. The first kind includes several integral representations for Jack characters of one variable. The second identity we prove is the Pieri integral formula for Jack characters which, in a sense, is dual to the well known Pieri rule for Jack polynomials. The Pieri integral formula can also be seen as a functional equation for irreducible spherical functions of virtual Gelfand pairs.

As an application of our formulas, we study the asymptotics of Jack characters as the corresponding signatures grow to infinity in the sense of Vershik-Kerov. We prove the existence of a small $\delta > 0$ such that the Jack characters of $m$ variables have a uniform limit on the $\delta$-neighborhood of the $m$-dimensional torus. Our result specializes to a theorem of Okounkov and Olshanski.
\end{abstract}

\maketitle

\setcounter{tocdepth}{1}
\tableofcontents

\section*{Introduction}\label{introduction}

A signature of length $N$ is a nonincreasing integer sequence $\lambda = (\lambda_1 \geq \lambda_2 \geq \dots \geq \lambda_N)$, and $\GT_N$ denotes the set of all such signatures. One can define the \textit{Jack (Laurent) polynomial} $J_{\lambda}(x_1, \ldots, x_N; \theta)$ associated to $\lambda\in\GT_N$ and with parameter $\theta > 0$, \cite{M, St}.
In the literature, the deformation parameter of Jack polynomials is sometimes $\alpha$ instead of $\theta$; both parameters are related by $\theta = 1/\alpha$.
For $\theta = 1$, the Jack polynomials are the well known \textit{Schur polynomials}.

The main object of our study are the \textit{Jack unitary characters}.
We define the Jack unitary character of rank $N$, with $m$ variables and parametrized by $\lambda$ to be
\begin{equation*}
J_{\lambda}(x_1, \ldots, x_m; N, \theta) := \frac{J_{\lambda}(x_1, \ldots, x_m, 1^{N-m}; \theta)}{J_{\lambda}(1^N; \theta)},
\end{equation*}
where $1^k$, $k\in\N$, stands for a string of $k$ ones.
Our terminology is based on the fact that, when $\theta = 1$ and $N = m$, the Jack unitary character $J_{\lambda}(x_1, \dots, x_N; N, \theta)$ becomes the normalized character of the unitary group $U(N)$ corresponding to its irreducible representation parametrized by the highest weight $\lambda$.
Therefore Jack unitary characters are natural one-parameter deformations of the normalized characters of the unitary groups.
We will call Jack unitary characters simply \textit{Jack characters}, for convenience.
The goal of this paper is to establish several formulas for Jack characters that are suitable to study their asymptotics, and we also work out in detail one asymptotic theorem about Jack characters, by using our toolbox.

The paper \cite{C1} is this article's companion; there we introduced \textit{Macdonald characters}, which depend on the well known Macdonald polynomials just like Jack characters depend on Jack polynomials.
We proved several formulas for them and studied their asymptotics as the signatures grow to infinity, while the number of variables remains finite.
As a consequence of our asymptotic results, we characterized the \textit{boundary of the $(q, t)$-Gelfand-Tsetlin graph}, a generalization of one of the main results of Gorin's article \cite{G}.
It was natural to consider the degeneration of the parameters $q, t$ used for Macdonald characters in the ``Jack regime'' $t = q^{\theta}$, $q\rightarrow 1$, and ask whether the formulas from \cite{C1} degenerate into formulas for Jack characters. It will be shown that certain formulas do admit a degeneration, but not all: the \textit{multiplicative formula for Macdonald characters}, \cite[Thm. 4.1]{C1}, cannot be degenerated. One of the main results of this paper is the \textit{Pieri integral formula}, which effectively replaces the multiplicative formula in our setting and is suitable to study the asymptotics of Jack characters.

We make a final remark on terminology. The name \textit{Jack character} has been used before to denote certain natural one-parameter deformations of the characters of the \textit{symmetric groups} $S(N)$. The Jack symmetric group characters were introduced in Lassalle's work, \cite{L}, though he did not use that terminology.
In recent years, Maciej Doł\k{e}ga, Valentin F\'{e}ray and Piotr \'{S}niady further studied Jack symmetric group characters, their structural theory and asymptotics, see e.g. \cite{DF13, DF16, DFS, Sn}. Their work and ours appear to be in different directions.

The contents of this paper are divided into two parts. In the first part, we find integral representations for Jack characters of one variable and arbitrary rank, and we also prove the Pieri integral formula.
In the second part, we make use of the formulas obtained in the first part to study the limits of Jack characters of a fixed number of variables as the signatures $\{\lambda(N)\}_{N\geq 1}$ grow to infinity in the sense of Vershik-Kerov.
We proceed with more details on each of the two parts of this paper. For the reader's convenience, all definitions in the introduction are repeated later in the text.

\subsection{Formulas for Jack characters}

The formulas we obtain in this paper fall into two categories:

(A) Integral representations for Jack characters of one variable and arbitrary rank $N$.

The integral representations in this article are degenerations of formulas for Macdonald characters, \cite{C1}. An example is the following.

\begin{thm}[Theorem $\ref{jackthm2}$ below]\label{thm:integralsample}
Let $\theta > 0$, $N\in\N$, $\lambda\in\GT_N$ and $|x| < 1$. Then the integral below converges absolutely and the identity holds
\begin{equation*}
J_{\lambda}(x; N, \theta) = -\frac{\Gamma(\theta N)}{(1 - x)^{\theta N - 1}}\frac{1}{2\pi\ii}\int_{\CC^+}{x^z\prod_{i=1}^N{\frac{\Gamma(\lambda_i + \theta(N-i)-z)}{\Gamma(\lambda_i + \theta(N-i+1)-z)}}dz},
\end{equation*}
where the positively oriented contour $\CC^+$ consists of the segment $[-M+r\ii, \ -M-r\ii]$ and horizontal lines $[-M+r\ii, +\infty+r\ii)$, $[-M-r\ii, +\infty-r\ii)$, for any $M > -\lambda_N$ and any $r > 0$, see Figure $\ref{fig:C}$.
\end{thm}

(B)  Formulae expressing the product $J_{\lambda}(x_1, \ldots, x_{m-1}; N, \theta)\cdot J_{\lambda}(x; N, \theta)$ as a (continual) convex combination of Jack characters $\{J_{\lambda}(x_1w_1, \ldots, x_{m-1}w_{m-1}, x(w_1\cdots w_{m-1})^{-1}; N, \theta) : (w_1, \dots, w_m)\in\U_x \}$ of $m$ variables, where $\U_x \subset \R^m$ is the compact subset defined in $(\ref{defn:Ax})$.
An example is given by the following theorem.

\begin{thm}[reformulation of Theorem $\ref{thm:pieri}$ below for $m = 1$]\label{thm:pierisample}
Let $\theta > 0$, $N\in\N$, $N \geq 2$, $\lambda\in\GT_N$, and any real numbers $1 > y > x > 0$. Then
\begin{gather*}
\begin{gathered}
J_{\lambda}(x; N, \theta)J_{\lambda}(y; N, \theta) = \frac{\Gamma(N\theta)y^{\theta}}{\Gamma((N-1)\theta)\Gamma(\theta)((1-x)(1 - y))^{\theta}}\\
\times\int_y^1{ J_{\lambda}\left( xw, \frac{y}{w}; N, \theta \right) \left(1 - \frac{xw^2}{y} \right)\left(1 - \frac{xw}{y}\right)^{\theta - 1} \left(\frac{1}{w} - 1\right)^{\theta - 1} \left( \frac{(1 - xw)(1 - \frac{y}{w})}{(1-x)(1-y)} \right)^{\theta(N-1)-1} \frac{dw}{w^2}}.
\end{gathered}
\end{gather*}
\end{thm}
At first sight, the identity above is mysterious because the left side is symmetric in $x, y$, whereas the right side does not seem to be.
But in fact, the left side is not symmetric in $x, y$ either: the condition $1 > y > x > 0$ cannot be dropped in the statement of the theorem (otherwise the integrand would not be well-defined for any $\theta > 0$).

We call the formula in Theorem $\ref{thm:pierisample}$, as well as its general version in Theorem $\ref{thm:pieri}$, the Pieri integral formula for Jack characters.
The reason is that these formulas are, in a sense, ``dual'' to the well known Pieri rule for Jack polynomials. When $\theta = 1$, the Pieri integral formula computes certain structure coefficients of the hypergroup of conjugacy classes of the group of unitary matrices $U(N)$. There is a similar group interpretation for the special values $\theta = \frac{1}{2}, 2$. From this viewpoint, for a general $\theta > 0$, the Pieri integral formula amounts to the computation of certain structure coefficients for a ``virtual hypergroup'' of conjugacy classes, see Section $\ref{sec:olshanski}$ for more comments in this direction.

Let us say a few more words about the duality between the Pieri integral formula and the Pieri rule. The following is the well known argument-index symmetry for Macdonald polynomials $P_{\lambda}(x_1, \ldots, x_N; q, t)$:
\begin{equation*}
\frac{P_{\lambda}(q^{\mu_1}t^{N-1}, q^{\mu_2}t^{N-2}, \ldots, q^{\mu_N}; q, t)}{P_{\lambda}(t^{N-1}, t^{N-2}, \ldots, 1; q, t)} = \frac{P_{\mu}(q^{\lambda_1}t^{N-1}, q^{\lambda_2}t^{N-2}, \ldots, q^{\lambda_N}; q, t)}{P_{\mu}(t^{N-1}, t^{N-2}, \ldots, 1; q, t)},
\end{equation*}
for any partitions $\lambda, \mu$ of lengths $\leq N$. Thus identities for Macdonald polynomials \textit{come in pairs}; this observation was heavily used by Andrei Okounkov in \cite{Ok1}.

Instead of considering such symmetry for partitions of lengths $\leq N$, fix some $m < N$ and consider only partitions of lengths $\leq m$ (there are also some technical details having to do with analytic continuation), we are able to pair identities of Macdonald polynomials and of Macdonald characters of $m$ variables.
This observation allowed us to obtain integral representations for Macdonald characters of one variable and multiplicative formulas as duals to the closed form for the generating function of one-row Macdonald polynomials and to the Jacobi-Trudi formula for Macdonald polynomials respectively, see \cite{C1} for the details.

\begin{eqnarray*}
\minibox[frame]{Generating function for \\ one-row Macd. polys.} &\xleftrightarrow{\text{duality}}& \minibox[frame]{Integrals reps. for Macd. \\ characters of $1$ variable} \xrightarrow{\textrm{degenerates}} \minibox[frame]{Integrals reps. for Jack \\ characters of $1$ variable}\\
\minibox[frame]{Jacobi Trudi formula \\ for Macdonald polys.} &\xleftrightarrow{\text{duality}}& \minibox[frame]{Multiplicative formulas \\ for Macdonald polys.} \xrightarrow{\textrm{degenerates}} \minibox[frame]{Cannot degenerate easily \\ (unless $\theta = 1$ or $m = 2$)}
\end{eqnarray*}

A natural question is to ask for a dual formula to the Pieri rule for Macdonald polynomials.
The answer is given by an $(m - 1)$-dimensional $q$-integral, but nonetheless it is unsatisfactory.
Even though it does express the product $P_{\lambda}(x_1, \ldots, x_{m-1}; N, q, t)\cdot P_{\lambda}(x; N, q, t)$ as a convex combination of Macdonald characters $\{ P_{\lambda}(x_1w_1, \ldots, x_{m-1}w_{m-1}, x(w_1\cdots w_{m-1})^{-1}; N, q, t) : (w_1, \dots, w_{m-1})\in V_x\}$, it is difficult to obtain asymptotic results from such a formula because the probability measure over $V_x$ that determines the convex combination does not concentrate at a single point as $N$ tends to infinity.
It is only after we degenerate the $q$-integral to the Jack setting that we are able to see an actual integral formula, as that of Theorem $\ref{thm:pierisample}$.
Moreover, in the Jack setting, the probability measures expressing the mixtures do concentrate at a single point, as we will discuss shortly.

\begin{eqnarray*}
\minibox[frame]{Pieri rule for\\ Macdonald polys.} &\xleftrightarrow{\text{duality}}& \minibox[frame]{The sum $(\ref{proof:pieri2})$ below, which \\ is a $q$-integral in disguise} \xrightarrow{\textrm{degenerates}} \minibox[frame]{Pieri integral formula}
\end{eqnarray*}

To prove the usefulness of our formulas, which look complicated at first, we study the asymptotics of Jack characters when the sequence of signatures $\{\lambda(N)\}_{N \geq 1}$ that parametrize them have a limit in the sense of Vershik-Kerov \cite{OO1, VK}.
Next we describe our asymptotic result in more detail.

\subsection{Asymptotics of Jack characters}\label{sec:asymptoticjacks}

Given a growing sequence of signatures $\{\lambda(N)\}_{N \geq 1}$ and an arbitrary $m\in\N$, the formulas from Theorems $\ref{thm:integralsample}$, $\ref{thm:pierisample}$, and their general versions in the text, allow us to study uniform limits of Jack characters $\{J_{\lambda(N)}(z_1, \ldots, z_m; N, \theta)\}_{N \geq 1}$.
The general strategy is the following. For $m = 1$, the limits can be studied by using the integral representations of Jack characters with one variable, cf. Theorem $\ref{thm:integralsample}$. For general $m$, we apply induction and the Pieri integral formula for Jack characters, cf. Theorem $\ref{thm:pierisample}$. For instance, when $m = 2$, the idea is to prove the following approximation of probability measures, as $N$ goes to infinity:

\begin{equation*}
\begin{gathered}
\mathbf{1}_{\{w\in [y, 1]\}}\cdot\frac{\Gamma(N\theta)y^{\theta}}{\Gamma((N-1)\theta)\Gamma(\theta)((1-x)(1 - y))^{\theta}}\left(1 - \frac{xw^2}{y} \right)\left(1 - \frac{xw}{y}\right)^{\theta - 1} \left(\frac{1}{w} - 1\right)^{\theta - 1}\\
\times\left( \frac{(1 - xw)(1 - \frac{y}{w})}{(1-x)(1-y)} \right)^{\theta(N-1)-1} \frac{dw}{w^2} \approx \delta_{\{w = 1\}}, \textrm{ as } N \rightarrow\infty,
\end{gathered}
\end{equation*}
where $\mathbf{1}_{\{w\in [y, 1]\}}$ is an indicator function, the approximation symbol $\approx$ should be interpreted appropriately, and $\delta_{\{w = 1\}}$ is the delta mass at $w = 1$. Then the Pieri integral formula would show $J_{\lambda(N)}(x; N, \theta)\cdot J_{\lambda(N)}(y; N, \theta) \approx J_{\lambda(N)}(x, y; N, \theta)$ and we are able to find asymptotics of Jack characters of two variables in terms of the asymptotics of Jack characters of one variable.

In the specific limit regime that we consider, we have $\lambda(N) \in \GT_N$ for each $N \geq 1$, i.e., $\lambda(N) = (\lambda_1(N) \geq \lambda_2(N) \geq \ldots \geq \lambda_N(N)) \in \Z^N$, and they grow to infinity in the sense of Vershik-Kerov. This is explained next.

For $\lambda(N)\in\GT_N$, let $p$ be any index such that $\lambda_p(N) \geq 0 \geq \lambda_{p+1}(N)$ and let $q = N-p$. Then if we write $\lambda(N) = (\lambda^+_1(N) \geq \dots \geq \lambda_p^+(N) \geq -\lambda_q^-(N) \geq \dots \geq -\lambda_1^-(N))$, we see that $\lambda^+(N)$ and $\lambda^-(N)$ are partitions of length $\leq N$. Let $d(\lambda^{\pm}(N))$ be the lengths of the main diagonals of the Young diagrams of $\lambda^{\pm}(N)$, and let $(a_1^{\pm}(N), a_2^{\pm}(N), \ldots)$, $(b_1^{\pm}(N), b_2^{\pm}(N), \ldots)$ be their modified Frobenius coordinates:
\begin{equation*}
a_i^{\pm}(N) \myeq \lambda^{\pm}(N)_i - i + \frac{1}{2}, \hspace{.2in} b_i^{\pm}(N) \myeq \lambda^{\pm}(N)'_i - i + \frac{1}{2}, \hspace{.3in}\forall\ 1\leq i\leq d(\lambda^{\pm}(N)).
\end{equation*}
We denoted $\lambda^{\pm}(N)_i'$ the $i$-th part of the partition that is conjugate to $\lambda^{\pm}(N)$, or $\lambda^{\pm}(N)'_i = |\{j \geq 1 : \lambda^{\pm}(N)_j \geq i\}|$. Observe that $\sum_{i=1}^{d(\lambda^{\pm}(N))}{(a_i^{\pm}(N) + b_i^{\pm}(N))} = |\lambda^{\pm}(N)|$.

The sequence $\{\lambda(N)\}_{N\geq 1}$ is a \textit{Vershik-Kerov sequence} if the following limits exist and are finite
\begin{eqnarray*}
\lim_{N\rightarrow\infty}{\frac{a_i^{\pm}(N)}{N}} &=& \alpha_i^{\pm} \hspace{.1in} \forall i = 1, 2, \ldots\\
\lim_{N\rightarrow\infty}{\frac{b_i^{\pm}(N)}{N}} &=& \beta_i^{\pm} \hspace{.1in} \forall i = 1, 2, \ldots\\
\lim_{N\rightarrow\infty}{\frac{|\lambda^{\pm}(N)|}{N}} &=& \delta^{\pm}.
\end{eqnarray*}
Evidently $\alpha_1^{\pm} \geq \alpha_2^{\pm} \geq \ldots \geq 0$, $\beta_1^{\pm} \geq \beta_2^{\pm} \geq \ldots \geq 0$, $\beta_1^+ + \beta_1^- \leq 1$ and $\delta^{\pm} \geq \sum_{i=1}^{\infty}{(\alpha_i^{\pm} + \beta_i^{\pm})}$.
It is convenient to denote $\gamma^{\pm}  = \delta^{\pm} - \sum_{i=1}^{\infty}{(\alpha_i^{\pm} + \beta_i^{\pm})}$.
We say $\omega = (\alpha^+, \alpha^-, \beta^+, \beta^-, \gamma^+, \gamma^-)\in\R_{\geq 0}^{4\infty + 2}$ is the \textit{boundary point} of $\{\lambda(N)\}_{N\geq 1}$, and write $\omega = (\alpha^{\pm}, \beta^{\pm}, \gamma^{\pm})$.

The main theorem of the second part of the paper is the following.

\begin{thm}[Theorem $\ref{thm:application1}$ below]\label{thm:applicationjacks}
Assume that $\{\lambda(N)\}_{N \geq 1}$ is a Vershik-Kerov sequence with boundary point $\omega = (\alpha^{\pm}, \beta^{\pm}, \gamma^{\pm})$. Then
\begin{equation}\label{limitOkOl}
\lim_{N\rightarrow\infty}{J_{\lambda(N)}(z_1, \ldots, z_m; N, \theta)} = \prod_{i=1}^m{\Psi(z_i; \omega; \theta)}
\end{equation}
holds uniformly on $z_1, \ldots, z_m\in\{z\in\C : 1 - \delta < |z| < 1 + \delta \}$, for some small $\delta > 0$, and $\Psi(z; \omega, \theta)$ is a holomorphic function on a neighborhood of the unit circle, given by the expression
\begin{equation*}
\Psi(z; \omega, \theta) := \exp\left(\gamma^+(z - 1) + \gamma^-(z^{-1} - 1)\right)\cdot\prod_{i=1}^{\infty}{\frac{(1 + \beta_i^+(z-1))}{(1 - \alpha_i^+(z-1)/\theta)^{\theta}}\frac{(1 + \beta_i^-(z^{-1}-1))}{(1 - \alpha_i^-(z^{-1}-1)/\theta)^{\theta}}}.
\end{equation*}
\end{thm}

\subsection{Comments about the limits of Jack characters}

For the values $\theta = \frac{1}{2}, 1, 2$, there is a representation-theoretic implication of the asymptotic statement above.
This comes from the fact that Jack characters $P_{\lambda(N)}(z_1, \ldots, z_N; N, \theta)$, for $\theta = \frac{1}{2}, 1$ and $2$, represent normalized, irreducible spherical functions of the Riemannian symmetric spaces $G(N)/K(N)$, for $U(N)/O(N)$, $U(N)\times U(N)/U(N)$ and $U(2N)/Sp(N)$, respectively.
The statement in Theorem $\ref{thm:applicationjacks}$ shows how (normalized and irreducible) spherical functions of $G(N)/K(N)$ approximate the \textit{irreducible spherical functions of the infinite-dimensional pairs} $(G(\infty), K(\infty))$, $G(\infty) = \lim_{\rightarrow}{G(N)}$, $K(\infty) = \lim_{\rightarrow}{K(N)}$, when specialized at $\theta = \frac{1}{2}, 1, 2$.
The reader is referred to \cite{OO1}, and references therein, for more details.

The limit $(\ref{limitOkOl})$ above was proved by Andrei Okounkov and Grigori Olshanski in \cite{OO1}, except that they proved the convergence on the unit circle $z_1, \ldots, z_m\in\TT = \{z \in \C : |z| = 1\}$, and not on an open neighborhood of $\TT$.
Moreover, they proved certain converse statement. In theory, such converse may be proved with our formulas, but we do not do it here. In comparison to their proof, ours is more direct and the explicit formulas for Jack characters can be readily used to study the asymptotics of Jack characters in other limit regimes.

Finally, we remark that several other proofs of Theorem $\ref{thm:applicationjacks}$ exist in the literature in the special case $\theta = 1$, see \cite{BOl, GP, Pe, VK}. All of the methods of proof make heavy use of the determinantal structure of Schur polynomials when $\theta = 1$, and thus they cannot generalize for an arbitrary $\theta > 0$.

\subsection{Further research}

The formulas obtained in \cite{C1} and the present paper allow us to study various limit regimes for Macdonald characters and Jack characters, with a finite number of variables, as the rank tends to infinity. In these papers, we have proved the strength of these formulas in asymptotic representation theory, but we believe that there are applications in other areas. In fact, it would not be surprising that some of the various applications of the formulas of this paper for $\theta = 1$, e.g. \cite{G0, GP, BuG1, BuG2, Pa}, admit a one-parameter $\theta > 0$ degeneration.

As a different application of our toolbox, we studied a \textit{Jack-Gibbs} model of lozenge tilings that is defined in the spirit of \cite{BGG, GS}. The author was able to prove the weak convergence of statistics of the Jack-Gibbs lozenge tilings model near the edge of the boundary (the so-called \textit{turning point} asymptotics) to the \textit{Gaussian beta ensemble}, see e.g. Forrester \cite[Ch. 20]{Ox} and references therein. This result, and its rational limit concerning corner processes of Gaussian matrix ensembles, will appear in a forthcoming publication.

\subsection{Organization of the paper}

The first part of the paper proves several formulas for Jack characters that are suitable for asymptotics. In Section $\ref{sec:macdonaldsection}$, we briefly recall some algebraic properties of Macdonald and Jack polynomials. In Section $\ref{sec:jackintegral}$, we degenerate the integral representations for Macdonald characters of one variable and obtain analogous representations for Jack characters of one variable. Next in Section $\ref{sec:pieri}$, we state and prove the Pieri integral formula.

In the second part of the paper, we work out the asymptotics of Jack characters of a given fixed number of variables $m\in\N$, in a specific limit regime. The main theorem of the second part of the paper, as well as an overview of its proof, is given in Section $\ref{sec:appjack1}$. Later in Section $\ref{sec:casem1}$, we make use of the integral representations of Jack characters to deal with the case $m = 1$. The case of general $m\in\N$ is proved in Section $\ref{sec:proofthm}$.

\subsection*{Acknowledgments}

I wish to express my deep gratitude to Alexei Borodin and Vadim Gorin, for many helpful discussions and continued encouragement. I am grateful to Grigori Olshanski for explaining to me his viewpoint on the Pieri intergral formula, which is given in Section $\ref{sec:olshanski}$. I'm also grateful to the referee for valuable suggestions.

\section{Macdonald and Jack polynomials}\label{sec:macdonaldsection}

All the notions and results that we need about Macdonald and Jack polynomials are in \cite{M}. We briefly review some material in order to set our terminology and also to introduce lesser known concepts, such as Jack Laurent polynomials.

\subsection{Partitions, signatures and symmetric (Laurent) polynomials}

A \textit{partition} is a finite sequence of weakly decreasing nonnegative integers $\lambda = (\lambda_1 \geq \lambda_2 \geq \ldots\geq \lambda_k)\in\Z_{\geq 0}^k$. We identify partitions that differ by trailing zeroes, and we often assume a partition has as many zeroes as we want at the end, for instance, $(4, 2, 2, 0, 0)$ and $(4, 2, 2)$ are the same partition. The \textit{size} of $\lambda$ is the sum $|\lambda| \myeq \lambda_1 + \ldots + \lambda_k$, and the \textit{length} $\ell(\lambda)$ is the number of strictly positive elements of $\lambda$.

Partitions can be graphically represented by their \textit{Young diagrams}, see Figure $\ref{youngdiagram1}$. The (main) diagonal of a Young diagram is its set of squares with coordinates of the form $(k, k)$. The diagonal length of a partition $\lambda$, denoted $d(\lambda)$, is the cardinality of the main diagonal of its Young diagram.

A \textit{signature} is a sequence of weakly decreasing integers $\lambda = (\lambda_1 \geq \lambda_2 \geq \ldots\geq \lambda_k)\in\Z^k$. A \textit{positive signature} is a signature whose elements are all nonnegative. The \textit{length} of a signature, or positive signature, is the number $k$ of elements of it. Positive signatures which differ by trailing zeroes are not identified, in contrast to partitions; for example, $(4, 2, 2, 0, 0)$ and $(4, 2, 2)$ are different positive signatures of lengths $5$ and $3$, respectively. We shall denote $\GT_N$ (resp. $\GTp_N$) the set of signatures (resp. positive signatures) of length $N$. Evidenty $\GTp_N$ can be identified with the set of all partitions of length $\leq N$. Under this identification, we are allowed to talk about the Young diagram of a positive signature $\lambda\in\GTp_N$, its size and other attributes that are typically associated to partitions. Note, however, that length is defined differently for partitions and for positive signatures.

\begin{figure}
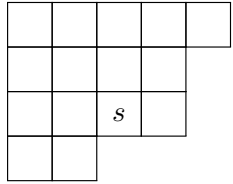

\ytableausetup{centertableaux}
\begin{ytableau}
\ &  &  & & \\
 &  &  & \\
 & & s & \\
&
\end{ytableau}
\caption{Young diagram for the partition $\lambda = (5, 4, 4, 2)$. Diagonal length of $\lambda$ is $3$. Square $s = (3, 3)$ in the main diagonal has arm length, arm colength, leg length and leg colength given by $a(s) = 1$, $a'(s) = 2$, $l(s) = 0$, $l'(s) = 2$}
\label{youngdiagram1}
\end{figure}

Fix $N\in\N$ and any field $F$. Denote by $\Lambda_F[x_1, \ldots, x_N]$ the algebra of symmetric polynomials in $x_1, \ldots, x_N$ and with coefficients in $F$. Denote also by $\Lambda_F[x_1^{\pm}, \ldots, x_N^{\pm}]$ the algebra of symmetric Laurent polynomials in $x_1, \ldots, x_N$. For any partition $\lambda$, denote the corresponding monomial symmetric polynomial by $m_{\lambda}(x_1, \ldots, x_N)$. The set $\{m_{\lambda}(x_1, \ldots, x_N)\}_{\lambda}$, where $\lambda$ varies over all partitions of length $\leq N$ is an $F$-linear basis of $\Lambda_F[x_1, \ldots, x_N]$. From the discussion in the previous paragraph we can parametrize monomial symmetric polynomials by elements of $\GTp_N$ instead of partitions of length $\leq N$.

The monomial symmetric polynomial $m_{\lambda}$ is homogeneous of degree $|\lambda|$. Moreover they satisfy the following \textit{index-stability property}
\begin{equation}\label{indexstabilitymonomials}
m_{(\lambda_1 + 1, \ldots, \lambda_N + 1)}(x_1, \ldots, x_N) = (x_1 \cdots x_N)\cdot m_{\lambda}(x_1, \ldots, x_N).
\end{equation}

We can then define monomial symmetric Laurent polynomials $m_{\lambda}(x_1, \ldots, x_N)$ for any signature $\lambda\in\GT_N$. Indeed for any $\lambda\in\GT_N$, find $M\in\N$ large enough so that $(\lambda_1 + M, \ldots, \lambda_N + M)\in\GTp_N$ and define
\begin{equation*}
m_{\lambda}(x_1, \ldots, x_N) \myeq (x_1\cdots x_N)^{-M} \cdot m_{(\lambda_1 + M, \ldots, \lambda_N + M)}(x_1, \ldots, x_N).
\end{equation*}
Because of the index-stability $(\ref{indexstabilitymonomials})$, the definition above does not depend on the choice of $M$. The set $\{m_{\lambda}(x_1, \ldots, x_N) : \lambda\in\GT_N\}$ is an $F$-linear basis of $\Lambda_F[x_1^{\pm}, \ldots, x_N^{\pm}]$.

\subsection{Macdonald polynomials}\label{macdonaldpolysection}

Let $N\in\N$, also consider variables $q, t$ and let $F = \C(q, t)$. For any partition $\lambda$ of length $\ell(\lambda) \leq N$, the \textit{Macdonald polynomial} $P_{\lambda}(x_1, \dots, x_N; q, t)$ is defined as the unique homogeneous, symmetric polynomial of degree $|\lambda|$, such that
\begin{equation*}
P_{\lambda}(x_1, \dots, x_N; q, t) = x_1^{\lambda_1}\cdots x_N^{\lambda_N} + \textrm{ terms $c_{\mu}x_1^{\mu_1}\cdots x_N^{\mu_N}$ where $\mu \leq \lambda$ in the lexicographic order},
\end{equation*}
and satisfying also certain orthogonality property that we do not need here, see \cite[Ch. VI]{M} for the details. When $t = q$, the Macdonald polynomials become the well known Schur polynomials.

If $N < \ell(\lambda)$, we set $P_{\lambda}(x_1, \ldots, x_N; q, t) \myeq 0$ for convenience.

When we are talking about Macdonald polynomials and some of their properties which hold regardless of the number $N$ of variables (as long as $N$ is large enough), we simply write $P_{\lambda}(q, t)$ instead of $P_{\lambda}(x_1, \ldots, x_N; q, t)$.

Like the monomial symmetric polynomials, we have the following index-stability property:
\begin{equation}\label{macdonaldsignatures1}
P_{(\lambda_1 + 1, \ldots, \lambda_N + 1)}(x_1, \ldots, x_N; q, t) = (x_1\cdots x_N)\cdot P_{\lambda}(x_1, \ldots, x_N; q, t).
\end{equation}

As pointed out before, the set of partitions of length $\leq N$ is in bijection with $\GTp_N$ and thus we can index the Macdonald polynomials by positive signatures rather than by partitions. We can now introduce Macdonald (Laurent) polynomials $P_{\lambda}(x_1, \ldots, x_N; q, t)$ for any $\lambda\in\GT_N$ as follows. Let $\lambda\in\GT_N$ be arbitrary. If $\lambda_N \geq 0$, then $\lambda\in\GTp_N$ and $P_{\lambda}(x_1, \ldots, x_N; q, t)$ is already defined. If $\lambda_N < 0$, then choose $M\in\N$ large enough so that $\lambda_N + M \geq 0$ and so $(\lambda_1 + M, \ldots, \lambda_N + M)\in\GTp_N$. Then define
\begin{equation}\label{macdonaldsignatures11}
P_{\lambda}(x_1, \ldots, x_N; q, t) \myeq (x_1\cdots x_N)^{-M}\cdot P_{(\lambda_1 + M, \ldots, \lambda_N + M)}(x_1, \ldots, x_N; q, t).
\end{equation}
By virtue of the index-stability property $(\ref{macdonaldsignatures1})$, the Macdonald Laurent polynomial $P_{\lambda}(x_1, \ldots, x_N; q, t)$ is well-defined and does not depend on the choice of $M$.

Recall that the \textit{arm-length, arm-colength, leg-length, leg-colength} $a(s), a'(s), l(s), l'(s)$ of the square $s = (i, j)$ of the Young diagram of $\lambda$, are given by $a(s) = \lambda_i - j$, $a'(s) = j - 1$, $l(s) = \lambda_j' - i$, $l'(s) = i - 1$.
We have denoted $\lambda_j' = \left|\{ i \geq 1 : \lambda_i \geq j \}\right|$ the length of the $j$-th part of the \textit{conjugate partition} $\lambda'$, see Figure $\ref{youngdiagram1}$.

For $\lambda\in\bigsqcup_{N\geq 1}{\GTp_N}$, define the \textit{dual Macdonald polynomials} $Q_{\lambda}(q, t)$ by
\begin{equation}\label{PtoQ}
Q_{\lambda}(q, t) \myeq b_{\lambda}(q, t)P_{\lambda}(q, t), \hspace{.2in} b_{\lambda}(q, t) \myeq \prod_{s\in\lambda}{\frac{1 - q^{a(s)}t^{l(s) + 1}}{1 - q^{a(s) + 1}t^{l(s)}}}.
\end{equation}

For $n\geq 0$, define the \textit{complete homogeneous symmetric (Macdonald) polynomials} as the one-row dual Macdonald polynomials:
\begin{equation}\label{gtoQ}
g_n(q, t) \myeq Q_{(n)}(q, t) = \frac{(q; q)_n}{(t; q)_n}P_{(n)}(q, t).
\end{equation}
Above we used the well known $q$-analysis notation, \cite[Ch. 10]{AAR}, $(a; q)_n = (1 - a)(1 - aq)\cdots (1 - aq^{n-1})$ if $n \geq 1$, and $(a; q)_0 = 1$.

Next we state several important theorems on Macdonald polynomials, which will be our main tools.

\begin{thm}\label{symmetry}[Index-argument symmetry]
Let $N\in\N$, $\lambda, \mu\in\GTp_N$, then
\begin{equation*}
\frac{P_{\lambda}(q^{\mu_1}t^{N-1}, q^{\mu_2}t^{N-2}, \ldots, q^{\mu_N}; q, t)}{P_{\lambda}(t^{N-1}, t^{N-2}, \ldots, 1; q, t)} = \frac{P_{\mu}(q^{\lambda_1}t^{N-1}, q^{\lambda_2}t^{N-2}, \ldots, q^{\lambda_N}; q, t)}{P_{\mu}(t^{N-1}, t^{N-2}, \ldots, 1; q, t)}
\end{equation*}
\end{thm}

\begin{thm}\label{evaluation}[Evaluation identity]
Let $m, N\in\N$, $N \geq m$, and let the positive signature $\lambda\in\GTp_N$ be of the form $\lambda = (\lambda_1 \geq \lambda_2 \geq \ldots \geq \lambda_m \geq 0 \geq \dots \geq 0)$, then
\begin{equation*}
\begin{gathered}
Q_{\lambda}(t^{N-1}, t^{N-2}, \ldots, 1; q, t) = t^{\lambda_2 + 2\lambda_3 + \dots + (m-1)\lambda_m}\cdot\prod_{i=1}^m{\frac{(t^{N-i+1}, q^{\lambda_i + 1}t^{m-i}; q)_{\infty}}{(q, q^{\lambda_i}t^{N-i+1}; q)_{\infty}}}\\
\times \prod_{1\leq i < j\leq m}{\frac{(q^{\lambda_i - \lambda_j + 1}t^{j-i-1}; q)_{\infty}}{(q^{\lambda_i - \lambda_j + 1}t^{j-i}; q)_{\infty}}},
\end{gathered}
\end{equation*}
where we used usual $q$-analysis notation $(x; q)_{\infty} = (1-x)(1-qx)(1-q^2x)\cdots$ and $(x_1, \dots, x_n; q)_{\infty} = \prod_{i=1}^n{(x_i; q)_{\infty}}$.
\end{thm}

For the next two statement, we need a couple of notions. Given positive signatures $\mu\in\GTp_N$, $\lambda\in\GTp_{N+1}$,  we say that they \textit{interlace} if
\begin{equation*}
\lambda_{N+1} \leq \mu_N \leq \lambda_N \leq \dots \leq \lambda_2 \leq \mu_1 \leq \lambda_1
\end{equation*}
and we denote it by $\mu\prec\lambda$. Also, for two interlacing signatures $\mu\in\GTp_N$, $\lambda\in\GTp_{N+1}$, $\mu\prec\lambda$, define the \textit{branching coefficient} $\psi_{\lambda/\mu}(q, t)$ by
\begin{equation}\label{psidef}
\psi_{\lambda/\mu}(q, t) \myeq \prod_{1\leq i \leq j\leq N}{\frac{(q^{\mu_i - \mu_j}t^{j-i+1}, q^{\lambda_i - \lambda_{j+1}}t^{j-i+1}, q^{\lambda_i - \mu_j+1}t^{j-i}, q^{\mu_i - \lambda_{j+1} + 1}t^{j-i}; q)_{\infty} }{(q^{\lambda_i - \mu_j}t^{j-i+1}, q^{\mu_i - \lambda_{j+1}}t^{j-i+1}, q^{\mu_i - \mu_j+1}t^{j-i}, q^{\lambda_i - \lambda_{j+1}+1}t^{j-i}; q)_{\infty} } }.
\end{equation}
Let us also agree that $\psi_{\lambda/\mu}(q, t) \myeq 0$ if $\mu\not\prec\lambda$, that is, if $\mu$ and $\lambda$ do not interlace.

\begin{thm}\label{branchingmacdonald}[Branching rule]
Let $N\in\N$, $\lambda\in\GTp_{N+1}$, then
\begin{equation*}
P_{\lambda}(x_1, \ldots, x_N, x_{N+1}; q, t) = \sum_{\mu \in\GTp_N}{\psi_{\lambda/\mu}(q, t)x_{N+1}^{|\lambda| - |\mu|}P_{\mu}(x_1, \ldots, x_N; q, t)}.
\end{equation*}
\end{thm}

\begin{thm}\label{thm:pierirule}[Pieri rule]
Let $N, p\in\N$, $\mu\in\GTp_N$, then
\begin{equation*}
Q_{\mu}(q, t) g_p(q, t) = \sum_{\lambda\in\GTp_{N+1}}{\psi_{\lambda/\mu}(q, t) Q_{\lambda}(q, t)},
\end{equation*}
where the sum runs over $\lambda\in\GTp_{N+1}$ such that $\mu \prec \lambda$ and $|\lambda| - |\mu| = p$.
\end{thm}

Lastly recall the definition of Macdonald characters, \cite{C1}. For any $m, N\in\N$ with $1\leq m\leq N$, and $\lambda\in\GT_N$, define
\begin{equation}\label{normalmacdonald}
P_{\lambda}(x_1, \ldots, x_m; N, q, t) \myeq \frac{P_{\lambda}(x_1, \ldots, x_m, 1, t, \ldots, t^{N-m-1}; q, t)}{P_{\lambda}(1, t, t^2, \ldots, t^{N-1}; q, t)}
\end{equation}
and call $P_{\lambda}(x_1, \ldots, x_m; N, q, t)$ the \textit{Macdonald character of rank $N$, number of variables $m$ and parametrized by $\lambda$}. Observe that, by Theorem $\ref{evaluation}$, the denominator in $(\ref{normalmacdonald})$ is nonzero.

\subsection{Jack polynomials and Jack characters}\label{sec:jacks}

Our reference for Jack polynomials is \cite[Ch. VI.10]{M}, see also \cite{St, F}.

Consider the field $F = \C(\theta)$ and let $\Lambda_F[x_1^{\pm}, \ldots, x_N^{\pm}]$ be the corresponding ring of symmetric Laurent polynomials. For any $\lambda\in\GT_N$, we can define the Jack (Laurent) polynomial $J_{\lambda}(x_1, \dots, x_N; \theta)$ as the limit
\begin{equation}\label{macdonaldtojack}
    \lim_{q\rightarrow 1}{P_{\lambda}(x_1, \ldots, x_N; q, q^{\theta})} = J_{\lambda}(x_1, \ldots, x_N; \theta).
\end{equation}
The set of Jack polynomials $\{J_{\lambda}(x_1, \ldots, x_N; \theta) : \lambda\in\GT_N\}$ is a basis of $\Lambda_F[x_1^{\pm}, \ldots, x_N^{\pm}]$.

In the limit $t = q^{\theta}$, $q\rightarrow 1$, we see that most formulas for Macdonald polynomials can be degenerated into analogues for Jack polynomials, for example the following index-stability for Jack polynomials is evident
\begin{equation}\label{indexstabilityjack}
J_{(\lambda_1 + 1, \ldots, \lambda_N + 1)}(x_1, \ldots, x_N; \theta) = (x_1\cdots x_N)\cdot J_{\lambda}(x_1, \ldots, x_N; \theta).
\end{equation}

Moreover we can obtain analogues of the evaluation identity and the branching rule for Jack polynomials. For both of the results below, recall the notation
\begin{equation*}
(x)_n \myeq x(x+1)\cdots (x+n-1) = \Gamma(x+n)/\Gamma(x).
\end{equation*}

\begin{thm}\label{evaluationjacks}[Evaluation identity of Jack polynomials]
Let $N\in\N$, $\lambda\in \GTp_N$, then
\begin{equation*}
J_{\lambda}(1^N; \theta) = \prod_{1\leq i < j\leq N}{\frac{(\theta(j-i+1))_{\lambda_i - \lambda_j}}{(\theta(j-i))_{\lambda_i - \lambda_j}}} = \prod_{1\leq i < j\leq N}{\frac{\Gamma(\lambda_i - \lambda_j + \theta(j - i + 1))\Gamma(\theta(j-i))}{\Gamma(\lambda_i - \lambda_j + \theta(j - i))\Gamma(\theta(j - i + 1))}}.
\end{equation*}
\end{thm}

\begin{thm}\label{branchingjacks}[Branching rule for Jack polynomials]
Let $N\in\N$, $\lambda\in \GTp_{N+1}$, then
\begin{equation*}
J_{\lambda}(x_1, \ldots, x_N, x_{N+1}; \theta) = \sum_{\mu : \mu\prec\lambda}{\psi_{\lambda/\mu}(\theta)x_{N+1}^{|\lambda| - |\mu|}J_{\mu}(x_1, \ldots, x_N; \theta)},
\end{equation*}
where the branching coefficients are
\begin{equation*}
\psi_{\lambda/\mu}(\theta) \myeq \prod_{1\leq i\leq j\leq N}{\frac{(\mu_i - \mu_j + \theta(j-i)+\theta)_{\mu_j - \lambda_{j+1}}}{(\mu_i - \mu_j + \theta(j-i)+1)_{\mu_j - \lambda_{j+1}}}\frac{(\lambda_i - \mu_j + \theta(j-i) + 1)_{\mu_j - \lambda_{j+1}}}{(\lambda_i - \mu_j + \theta(j-i) + \theta)_{\mu_j - \lambda_{j+1}}}}.
\end{equation*}
\end{thm}

It is better for our purposes to specialize the formal parameter $\theta$ to a positive real number, and effectively assume hereinafter $\theta > 0$.
Since Theorem $\ref{branchingjacks}$ shows that all coefficients of the polynomial are well-defined for any $\theta > 0$, this specialization poses no problem; all prior statements continue to hold.
The set of Jack polynomials $\{J_{\lambda}(x_1, \ldots, x_N; \theta) : \lambda\in\GT_N\}$ is now a basis of $\Lambda_{\C}[x_1^{\pm}, \ldots, x_N^{\pm}]$.

Let us introduce the main object of our study.

\begin{df}\label{jackunitaries}
For any integers $1\leq m\leq N$, and $\lambda\in\GT_N$, define
\begin{equation}\label{normaljack}
J_{\lambda}(x_1, \ldots, x_m; N, \theta) \myeq \frac{J_{\lambda}(x_1, \ldots, x_m, 1^{N-m}; \theta)}{J_{\lambda}(1^N; \theta)}
\end{equation}
and call $J_{\lambda}(x_1, \ldots, x_m; N, \theta)$ the \textit{Jack unitary character of rank $N$, number of variables $m$ and parametrized by $\lambda$}. For simplicity of terminology, we call $J_{\lambda}(x_1, \ldots, x_m; N, \theta)$ a Jack character rather than a Jack unitary character. As a result of the evaluation identity for Jack polynomials, Theorem $\ref{evaluationjacks}$, the denominator in $(\ref{normaljack})$ is nonzero.
\end{df}

\begin{rem}
A different notion of Jack character was introduced previously in the literature by Lassalle, \cite{L}. It has been studied heavily in recent years by Maciej Doł\k{e}ga, Valentin F\'{e}ray and Piotr \'{S}niady, see e.g. \cite{DF13, DF16, DFS, Sn}. These Jack characters are related to symmetric groups $S(N)$ like the Jack characters in Definition $\ref{jackunitaries}$ are related to unitary groups $U(N)$.
\end{rem}

\section{Integral representations for Jack characters of one variable}\label{sec:jackintegral}

In this section, and the rest of the paper, we often use $i$ for the index of a product or sum, so we denote the imaginary unit by the bold letter $\ii = \sqrt{-1}$.

\subsection{Integral formulas for Macdonald characters}

The following three theorems were proved in \cite{C1}. The latter two are analytic continuations of the first one which holds when the variable $\theta$ is a positive integer. For all the formulas below assume that $q\in (0, 1)$ and $t = q^{\theta}$ for some $\theta > 0$.

\begin{thm}[\cite{C1}, Thm. 3.1]\label{macdonaldthm1}
Let $\theta\in\N$, $t = q^{\theta}$, $N\in\N$, $\lambda\in\GT_N$ and $x\in\C \setminus \{0, q, q^2, \ldots, q^{\theta N - 1}\}$. Then
\begin{equation}\label{macdonaldthm1eqn}
\frac{P_{\lambda}(x, t, t^2, \ldots, t^{N-1}; q, t)}{P_{\lambda}(1, t, t^2, \ldots, t^{N-1}; q, t)} = \ln(1/q)\prod_{i=1}^{\theta N - 1}{\frac{1 - q^i}{x - q^i}}\frac{1}{2\pi\ii}\oint_{\CC_0}{\frac{x^z}{\prod_{i=1}^N\prod_{j=0}^{\theta -1}{(1 - q^{z - (\lambda_i + \theta(N-i) + j)})}}dz},
\end{equation}
where $\CC_0$ is a closed, positively oriented contour enclosing the poles $\{\lambda_i + \theta(N-i) + j :\ i = 1, 2, \ldots, N, \ j = 0, 1, \ldots, \theta - 1\}$ and no other poles of the integrand. For instance, the rectangular contour with vertices $-M-r\ii, \ -M+r\ii, \ M+r\ii$ and $M-r\ii$, for any $-\frac{2\pi}{\ln{q}} > r > 0$ and any $M > \max\{0, -\lambda_N, \lambda_1 + \theta N - 1\}$, is a suitable contour.
\end{thm}

\begin{thm}[\cite{C1}, Thm. 3.2]\label{macdonaldthm2}
Let $\theta > 0$, $t = q^{\theta}$, $N\in\N$, $\lambda\in\GT_N$ and $x\in\C \setminus \{0\}$, $|x| < 1$. The integral below converges absolutely and the equality holds
\begin{equation}\label{macdonaldthm2eqn}
\frac{P_{\lambda}(xt^{N-1}, t^{N-2}, \ldots, t, 1; q, t)}{P_{\lambda}(t^{N-1}, t^{N-2}, \ldots, t, 1; q, t)} = \frac{\ln q}{1 - q}\frac{(xt^N; q)_{\infty}}{(xq; q)_{\infty}}\frac{\Gamma_q(\theta N)}{2\pi\ii}\int_{\CC^+}{x^z\prod_{i=1}^N{\frac{\Gamma_q(\lambda_i + \theta(N-i)-z)}{\Gamma_q(\lambda_i + \theta(N-i+1)-z)}}dz}.
\end{equation}
Contour $\CC^+$ is the positively oriented contour consisting of the segment $[M+r\ii, \ M-r\ii]$ and the horizontal lines $[M+r\ii, +\infty+r\ii)$, $[M-r\ii, +\infty-r\ii)$, for any $-\frac{\pi}{2\ln{q}} > r > 0$ and $\lambda_N > M$, see Figure $\ref{fig:C}$.
Observe that $\CC^+$ encloses all real poles (which accumulate at $+\infty$) of the integrand, and no other poles.
\end{thm}

\begin{thm}[\cite{C1}, Thm. 3.3]\label{macdonaldthm25}
Let $\theta > 0$, $t = q^{\theta}$, $N\in\N$, $\lambda\in\GT_N$ and $x\in\C \setminus \{0\}$, $|x| > 1$. The integral below converges absolutely and the equality holds
\begin{equation}\label{macdonaldthm25eqn}
\frac{P_{\lambda}(x, t, t^2, \ldots, t^{N-1}; q, t)}{P_{\lambda}(1, t, t^2, \ldots, t^{N-1}; q, t)} = \frac{\ln q}{q - 1}\frac{(x^{-1}t^N; q)_{\infty}}{(x^{-1}q; q)_{\infty}}\frac{\Gamma_q(\theta N)}{2\pi\ii}\int_{\CC^-}{x^z\prod_{i=1}^N{\frac{\Gamma_q(z -( \lambda_i - \theta i + \theta))}{\Gamma_q(z - (\lambda_i - \theta i))}}dz}.
\end{equation}
Contour $\CC^-$ is the positively oriented contour consisting of the segment $[M-r\ii, \ M+r\ii]$ and the horizontal lines $[M-r\ii, -\infty-r\ii)$, $[M+r\ii, -\infty+r\ii)$, for any $-\frac{\pi}{2\ln{q}} > r > 0$ and any $M > \lambda_1$, see Figure $\ref{fig:Cminus}$.
Observe that $\CC^-$ encloses all real poles (which accumulate at $-\infty$) of the integrand, and no other poles.
\end{thm}

\begin{center}
\begin{figure}
\begin{center}
\begin{tikzpicture}[decoration={markings,
mark=at position 1.5cm with {\arrow[line width=1pt]{>}},
mark=at position 3.2cm with {\arrow[line width=1pt]{>}},
mark=at position 4.7cm with {\arrow[line width=1pt]{>}},
mark=at position 6.5cm with {\arrow[line width=1pt]{>}},
mark=at position 8.3cm with {\arrow[line width=1pt]{>}},
mark=at position 10cm with {\arrow[line width=1pt]{>}},
mark=at position 11.8cm with {\arrow[line width=1pt]{>}}
}
]

\draw[help lines,->] (-4,0) -- (4,0) coordinate (xaxis);
\draw[help lines,->] (0,-1) -- (0,1) coordinate (yaxis);

\path[draw,line width=1pt,postaction=decorate] (4,0.5) -- (-2,0.5) -- (-2,-0.5) -- (4,-0.5);

\node[below] at (xaxis) {$\Re z$};
\node[left] at (yaxis) {$\Im z$};
\node[below left] {};
\end{tikzpicture}
\end{center}
\caption{Contour $\CC^+$}
\label{fig:C}
\end{figure}
\end{center}

\begin{figure}
\begin{center}
\begin{tikzpicture}[decoration={markings,
mark=at position 1.5cm with {\arrow[line width=1pt]{>}},
mark=at position 3.2cm with {\arrow[line width=1pt]{>}},
mark=at position 5cm with {\arrow[line width=1pt]{>}},
mark=at position 6.5cm with {\arrow[line width=1pt]{>}},
mark=at position 8cm with {\arrow[line width=1pt]{>}},
mark=at position 9.8cm with {\arrow[line width=1pt]{>}},
mark=at position 11.7cm with {\arrow[line width=1pt]{>}}
}
]
\draw[help lines,->] (-4,0) -- (4,0) coordinate (xaxis);
\draw[help lines,->] (0,-1) -- (0,1) coordinate (yaxis);

\path[draw,line width=1pt,postaction=decorate] (-4,-0.5) -- (2,-0.5) -- (2,0.5) -- (-4, 0.5);

\node[below] at (xaxis) {$\Re z$};
\node[left] at (yaxis) {$\Im z$};
\node[below left] {};
\end{tikzpicture}
\end{center}
\caption{Contour $\CC^-$}
\label{fig:Cminus}
\end{figure}

\begin{rem}
In the formulas above, $x^z = \exp(z\ln{x})$. If $x \notin (-\infty, 0)$, we can use the principal branch of the logarithm to define $\ln{x}$, and if $x \in (-\infty, 0)$, then we can define the logarithm in the complex plane cut along $(-\mathbf{i}\infty, 0]$ such that $\Im \ln{a} = 0$ for all $a \in (0, \infty)$.
\end{rem}

\subsection{Integral formulas for Jack characters}

We degenerate the integral formulas for Macdonald characters in the regime $t = q^{\theta}$, $q\rightarrow 1$, and obtain analogous formulas for Jack characters.

\begin{thm}\label{jackthm1}
Let $\theta\in\N$, $N\in\N$, $\lambda\in\GT_N$ and $x\in\C\setminus\{0, 1\}$. Then
\begin{equation}\label{jackthm1eqn}
J_{\lambda}(x; N, \theta) = \frac{(\theta N - 1)!}{(x-1)^{\theta N - 1}}\frac{1}{2\pi\ii}\oint_{\CC_0}{ \frac{x^z}{\prod_{i=1}^N\prod_{j=0}^{\theta -1}{({z - (\lambda_i + \theta(N-i) + j)})}} dz},
\end{equation}
where the contour $\CC_0$ is a closed positively oriented contour enclosing all poles $\{\lambda_i + \theta(N-i) + j : \ i = 1, 2, \dots, N, \ j = 0, 1, \dots, \theta - 1\}$.
\end{thm}
\begin{proof}
Take any $x\in\C \setminus [0, 1]$. The identity above is essentially Theorem $\ref{macdonaldthm1}$ after the limit transition $t = q^{\theta}$, $q\rightarrow 1$. This is clear thanks to the following limits.
\begin{equation}\label{eqn:limitsjacks1}
\begin{gathered}
\lim_{\substack{t = q^{\theta} \\ q\rightarrow 1}}{\frac{P_{\lambda}(x, 1, t, \ldots, t^{N-2}; q, t)}{P_{\lambda}(1, t, \ldots, t^{N-1}; q, t)}} = \frac{J_{\lambda}(x, 1^{N-1}; \theta)}{J_{\lambda}(1^N; \theta)}\\
\lim_{q\rightarrow 1}{\prod_{i=1}^{\theta N - 1}{\frac{1}{x - q^i}}} = \frac{1}{(x - 1)^{\theta N - 1}}\\
\lim_{q\rightarrow 1}{ (1-q)^{-\theta N} \ln(1/q)\prod_{i=1}^{\theta N - 1}{(1 - q^i)} } = \lim_{q\rightarrow 1}{\frac{\ln{q}}{q - 1} \prod_{i=1}^{\theta N - 1}{[i]_q} } = (\theta N - 1)!\\
\lim_{q\rightarrow 1}{{  \frac{(1 - q)^{\theta N}}{\prod_{i=1}^N \prod_{j=0}^{\theta - 1}{(1 - q^{z - (\lambda_i + \theta(N - i) + j)}})  } } } = {\frac{1}{\prod_{i=1}^N \prod_{j=0}^{\theta - 1}{(z - (\lambda_i + \theta(N - i) + j))}}} ,
\end{gathered}
\end{equation}
the first limit is the convergence of Macdonald polynomials to Jack polynomials, see $(\ref{macdonaldtojack})$, the second limit is obvious, the third and fourth limit follow from basic $q$-analysis, see e.g. \cite[Ch. 10]{AAR}.

We still need to justify the exchange of limit and integral signs.
Let us look further at the fourth line in $(\ref{eqn:limitsjacks1})$ above. First observe that each factor $(z - (\lambda_i + \theta(N - i) + j))$, $1\leq i\leq N$, $0\leq j < \theta$, does not vanish for $z\in\CC_0$. Moreover each fraction
\begin{equation*}
\frac{1 - q^{z - (\lambda_i + \theta (N - i) + j)}}{1 - q}, \ 1\leq i\leq N, \ 0\leq j < \theta,
\end{equation*}
is analytic in the variables $(q, z)$ for $q$ in a neighborhood of $1$ and $z$ in a neighborhood of the contour $\CC_0$.
Moreover none of the fractions above vanishes if these neighborhoods are thin enough.
It follows that the fourth limit in $(\ref{eqn:limitsjacks1})$ is uniform for $q$ in a neighrborhood of $1$ and $z$ in a neighborhood of $\CC_0$.

We can now justify the exchange of the limit and integral by the dominated convergence theorem and the uniform convergence of the integrand, that we just proved. Then the identity $(\ref{jackthm1eqn})$ is proved for any $x\in\C\setminus [0, 1]$. To extend the identity for $x\in\C\setminus\{0, 1\}$, notice that both sides are rational functions on $x$ (the right hand side may be expressed as a sum of $\theta N$ residues at simple poles). The left side may have a pole at $x = 0$ while the right hand side has a pole at $x = 1$ and also it is not well defined at $x = 0$ because of $x^z = \exp(z\ln{x})$ in the integrand. Thus the identity holds for all $x\in\C\setminus\{0, 1\}$.
\end{proof}

\begin{rem}
Jack polynomials with parameter $\theta = 1$ become Schur polynomials $s_{\lambda}(x_1, \dots, x_N)$. In this special case, Theorem $\ref{jackthm1}$ recovers \cite[Thm. 3.8]{GP}.
\end{rem}

\begin{thm}\label{jackthm2}
Let $\theta > 0$, $N\in\N$, $\lambda\in\GT_N$. Also let $x\in\C \setminus \{0\}$, $|x| < 1$. Then the integral below converges absolutely and the identity holds
\begin{equation}\label{jackthm2eqn}
J_{\lambda}(x; N, \theta) = -\frac{\Gamma(\theta N)}{(1 - x)^{\theta N - 1}}\frac{1}{2\pi\ii}\int_{\CC^+}{x^z\prod_{i=1}^N{\frac{\Gamma(\lambda_i + \theta(N-i)-z)}{\Gamma(\lambda_i + \theta(N-i+1)-z)}}dz},
\end{equation}
where the positively oriented contour $\CC^+$ consists of the segment $[M+r\ii, \ M-r\ii]$ and horizontal lines $[M+r\ii, +\infty+r\ii)$, $[M-r\ii, +\infty-r\ii)$, for any $\lambda_N > M$ and any $r > 0$, see Figure $\ref{fig:C}$.
\end{thm}

\begin{proof}
Just like Theorem $\ref{jackthm1}$ follows from Theorem $\ref{macdonaldthm1}$, Theorem $\ref{jackthm2}$ will follow from Theorem $\ref{macdonaldthm2}$ after the limit transition $t = q^{\theta}$, $q\rightarrow 1^-$. The additional limits that we need are
\begin{gather*}
\lim_{q\rightarrow 1^-}{\frac{\ln{q}}{1-q}} = -1,\
\lim_{q\rightarrow 1^-}{\frac{(xq^{\theta N}; q)_{\infty}}{(xq; q)_{\infty}}} = (1 - x)^{1 - \theta N},\\
\lim_{q\rightarrow 1^-}{\Gamma_q(\theta N)} = \Gamma(\theta N), \hspace{.2in} \lim_{q\rightarrow 1^-}{\prod_{i=1}^N{ \frac{\Gamma_q(\lambda_i + \theta(N - i) - z)}{\Gamma_q(\lambda_i + \theta(N - i + 1) - z)} }} = \prod_{i=1}^N{ \frac{\Gamma(\lambda_i + \theta(N - i) - z)}{\Gamma(\lambda_i + \theta(N - i + 1) - z)} },
\end{gather*}
see \cite[Ch. 10]{AAR}. Let $\CC^+$ be a contour as the one in the statement of Theorem $\ref{jackthm2}$, for some $\lambda_N > M$ and some $r > 0$. Observe that for $q\in [\delta, 1)$ and $\delta \in (0, 1)$ very close to $1$, we have $0 < r < -\pi/(2\ln{q})$ and therefore the contour $\CC^+$ is also appropriate for the integral representation of Theorem $\ref{macdonaldthm2}$. Because of the limits as $q\rightarrow 1^-$ stated above, we simply need to justify the exchange between the limit and integral signs, and also to show the absolute convergence of the integral in $(\ref{jackthm2eqn})$.

We claim that
\begin{equation}\label{needbound1}
\sup_{q\in[\delta, 1]}\sup_{z\in\CC^+}{\left|\prod_{i=1}^N{\frac{\Gamma_q(\lambda_i + \theta(N - i) - z)}{\Gamma_q(\lambda_i + \theta(N - i + 1) - z)}}\right|} = O(1),
\end{equation}
for some $0 < \delta < 1$ very close to $1$.
Let us show that the theorem follows from the claim above, and we prove the claim later.
Since $|x| < 1$, the integral $\int_{\CC^+}{x^z dz}$ is absolutely convergent.
As a consequence of the claim $(\ref{needbound1})$, the absolute value of the integrand in $(\ref{jackthm2eqn})$ is upper bounded by a constant times $|x^z|$, so the absolute convergence of the integral in the theorem holds.
The claim above, and the dominated convergence theorem, also shows that the limit $q\rightarrow 1^-$ and integral signs can be interchanged and thus the theorem is proved.

We only need to prove the claim $(\ref{needbound1})$.
Let us first prove such a bound for $z\in\CC^+$ with $\Re z$ large enough.
It suffices to bound one of the $N$ ratios of $q$-Gamma functions, because they are all similar.
Since $\Im (\lambda_i + \theta(N - i)) = 0$ and $|\Im z| = r \in (0, -\pi/(2\ln{q}))$ is constant for $z\in\CC^+$ with $\Re z$ large, it suffices to show
\begin{equation}\label{needbound}
\sup_{q\in[\delta, 1]}\sup_{\substack{|\Im z| = r \\ \Re z \geq K}}{\left|\frac{\Gamma_q(-z)}{\Gamma_q(\theta - z)}\right|} = O(1),
\end{equation}
for some $K$ large enough, and $0 < \delta < 1$ close enough to $1$.

Let $a = a(q), b = b(q)\in (0, 1)$, $a^2 + b^2 = 1$, be such that $q^{\pm r\ii} = a \pm b\ii$, that is, $a = \cos(r\ln{q})$ and $b = \sin(r\ln{q})$. Thus for any $z\in\C$ with $|\Im z| = r$, $\Re z \geq K$, we find
\begin{equation*}
|1 - q^{-z}|^2 = \left|1 - q^{-\Re z}(a \pm b\ii)\right|^2 = (1 - q^{-\Re z}a)^2 + (q^{-\Re z} b)^2 = 1 - 2q^{-\Re z}a + q^{-2\Re z}
\end{equation*}
and similarly
\begin{equation*}
|1 - q^{\theta - z}|^2 = 1 - 2q^{-\Re z + \theta}a + q^{-2\Re z + 2\theta}.
\end{equation*}
From above, the inequality $|1 - q^{\theta - z}| \leq |1 - q^{-z}|$ is equivalent to $q^{-\Re z}(1 + q^{\theta}) \geq 2a = 2\cos(r\ln{q})$. And in fact, $q^{-K}(1 + q^{\theta}) \geq 2\cos(r\ln{q})$ holds for any $K > \theta$ because the function $f(q) = q^{-K}(1 + q^{\theta}) - 2\cos(r\ln{q})$ can be shown to be nonincreasing in an interval of the form $[\delta, 1]$, for $\delta < 1$ very close to $1$, and $f(1) = 0$. The inequality we just showed is equivalent to $|[\theta - z]_q| \leq |[-z]_q|$, which implies that for any $K > \theta$:
\begin{equation*}
\left| \frac{\Gamma_q(-z)}{\Gamma_q(\theta-z)} \right| = \left| \frac{\Gamma_q(-z+1)}{\Gamma_q(\theta-z+1)} \right|\times\left| \frac{[\theta-z]_q}{[-z]_q} \right| \leq \left| \frac{\Gamma_q(-z+1)}{\Gamma_q(\theta-z+1)} \right| \ \forall q \in [\delta, 1], \ |\Im z| = r, \ \Re z \geq K.
\end{equation*}

By applying the inequality above repeatedly, we deduce
\begin{equation}\label{eqn:ineq2}
\sup_{q\in[\delta, 1]}\sup_{\substack{|\Im z| = r \\ \Re z \geq K}}{\left|\frac{\Gamma_q(-z)}{\Gamma_q(\theta - z)}\right|}
\leq \sup_{q\in[\delta, 1]}\sup_{\substack{|\Im z| = r \\ K + 1 \geq \Re z \geq K}}{\left|\frac{\Gamma_q(-z)}{\Gamma_q(\theta - z)}\right|}.
\end{equation}

Next recall that $\lim_{q\rightarrow 1^-}{\Gamma_q(x)} = \Gamma(x)$ holds uniformly on compact subsets of $\C\setminus\{\dots, -2, -1, 0\}$. Also the set $\{z\in\C : |\Im z| = r, \ K + 1 \geq \Re z \geq K\}$ is a compact subset of $\C\setminus\{\dots, -2, -1, 0\}$. Thus, by making the value of $\delta$ closer to $1$ if necessary, we deduce
\begin{equation}\label{eqn:ineq3}
\sup_{q\in[\delta, 1]}\sup_{\substack{|\Im z| = r \\ K + 1 \geq \Re z \geq K}}{\left|\frac{\Gamma_q(-z)}{\Gamma_q(\theta - z)}\right|} \leq \sup_{\substack{|\Im z| = r \\ K + 1 \geq \Re z \geq K}}{\left|\frac{\Gamma(-z)}{\Gamma(\theta - z)}\right|} + 1 = O(1).
\end{equation}
From $(\ref{eqn:ineq2})$ and $(\ref{eqn:ineq3})$, the bound $(\ref{needbound})$ is proved.

We are left to look at values $z\in\CC^+$ with $\Re z \leq L$, $L$ is a fixed real number, and $q\in [\delta, 1]$, for some $\delta \in (0, 1)$ close enough to $1$. We again use the fact that $\lim_{q\rightarrow 1^-}{\Gamma(x)} = \Gamma(x)$ holds uniformly on compact subsets of $\C\setminus\{0, -1, \dots\}$. Also note that, by definition of the contour $\CC^+$, the integrand of $(\ref{jackthm2eqn})$ is holomorphic in a neighborhood of the compact subset $\{z \in \CC^+ : \Re z \leq L\}$. From the same reasoning as above, there exists $\delta\in (0, 1)$ close enough to $1$ such that
\begin{equation}\label{needbound2}
\sup_{\substack{q\in [\delta, 1] \\ z\in\CC^+, \Re z\leq L}}\left| \prod_{i=1}^N{\frac{\Gamma_q(\lambda_i + \theta(N - i) - z)}{\Gamma_q(\lambda_i + \theta(N - i + 1) - z)}} \right|
\leq \sup_{\substack{z\in\CC^+ \\ \Re z \leq L}}\left| \prod_{i=1}^N{\frac{\Gamma(\lambda_i + \theta (N - i) - z)}{\Gamma(\lambda_i + \theta(N - i + 1) - z)}} \right| + 1 = O(1).
\end{equation}
From the estimates $(\ref{needbound})$ and $(\ref{needbound2})$, the claim $(\ref{needbound1})$ follows. Thus the theorem is proved.
\end{proof}

The next theorem follows from Theorem $\ref{macdonaldthm25}$ in the same way as Theorem $\ref{jackthm2}$ above followed from Theorem $\ref{macdonaldthm2}$, so we omit a proof.

\begin{thm}\label{jackthm25}
Let $\theta > 0$, $N\in\N$, $\lambda\in\GT_N$ and $x\in\C$, $|x| > 1$. Then the integral below converges absolutely and the identity holds
\begin{equation}\label{jackthm25eqn}
J_{\lambda}(x; N, \theta) = \frac{\Gamma(\theta N)}{(1 - x^{-1})^{\theta N - 1}}\frac{1}{2\pi\ii}\int_{\CC^-}{x^z\prod_{i=1}^N{\frac{\Gamma(z - (\lambda_i -\theta i + \theta))}{\Gamma(z - (\lambda_i - \theta i) )}}dz},
\end{equation}
where the positively oriented contour $\CC^-$ consists of the segment $[M + r\ii, \ M - r\ii]$ and the horizontal lines $[M + r\ii, -\infty+r\ii)$, $[M - r\ii, -\infty - r\ii)$, for any $M > \lambda_1$ and any $r > 0$, see Figure $\ref{fig:Cminus}$.
\end{thm}
\begin{rem}
The assumptions $|x| < 1$ and $|x| > 1$, for Theorems $\ref{jackthm2}$ and $\ref{jackthm25}$ respectively, guarantee the absolute convergence of the integrals.
It might be possible to obtain a similar integral representation in the case $|x| = 1$, but this is not so clear; since we do not need such a formula for the application we have in mind, we ignore this issue.
\end{rem}

\begin{rem}\label{jackrewritten}
After the change of variables $z\mapsto z-\theta N + 1$, we can rewrite $(\ref{jackthm25eqn})$ as
\begin{equation*}
J_{\lambda}(x; N, \theta) = \frac{\Gamma(\theta N)}{(x - 1)^{\theta N - 1}}\frac{1}{2\pi\ii}\int_{\CC^-}{x^z\prod_{i=1}^N{\frac{\Gamma(z + 1 - (\lambda_i + \theta (N - i + 1)))}{\Gamma(z + 1 - (\lambda_i + \theta (N - i)) )}}dz},
\end{equation*}
where $\CC^-$ is an appropriate contour, that looks as in Figure $\ref{fig:Cminus}$. It is in this form that the formula is applied in Section $\ref{sec:deformationapp1}$ below.
\end{rem}

\section{Pieri integral formula for Jack characters}\label{sec:pieri}

\subsection{The Pieri integral formula}

In this section, we state and prove formula $(\ref{eqn:pieri})$ which we call the \textit{Pieri integral formula}, and in the section below we describe an interpretation of this formula as the computation of structure constants of a virtual hypergroup of conjugacy classes.

\begin{thm}\label{thm:pieri}
Let $\nu\in\GT_N$ be any signature, let $m, N\in\N$ be such that $1\leq m\leq N-1$. Consider also any real numbers
\begin{equation*}
\begin{gathered}
1 > x > x_m > \ldots > x_2 > x_1 > 0,\\
x > \max_{i = 1, 2, \ldots, m-1}{(x_i / x_{i+1})}.
\end{gathered}
\end{equation*}
Then
\begin{equation}\label{eqn:pieri}
\begin{gathered}
J_{\nu}(x_1, \ldots, x_m; N, \theta)J_{\nu}(x; N, \theta) = \frac{\Gamma(N\theta)x^{m\theta}}{\Gamma((N-m)\theta)\Gamma(\theta)^m(1-x)^{m\theta}\prod_{i=1}^m{(1-x_i)^{\theta}}}\\
\times\int\dots\int { G(x_1, \ldots, x_m, x; w_1, \ldots, w_m; \theta)\left( F(x_1, \ldots, x_m, x; w_1, \ldots, w_m) \right)^{\theta(N-m)-1}\prod_{i=1}^m{(1 - w_i)^{\theta-1}} }\\
{ J_{\nu}\left(x_1w_1, \ldots, x_mw_m, \frac{x}{w_1\cdots w_m}; N, \theta\right) \frac{dw_1}{w_1}\dots\frac{dw_m}{w_m} },
\end{gathered}
\end{equation}
where the functions $G$ and $F$ are defined as follows (we denote $w_{m+1} := x/(w_1\cdots w_m)$ and $x_{m+1} := 1$ to simplify notation):
\begin{equation}\label{defn:FGfns}
\begin{gathered}
G(x_1, \ldots, x_m, x; w_1, \ldots, w_m; \theta) \myeq (w_1^m w_2^{m-1}\cdots w_m)^{-\theta} \prod_{1\leq i < j\leq m}{\left(1 - x_ix_j^{-1}w_i\right)^{\theta-1}}\\
\times\prod_{1\leq i<j\leq m+1}{(1 - x_iw_ix_j^{-1}w_j^{-1}) (1 - x_ix_j^{-1}w_j^{-1})^{\theta-1} }\prod_{1\leq i < j\leq m}{(1 - x_ix_j^{-1})^{1 - 2\theta}},\\
F(x_1, \ldots, x_m, x; w_1, \ldots, w_m) \myeq (1 - x)^{-1} \prod_{j=1}^m{(1 - x_j)^{-1}} \prod_{i=1}^{m+1}{(1 - x_iw_i)},
\end{gathered}
\end{equation}
and the domain of integration $\U_x$ in the integral of $(\ref{eqn:pieri})$ depends only on $x$ and is the compact subset of $\R^m$ defined by the inequalities
\begin{equation}\label{defn:Ax}
\begin{gathered}
1 \geq w_1, \ldots, w_m \geq 0,\\
w_1\cdots w_m \geq x.
\end{gathered}
\end{equation}
\end{thm}

\begin{proof}
Let $\mu\in\GTp_m$, $k\in\Zp$. Begin with the Pieri rule for Macdonald polynomials, Theorem $\ref{thm:pierirule}$, which we write as
\begin{equation}\label{proof:pieri1}
\frac{Q_{\mu}(x_1, \ldots, x_N; q, t)}{Q_{\mu}(t^{N-1}, \ldots, t, 1; q, t)}\frac{g_k(x_1, \ldots, x_N; q, t)}{g_k(t^{N-1}, \ldots, t, 1; q, t)} = \sum_{\lambda}{C(\lambda, \mu; k; q, t) \frac{Q_{\lambda}(x_1, \ldots, x_N; q, t)}{Q_{\lambda}(t^{N-1}, \ldots, t, 1; q, t)}},
\end{equation}
the sum being over $\lambda\in\GTp_{m+1}$, $\mu\prec\lambda$, $|\lambda| - |\mu| = k$, and the coefficients are given by
\begin{equation*}
C(\lambda, \mu; k; q, t) \myeq \frac{\psi_{\lambda/\mu}(q, t)\cdot Q_{\lambda}(t^{N-1}, \ldots, t, 1; q, t)}{Q_{\mu}(t^{N-1}, \ldots, t, 1; q, t)\cdot g_k(t^{N-1}, \ldots, t, 1; q, t)}.
\end{equation*}
At this moment, let us make an important observation. If we assume $\mu_i - \mu_{i+1} > k$ for all $i = 1, 2, \ldots, m-1$ and $\mu_m > k$, then given any $k_1, \ldots, k_m\in\Zp$ with $k_1 + \ldots + k_m \leq k$, we can obtain $\lambda\in\GTp_{m+1}$ such that $\mu\prec\lambda$, $|\lambda| - |\mu| = k$, by setting
\begin{equation*}
\begin{gathered}
\lambda_i = \mu_i + k_i \ \forall i = 1, 2, \ldots, m, \hspace{.2in} \lambda_{m+1} = k - \sum_{i=1}^m{k_i}.
\end{gathered}
\end{equation*}
Conversely any such $\lambda\in\GTp_{m+1}$ arises from an $m$-tuple $(k_1, \ldots, k_m)\in\Zp^m$, $k_1 + \dots + k_m \leq k$, in such a way.

Let $\nu = (\nu_1 \geq \nu_2 \geq \ldots \geq \nu_N)\in\GTp_N$ be any positive signature of length $N$. In equation $(\ref{proof:pieri1})$, let $x_i = q^{\nu_i}t^{N-i} \ \forall i = 1, 2, \ldots, N$. Using the definition of the dual Macdonald polynomials $Q_{\mu} = b_{\mu}P_{\mu}$ and the index-argument symmetry of Theorem $\ref{symmetry}$, we deduce that if $\mu_i - \mu_{i+1} > k$, $\mu_m > k$, then equation $(\ref{proof:pieri1})$ leads to
\begin{equation}\label{proof:pieri2}
\begin{gathered}
P_{\nu}(q^{\mu_1}t^{N-1}, q^{\mu_2}t^{N-2}, \ldots, q^{\mu_m}t^{N - m}; N, q, t)\cdot  P_{\nu}(q^k t^{N-1}; N, q, t)\\
= \sum_{\substack{k_1, \ldots, k_m\in\Zp \\ k_1 + \ldots + k_m \leq k}}{C(\lambda, \mu; k; q, t) P_{\nu}(q^{\lambda_1}t^{N-1}, \ldots, q^{\lambda_{m+1}}t^{N-m-1}; N, q, t)},
\end{gathered}
\end{equation}
where in the sum above, we have denoted $\lambda\in\GTp_{N+1}$ the positive signature given by
\begin{equation}\label{proof:pieri3}
\begin{gathered}
\lambda_i := \mu_i + k_i \ \forall i = 1, 2, \ldots, m,\\
\lambda_{m+1} := k - \sum_{i=1}^m{k_i}.
\end{gathered}
\end{equation}

From Theorem $\ref{evaluation}$ and the explicit product formula for $\psi_{\lambda/\mu}(q, t)$, we obtain
\begin{equation}\label{eqn:Ccoeff}
\begin{gathered}
C(\lambda, \mu; k; q, t) = t^{\sum_{i=2}^m{(i-1)(\lambda_i - \mu_i)} + m\lambda_{m+1}}\cdot \frac{(1 - q)^m \cdot\Gamma_q(\theta N)}{\Gamma_q(\theta)^m\cdot\Gamma_q(\theta(N-m))}
\frac{(q^{k+\theta N}; q)_{\infty}}{(q^{k+1}; q)_{\infty}}\\
\times \prod_{1\leq i\leq j\leq m}{\frac{(q^{\mu_ i - \lambda_{j+1} + 1 + \theta(j-i)}; q)_{\infty} (q^{\lambda_i - \mu_j + 1 + \theta(j-i)}; q)_{\infty}}{(q^{\mu_i - \lambda_{j+1} + \theta(j - i + 1)}; q)_{\infty}(q^{\lambda_i - \mu_j + \theta(j - i + 1)}; q)_{\infty}}}\\
\times \prod_{1\leq i < j\leq m+1}{\frac{(q^{\lambda_i - \lambda_j + \theta(j - i)}; q)_{\infty}}{(q^{\lambda_i - \lambda_j + 1 + \theta(j-i)}; q)_{\infty}}}
\times \prod_{1\leq i < j\leq m}{\frac{(q^{\mu_i - \mu_j + \theta(j - i + 1)}; q)_{\infty}}{(q^{\mu_i - \mu_j + 1 + \theta(j - i - 1)}; q)_{\infty}}}\\
\times\prod_{i=1}^{m+1}{\frac{(q^{\lambda_i + 1 + \theta(m-i+1)}; q)_{\infty}}{(q^{\lambda_i + \theta(N-i+1)}; q)_{\infty}}}
\times \prod_{i=1}^m{ \frac{(q^{\mu_i + \theta(N - i + 1)}; q)_{\infty}}{(q^{\mu_i + 1 + \theta(m-i)}; q)_{\infty}} }.
\end{gathered}
\end{equation}

Next we want to make a limit transition to the formula $(\ref{proof:pieri2})$. For that, consider positive real numbers $y_1, \ldots, y_m, y$ such that
\begin{equation}\label{conditionsy}
\begin{gathered}
y_1 > y_2 > \ldots > y_m > y > 0,\\
y_i - y_{i+1} > y \ \forall i = 1, 2, \ldots, m-1.
\end{gathered}
\end{equation}

The limit transition we consider is
\begin{equation}\label{eqn:limittransition}
\begin{gathered}
q = \exp(-\epsilon), \ t = q^{\theta} = \exp(-\epsilon\theta),\\
k = \lfloor \epsilon^{-1}y \rfloor, \ \mu_i = \lfloor\epsilon^{-1}y_i \rfloor, \ k_i = \lfloor\epsilon^{-1}z_i \rfloor, \ \forall i = 1, 2, \ldots, m,\\
\lambda_i = \lfloor \epsilon^{-1}y_i \rfloor + \lfloor \epsilon^{-1}z_i \rfloor, \ \forall i = 1, 2, \dots, m, \hspace{.1in} \lambda_{m+1} = \lfloor \epsilon^{-1}y \rfloor - \sum_{i=1}^m{\lfloor \epsilon^{-1}z_i \rfloor},\\
\epsilon \rightarrow 0^+.
\end{gathered}
\end{equation}
We observe that conditions $(\ref{conditionsy})$ clearly imply $\mu_i - \mu_{i+1} \geq k$, $\mu_m \geq k$, and in fact the inequalities are strict if $\epsilon > 0$ is small enough. This implies that formula $(\ref{proof:pieri2})$ holds for the parametrization above and small enough $\epsilon > 0$.
The goal is to obtain the limit of the formula as $\epsilon \rightarrow 0^+$.
In particular, for any $z_1, z_2, \ldots, z_m \geq 0$ such that $z_1 + z_2 + \dots + z_m \leq y$, we show that the general term inside the sum of $(\ref{proof:pieri2})$, with variables as in $(\ref{eqn:limittransition})$, is approximately equal to $\epsilon^m$ times the integrand in $(\ref{eqn:pieri})$, and moreover this estimate is uniform.

Clearly, by the limit $(\ref{macdonaldtojack})$ from Macdonald to Jack polynomials, the left side of $(\ref{proof:pieri2})$ is
\begin{equation*}
\approx J_{\nu}(e^{-y_1}, \ldots, e^{-y_N}; N, \theta)\cdot J_{\nu}(e^{-y}; N, \theta),
\end{equation*}
under the limit transition $(\ref{eqn:limittransition})$, where $A(\epsilon) \approx B(\epsilon)$ means that both $A(\epsilon), B(\epsilon)$ are nonzero for small $\epsilon > 0$ and $\lim_{\epsilon\rightarrow 0}{A(\epsilon)/B(\epsilon)} = 1$. Similarly, for the Macdonald term inside the sum of the right side, we have
\begin{equation*}
\begin{gathered}
P_{\nu}(q^{\lambda_1}t^{N-1}, \ldots, q^{\lambda_m}t^{N-m}, q^{\lambda_{m+1}}t^{N-m-1}; N, q, t)\\
\approx J_{\nu}(e^{-(y_1+z_1)}, \ldots, e^{-(y_m+z_m)}, e^{-(y - z_1-\ldots-z_m)}; N, \theta).
\end{gathered}
\end{equation*}

Estimating $C(\lambda, \mu; k; q, t)$ in the regime $(\ref{eqn:limittransition})$ above is more tedious. We need to apply the following estimates, which are proved in \cite[Ch. 10]{AAR}:
\begin{equation*}
\lim_{q\rightarrow 1^-}{\frac{(xq^a; q)_{\infty}}{(xq^b; q)_{\infty}}} = (1 - x)^{b-a}, \ \lim_{q\rightarrow 1^-}{\Gamma_q(x)} = \Gamma(x).
\end{equation*}
Both estimates above are uniform on compact subsets of some domain for $x$: the left limit is uniform on compact subsets of the unit disk and the right limit is uniform on compact subsets of $\C \setminus\{0, -1, -2, \ldots\}$.
Then we can estimate $(\ref{eqn:Ccoeff})$, line-by-line, as follows:
\begin{equation*}
\begin{gathered}
C(\lambda, \mu; k; q, t) \approx \epsilon^m \cdot \exp\left(-\theta \sum_{i=2}^{m+1}{(i-1)z_i}\right)
\frac{\Gamma(\theta N)}{\Gamma(\theta)^m\Gamma(\theta(N-m))}
\left( 1 - e^{-y} \right)^{1 - \theta N}\\
\times \prod_{1\leq i\leq j\leq m}{\left( 1 - e^{-y_i + y_{j+1} + z_{j+1}} \right)^{\theta - 1}\left( 1 - e^{-y_i - z_i + y_j} \right)^{\theta - 1}}\\
\times \prod_{1\leq i < j\leq m+1}{\left( 1 - e^{-y_i - z_i + y_j + z_j} \right)}
\times \prod_{1\leq i < j\leq m}{\left(1 - e^{-y_i + y_j}\right)^{1 - 2\theta}}\\
\times\prod_{i=1}^{m+1}{\left( 1 - e^{-y_i - z_i} \right)^{\theta(N - m) - 1}}
\times \prod_{i=1}^m{ \left( 1 - e^{-y_i} \right)^{1 + \theta(m - N - 1)} },
\end{gathered}
\end{equation*}
where we set above
\begin{equation*}
z_{m+1} := y - \sum_{i=1}^m{z_i}, \hspace{.15in} y_{m+1} := 0.
\end{equation*}
Notice the factor $\epsilon^m$ in front of the estimate for $C(\lambda, \mu; k; q, t)$. This factor allows us to replace the $m$-dimensional sum into an $m$-dimensional integral as soon as we have uniform estimates, that is, as soon as we show that all estimates $\approx$ above are uniform on the compact subset $\{(z_1, \ldots, z_m)\in\R_{\geq 0}^m : z_1 + z_2 + \dots + z_m \leq y\}$. This is fairly obvious. As a result of the limit transition we obtain
\begin{equation}\label{eqn:almostfinal}
\begin{gathered}
J_{\nu}(e^{-y_1}, \ldots, e^{-y_m}; N, \theta)J_{\nu}(e^{-y}; N, \theta) = \frac{\Gamma(\theta N)e^{-\theta my}}{\Gamma(\theta(N-m))\Gamma(\theta)^m(1 - e^{-y})^{m\theta}\prod_{i=1}^m{(1 - e^{-y_i})^{\theta}}}\\
\times\int\dots\int { G(e^{-y_1}, \ldots, e^{-y_m}, e^{-y}; e^{-z_1}, \ldots, e^{-z_m}; \theta) F(e^{-y_1}, \ldots, e^{-y_m}, e^{-y}; e^{-z_1}, \ldots, e^{-z_m})^{\theta(N-m)-1} }\\
{ \prod_{i=1}^m{(1 - e^{-z_i})^{\theta-1}} \cdot J_{\nu}\left(e^{-y_1 - z_1}, \ldots, e^{-y_m - z_m}, e^{-y+z_1+\dots+z_m}; N, \theta\right) dz_1\dots dz_m },
\end{gathered}
\end{equation}
where the domain of integration $V_y$ is given by the inequalities
\begin{equation}\label{defn:Axbefore}
\begin{gathered}
z_1, \ldots, z_m \geq 0,\\
z_1 + \dots + z_m \leq y,
\end{gathered}
\end{equation}
 the functions $F, G$ are defined as in $(\ref{defn:FGfns})$, but with $x_{m+1} := 1$ and $w_{m+1} := e^{-y + z_1 + \dots + z_m}$ replaced in their definition (as it would be natural). Finally we make the change of variables
\begin{equation*}
\begin{gathered}
x_i = e^{-y_i} \ \forall i = 1, 2, \ldots, m, \hspace{.2in} x = e^{-y}\\
w_i = e^{-z_i} \ \forall i = 1, 2, \ldots, m.
\end{gathered}
\end{equation*}
Then the conditions $(\ref{conditionsy})$ on the variables $y, y_1, \ldots, y_m$ turn into the conditions on $x, x_1, \ldots, x_m$ in the theorem statement.
The domain of integration $V_y$ given by $(\ref{defn:Axbefore})$, for the variables $z_1, \ldots, z_m$, turns into the domain of integration $\U_x$ for the variables $w_1, \ldots, w_m$.
The theorem, for $\nu\in\GTp_N$, is then a consequence of $(\ref{eqn:almostfinal})$.

Finally, the Pieri integral formula for $\nu\in\GT_N$ follows from the Pieri integral formula for $\nu + M := (\nu_1 + M \geq \dots \geq \nu_N + M) \in \GTp_N$ with large enough $M\in\N$ and the index-stability property, $(\ref{indexstabilityjack})$, of Jack polynomials.
\end{proof}

\subsection{Virtual hypergroup of conjugacy classes}\label{sec:olshanski}

This section is largely based on Grigori Olshanski's comments.

Our aim is to describe the relation between the Pieri integral formula for $\theta = 1$ and certain structure constants of the hypergroup of two-sided cosets of $U(N) \backslash GL(N, \C) / U(N)$. We also say a few words about the cases $\theta = \frac{1}{2}, 2$ and argue that the Pieri integral formula, for general $\theta > 0$, amounts to the computation of certain structure constants of a \textit{virtual} hypergroup of two-sided cosets.

Recall that, in a hypergroup $X$, the product of two elements $x, y\in X$ is not a single element of $X$, but rather a probability measure $m_{x, y}$ on $X$. One can define the hypergroup algebra of $X$; then $m_{x, y}$ plays the role of its structure constants. The situation is especially clear when $X$ is a finite hypergroup.

The set of two-sided cosets of $U(N) \backslash GL(N, \C) / U(N)$ can be given the structure of a hypergroup.
In fact, if $x, y$ are two-sided cosets, then the measure $m_{x, y}$ on $U(N) \backslash GL(N, \C) / U(N)$ is constructed as follows. Take any representatives $g\in x$, $h\in y$; then $m_{x, y}$ is the pushforward of the normalized Haar measure under the map
\begin{equation*}
\begin{gathered}
U(N) \rightarrow U(N) \backslash GL(N, \C) / U(N)\\
u \mapsto \textrm{the class of }guh.
\end{gathered}
\end{equation*}
The result does not depend on the choice of $g, h$.

Recall now the functional equation for the Gelfand pair $(GL(N, \C), U(N))$. If $\chi$ is a zonal spherical function of $(GL(N, \C), U(N))$, then
\begin{equation}\label{eqn:functional}
\int_{u\in U(N)}{\chi(guh)du} = \chi(g)\chi(h) \hspace{.2in} \forall g, h\in U(N),
\end{equation}
where $du$ denotes the normalized Haar measure on $U(N)$. From the point of view of hypergroups, the functional equation above shows that the zonal spherical functions of $(GL(N, \C), U(N))$ are precisely the characters of the hypergroup of two-sided cosets.

Let us go further. The set of two-sided cosets $U(N) \backslash GL(N, \C) / U(N)$ can be identified with $\R_+^N / S_N$. the quotient is by the action of the symmetric group on the $N$-dimensional product.
From the previous discussion, the set $\R_+^N / S_N$ acquired a hypergroup structure; we denote the hypergroup by $H(N)$. Given $x, y\in H(N)$, denote the corresponding measure by $m_{x, y}^{(N)}$.

It is known that, as $N$ tends to infinity, the hypergroup $H(N)$ turns into a \textit{semigroup} in the following sense.
For $1\leq k\leq N$, let $H_k(N)$ be the subset of classes represented by some element $(x_1, \ldots, x_N)\in\R_+^N$ for which at most $k$ values $x_i$'s are different from $1$.
We may identify $H_k(N)$ with $H(k)$, for any $N \geq k$. It is readily seen that if $x\in H_k(N)$, $y\in H_l(N)$, and $k + l \leq N$, then $m_{x, y}^{(N)}$ is concentrated on $H_{k+l}^{(N)}$; thus in this situation we can regard $m_{x, y}^{(N)}$ as a measure on $H(k+l)$.
It turns out that, as $N$ tends to infinity and $x, y$ remain fixed, the measure $m_{x, y}^{(N)}$ converges to the delta measure at certain element $z(x, y)$ of $H(k+l)$.
The element $z(x, y)$ is represented by a tuple $(z_1, \ldots, z_N)\in\R_+^N$, whose set of entries that are not $1$ is the union of the corresponding sets for $x$ and $y$.
This fact is a manifestation of the well-known \textit{concentration of measure} phenomenon.

A consequence of the discussion in the previous paragraph is that the irreducible spherical functions of the infinite-dimensional Gelfand pair $(GL(\infty, \C), U(\infty))$ are multiplicative.
The compact Lie group analogue is perhaps more well-known: the characters of the infinite-dimensional unitary group $U(\infty)$ are multiplicative.
Such (asymptotic) multiplicativity can be proved in various ways; for more details, see \cite[Sec. 23]{Ol} and \cite{Ol1}. Section $\ref{sec:proofthm}$ in the special case $\theta = 1$ is, essentially, a new proof of the multiplicativity.

The Pieri integral formula in Theorem $\ref{thm:pieri}$, for $\theta = 1$, is equivalent to the explicit computation of the probability measures $m_{x, y}^{(N)}$, when $l = 1$ and so $y$ is the conjugacy class consisting of matrices with a single real eigenvalue distinct from $1$.
There are some technical conditions on $x, y$ in the statement of our theorem, but those are in place so that the measure $m_{x, y}^{(N)}$ has the simplest possible form; a similar analysis as that given above should suffice to compute the probability measure $m_{x, y}^{(N)}$, for an arbitrary $x\in H_k(N)$ and $y\in H_1(N)$.

For the special values $\theta = \frac{1}{2}, 2$, the Pieri integral formula also admits a group-theoretic interpretation. In fact, for $\theta = \frac{1}{2}, 2$, the corresponding hypergroup of two-sided cosets of $K(N)\backslash G(N)/K(N)$, for $(G(N), K(N))$ equal to $(GL(N, \R), O(N))$ and $(GL(N, \mathbb{H}), U(N, \mathbb{H}))$, respectively.
The concentration of measures $m_{x, y}^{(N)}$, as $N$ tends to infinity, continues to hold in both cases.

From this viewpoint, the Pieri integral formula is equivalent to the computation of certain structure constants for a \textit{virtual hypergroup} with a general Jack parameter $\theta > 0$. Such virtual hypergroup is an actual hypergroup of two-sided cosets as explained above, for the values $\theta = \frac{1}{2}, 1, 2$, but for a general $\theta > 0$, there is no group-theoretic interpretation. It is noteworthy that a whole direction of research in this direction has emerged in the past decades, e.g. \cite{He, HO, Op1, Op2, Op, Ne}. Much like these works, our approach can likely be extended for more general (complex) values of $\theta$, as well as for root systems of general type (other than type A).

\section{Asymptotics of Jack charaters: statement of the theorem}\label{sec:appjack1}

As an application of the formulas obtained in previous sections, we shall prove a refinement of the asymptotics of Jack characters, proved initially by Okounkov-Olshanksi, \cite{OO1}. In this section, we simply state the theorem and outline the approach we follow for its proof. Our plan is realized in the final two sections of this article.

\subsection{Statement of the Theorem}

Let $\{\lambda(N)\}_{N\geq 1}$ be sequences of signatures such that $\lambda(N)\in\GT_N$ for all $N\geq 1$. Let $\lambda^+(N) = (\lambda^+_1(N) \geq \dots \geq \lambda^+_N(N)),\ \lambda^-(N) = (\lambda^-_1(N) \geq \dots \geq \lambda^-_N(N)) \in \GTp_N$ be defined by
\begin{gather*}
\lambda^+_i(N) \myeq \lambda_i(N), \textrm{ if }\lambda_i(N) \geq 0, \hspace{.2in} \lambda^+_i(N) \myeq 0, \textrm{ if }\lambda_i(N) < 0,\\
\lambda^-_i(N) \myeq -\lambda_{N+1-i}(N), \textrm{ if }\lambda_{N+1-i}(N) \leq 0, \hspace{.3in} \lambda^-_i(N) \myeq 0, \textrm{ if }\lambda_{N+1-i}(N) > 0.
\end{gather*}
Let $d(\lambda^{\pm}(N))$ be the lengths of the main diagonals of the partitions corresponding to $\lambda^{\pm}(N)$, and let $(a_1^{\pm}(N), a_2^{\pm}(N), \ldots)$, $(b_1^{\pm}(N), b_2^{\pm}(N), \ldots)$ be their \textit{modified Frobenius coordinates}, defined as
\begin{equation*}
a_i^{\pm}(N) \myeq \lambda^{\pm}(N)_i - i + \frac{1}{2}, \hspace{.2in} b_i^{\pm}(N) \myeq \lambda^{\pm}(N)'_i - i + \frac{1}{2}, \hspace{.3in} \forall \ 1\leq i\leq d(\lambda^{\pm}(N)).
\end{equation*}
We denoted by $\lambda^{\pm}(N)_i'$ the $i$-th part of the conjugate partition of $\lambda^{\pm}(N)$; equivalently this means $\lambda^{\pm}(N)_i' = \left|\{ j : \lambda^{\pm}(N)_j \geq i \}\right|$. The quantity $a_i^{\pm}(N)$ (resp. $b_i^{\pm}(N)$) is roughly the arm-length (resp. leg-length) of the $i$-th square in the main diagonal of the Young diagram of $\lambda^{\pm}(N)$. Moreover, $\sum_{i=1}^{d(\lambda^{\pm}(N))}{(a_i^{\pm}(N) + b_i^{\pm}(N))} = |\lambda^{\pm}(N)|$.

\begin{df}\label{defVK}
The sequence $\{\lambda(N)\}_{N\geq 1}$ is a \textit{Vershik-Kerov sequence}\footnote{VK sequence, for short} if the following limits exist and are finite
\begin{eqnarray*}
\lim_{N\rightarrow\infty}{\frac{a_i^{\pm}(N)}{N}} &=& \alpha_i^{\pm} \hspace{.1in} \forall i = 1, 2, \ldots\\
\lim_{N\rightarrow\infty}{\frac{b_i^{\pm}(N)}{N}} &=& \beta_i^{\pm} \hspace{.1in} \forall i = 1, 2, \ldots\\
\lim_{N\rightarrow\infty}{\frac{|\lambda^{\pm}(N)|}{N}} &=& \delta^{\pm}.
\end{eqnarray*}
Evidently we have $\alpha_1^{\pm} \geq \alpha_2^{\pm} \geq \ldots \geq 0$, $\beta_1^{\pm} \geq \beta_2^{\pm} \geq \ldots \geq 0$. From Fatou's lemma, we also have $\beta_1^+ + \beta_1^- \leq 1$ and $\delta^{\pm} \geq \sum_{i=1}^{\infty}{(\alpha_i^{\pm} + \beta_i^{\pm})}$. It is convenient to set $\gamma^{\pm}  = \delta^{\pm} - \sum_{i=1}^{\infty}{(\alpha_i^{\pm} + \beta_i^{\pm})}$, so $\gamma^{\pm} \geq 0$. We say $\omega = (\alpha^+, \alpha^-, \beta^+, \beta^-, \gamma^+, \gamma^-)\in\R_{\geq 0}^{4\infty + 2}$ is the \textit{boundary point} of $\{\lambda(N)\}_{N\geq 1}$, and write $\omega = (\alpha^{\pm}, \beta^{\pm}, \gamma^{\pm})$.
\end{df}
The main theorem of the second part of the paper is the following.
\begin{thm}\label{thm:application1}
Let $m \in\N$ and let $\{\lambda(N)\}_{N\geq 1}$ be a VK sequence of signatures with boundary point $\omega = (\alpha^{\pm}, \beta^{\pm}, \gamma^{\pm})$. There exists $\delta > 0$ such that the limit
\begin{equation}\label{thm:variate}
\lim_{N\rightarrow\infty}{J_{\lambda(N)}(z_1, \ldots, z_m; N, \theta)} = \prod_{i = 1}^m{\Psi(z_i; \omega, \theta)},
\end{equation}
holds uniformly on $\TT_{\delta}^m = \underbrace{\TT_{\delta} \times \cdots \times \TT_{\delta}}_{\textrm{$m$ times}}$, where $\TT_{\delta}\myeq\{z\in\C : 1 - \delta < |z| < 1 + \delta\}$ and
\begin{equation}\label{Psidef}
\Psi(z; \omega, \theta) \myeq \exp(\gamma^+(z - 1) + \gamma^-(z^{-1} - 1))\prod_{i=1}^{\infty}{\frac{(1 + \beta_i^+(z-1))}{(1 - \alpha_i^+(z-1)/\theta)^{\theta}}\frac{(1 + \beta_i^-(z^{-1}-1))}{(1 - \alpha_i^-(z^{-1}-1)/\theta)^{\theta}}}.
\end{equation}
In the denominators, the terms $x^{\theta} = \exp(\theta \ln{x})$ have their principal values. The function $\Psi(z; \omega, \theta)$ is holomorphic on a neighborhood of the $m$-dimensional torus $\TT^m$, and we are assuming that $\delta > 0$ is small enough so that $\Psi(z; \omega, \theta)$ is holomorphic on $\TT_{\delta}^m$.
\end{thm}

\begin{rem}
Via the contour integral representations in Section $\ref{sec:jackintegral}$, we could actually prove that the limit $(\ref{thm:variate})$ for $m = 1$ holds uniformly on $\{ z\in\C : 1 - \delta_1 < |z| < 1 + \delta_2 \}$, for any
\[
0 < \delta_1 < \theta/(\alpha_1^- + \theta), \ 0 < \delta_2 < \theta/\alpha_1^+.
\]
This is best possible because $\Psi(z; \omega, \theta)$ has singularities at $z = 1 - \theta/(\alpha_1^- + \theta)$ and $z = 1 + \theta/\alpha_1^+$.
We believe that the limit $(\ref{thm:variate})$ holds uniformly on $\{ z\in\C : 1 - \delta_1 < |z| < 1 + \delta_2 \}^m$, for any $\delta_1, \delta_2$ as above.
However, our analysis in Section $\ref{sec:proofthm}$ does not allow us to prove this statement.
We can at least obtain explicit values of $\delta_m$ for which we can prove the limit $(\ref{thm:variate})$ for a given $m$, by looking at the inequalities at the beginning of Subsection $\ref{sec:convergencepieri}$.
\end{rem}

\begin{rem}
The product above is absolutely convergent, for $z$ in a neighborhood of the unit circle $\TT$, because $\sum_i{(\alpha_i^+ + \alpha_i^- + \beta_i^+ + \beta_i^-)} < \infty$.
\end{rem}

\begin{rem}
Theorem $\ref{thm:application1}$ above is slightly stronger than the ``if'' statement of the main theorem in \cite{OO1}. In fact, the uniform convergence on $\TT_{\delta}^m$, for any $\delta > 0$, implies uniform convergence on the torus $\TT^m$, which is the statement proved in \cite{OO1}.
\end{rem}

\subsection{Outline of the proof of Theorem $\ref{thm:application1}$}\label{outlineapp1}

The proof of Theorem $\ref{thm:application1}$ is by induction on the number of variables $m$.
For the base case $m = 1$, we make use of the integral representations for Jack characters of one variable. In fact, Theorem $\ref{thm:application1}$ is established in Section $\ref{sec:proofthm}$ for $m = 1$ (and any $\theta > 0$) by means of Theorems $\ref{jackthm2}$, $\ref{jackthm25}$, and the saddle-point method, \cite{C}. 
For the inductive step, we shall need the Pieri integral formula for Jack polynomials, Theorem $\ref{thm:pieri}$, as well as some easy estimates; this is worked out in Section $\ref{sec:proofthm}$.

In the remaining of this section, we only outline the steps of the proof of Theorem $\ref{thm:application1}$ for $m = 1$, and carry out the rest of the proof in Section $\ref{sec:casem1}$.

\begin{lem}\label{claimconvergence}
Assume there exists $\frac{1}{2} > \epsilon > 0$ such that the convergence
\begin{equation}\label{toproveapp1}
\lim_{N\rightarrow\infty}{J_{\lambda(N)}(z; N, \theta)} = \Psi(z; \omega, \theta)
\end{equation}
holds uniformly on compact subsets of $\{z\in\C : 1 < |z| < 1+ \epsilon \}\setminus (-\infty, 0]$ and $\{z\in\C : 1 - \epsilon < |z| < 1\}\setminus [0, \infty)$. Then $(\ref{toproveapp1})$ holds uniformly on $\TT_{\epsilon/2} := \{z\in\C : 1 - \frac{\epsilon}{2} < |z| < 1 + \frac{\epsilon}{2}\}$.
\end{lem}
\begin{proof}
From the branching rule for Jack polynomials, Theorem $\ref{branchingjacks}$, it is clear that all branching coefficients $\psi_{\mu/\nu}(\theta)$ are nonnegative for $\theta > 0$; therefore for $z\in\overline{\TT_{2\epsilon/3}} := \{z\in\C : 1 - \frac{2\epsilon}{3} \leq z \leq 1 + \frac{2\epsilon}{3}\}$, we have
\begin{equation*}
|J_{\lambda(N)}(z; N, \theta)| \leq J_{\lambda(N)}(|z|; N, \theta) < J_{\lambda(N)}\left(\left(1 - \frac{2\epsilon}{3}\right)^{-1}; N, \theta\right) + J_{\lambda(N)}\left(1 + \frac{2\epsilon}{3}; N, \theta\right).
\end{equation*}
Note that $\frac{1}{2} > \epsilon > 0$ implies $1 < (1 - 2\epsilon/3)^{-1} < 1 + \epsilon$, whereas $1 < 1 + 2\epsilon/3 < 1 + \epsilon$ is obvious. Then, by assumption, we have
\begin{equation*}
\begin{gathered}
\lim_{N\rightarrow\infty}{J_{\lambda(N)}\left( (1 - 2\epsilon/3)^{-1}; N, \theta \right)} = \Psi\left((1 - 2\epsilon/3)^{-1}; \omega, \theta\right),\\
\lim_{N\rightarrow\infty}{J_{\lambda(N)}\left(1 + 2\epsilon/3; N, \theta \right)} = \Psi\left(1 + 2\epsilon/3; \omega, \theta\right).
\end{gathered}
\end{equation*}
In particular, the sequences $\left\{\left|J_{\lambda(N)}\left(1 + \frac{2\epsilon}{3}; N, \theta\right)\right| : N \geq 1 \right\}$, $\left\{\left|J_{\lambda(N)}\left( (1 - \frac{2\epsilon}{3})^{-1}; N, \theta\right)\right| : N\geq 1 \right\}$ are bounded, and therefore the sequence $\{J_{\lambda(N)}(z; N, \theta)\}_{N\geq 1}$ of holomorphic functions are uniformly bounded on the closure of $\TT_{2\epsilon/3}$.

Next, for any $x\in\TT_{\epsilon/2}$, an application of Cauchy's integral formula yields the equality
\begin{equation}\label{cauchyintegralclaim}
J_{\lambda(N)}(x; N, \theta) = \frac{1}{2\pi\ii}\oint_{C^+_{1+2\epsilon/3} \cup C^-_{1-2\epsilon/3}}{\frac{J_{\lambda(N)}(z; N, \theta)}{z - x}dz},
\end{equation}
where $C_{1+2\epsilon/3}^+$ is the positively oriented contour centered at $0$ with radius $1 + \frac{2\epsilon}{3}$ and $C^-_{1 - 2\epsilon/3}$ is the negatively oriented contour centered at $0$ with radius $1 - \frac{2\epsilon}{3}$. Due to $\lim_{N\rightarrow\infty}{J_{\lambda(N)}(z; N, \theta)} = \Psi(z; \omega, \theta)$ on $\left(C^+_{1 + 2\epsilon/3} \cup C^-_{1 - 2\epsilon/3}\right) \setminus \{- 1 - \frac{2\epsilon}{3}, 1 - \frac{2\epsilon}{3}\}$ (an implication of the assumption in the theorem), combined with the uniform boundedness of $\{J_{\lambda(N)}(z; N, \theta)\}_{N\geq 1}$ on the closure of $\TT_{2\epsilon/3}$, we obtain the limit
\begin{equation}\label{cauchyintegralclaim2}
\lim_{N\rightarrow\infty}{\frac{1}{2\pi\ii}\oint_{C^+_{1+2\epsilon/3} \cup C^-_{1-2\epsilon/3}}{\frac{J_{\lambda(N)}(z; N, \theta)}{z - x}dz}} = \frac{1}{2\pi\ii}\oint_{C^+_{1+2\epsilon/3} \cup C^-_{1-2\epsilon/3}}{\frac{\Psi(z; \omega, \theta)}{z - x}dz} = \Psi(x; \omega, \theta),
\end{equation}
uniformly on $x\in\TT_{\epsilon/2}$. The claim follows from $(\ref{cauchyintegralclaim})$ and $(\ref{cauchyintegralclaim2})$.
\end{proof}

By virtue of Lemma $\ref{claimconvergence}$, we need only to prove the desired limit of Theorem $\ref{thm:application1}$ on compact subsets of $\{z\in\C : 1 < |z| < 1+ \epsilon \} \setminus (-\infty, 0]$ and $\{z\in\C : 1 - \epsilon < |z| < 1\} \setminus [0, \infty)$, for some small $1/2 > \epsilon > 0$ (since then, Theorem $\ref{thm:application1}$ follows if we set $\delta = \epsilon/2$). We use Theorems $\ref{jackthm2eqn}$ and $\ref{jackthm25eqn}$, respectively. In Section $\ref{sec:casem1}$ we only work out the uniform convergence on compact subsets of $\{z\in\C : 1 < |z| < 1+ \epsilon \}\setminus (-\infty, 0]$, for some small $\epsilon > 0$, since the uniform convergence on compact subsets of $\{z\in\C : 1 - \epsilon < |z| < 1\} \setminus [0, \infty)$ can be proved similarly.

\section{One variable case: proof of Theorem $\ref{thm:application1}$ for $m = 1$}\label{sec:casem1}

In this section, we prove the assumption of Lemma $\ref{claimconvergence}$. As we mentioned above we concentrate in proving the convergence $(\ref{toproveapp1})$ on the domain $\{z\in\C : 1 < |z| < 1 + \epsilon\}\setminus (-\infty, 0]$, for some small $\epsilon > 0$, by means of the integral representation in Theorem $\ref{jackthm25}$. The convergence in the domain $\{z\in\C : 1 - \epsilon < |z| < 1\}\setminus [0, \infty)$ can be proved similarly, by making use of the integral representation in Theorem $\ref{jackthm2}$ instead.

After the change of variables $x = e^y$, the statement we wish to show becomes

\begin{prop}\label{toproveclaim2}
Assume that $\{\lambda(N)\}$ is a VK sequence with boundary point $\omega = (\alpha^{\pm}, \beta^{\pm}, \gamma^{\pm})$. Then
\begin{equation*}
\lim_{N\rightarrow\infty}{J_{\lambda(N)}(e^y; N, \theta)} = \Psi(e^y; \omega, \theta),
\end{equation*}
where $\Psi$ is defined in $(\ref{Psidef})$, and the limit holds uniformly for $y$ on compact subsets of the region
\begin{equation*}
V_{\delta} \myeq \{v\in\C : 0 < \Re v < \delta, \ -\pi < \Im v < \pi\},
\end{equation*}
for some small $\delta > 0$.
\end{prop}

In the rest of this section, we prove Proposition $\ref{toproveclaim2}$. We can use the integral representation of Theorem $\ref{jackthm25}$, in the form of Remark $\ref{jackrewritten}$:
\begin{equation}\label{rewrittenformula1}
J_{\lambda(N)}(e^y; N, \theta) = \frac{\Gamma(\theta N)N^{1-\theta N}}{(e^y - 1)^{\theta N - 1}2\pi\ii}\int_{\CC^-}{\exp(Nw(z))\psi(z; \lambda(N), \theta) dz}, \hspace{.1in}\Re y > 0,
\end{equation}
where
\begin{equation}\label{defsasymptotics}
\begin{gathered}
w(z) \myeq yz - H(z; \theta), \hspace{.2in} H(z; \theta) \myeq z\ln{z} - (z - \theta)\ln{(z - \theta)} - \theta,\\
\psi(z; \lambda(N), \theta) \myeq \exp(NH(z; \theta))N^{\theta N}\prod_{i=1}^N{\frac{\Gamma(Nz + 1 - (\lambda_i(N) + \theta(N - i + 1)))}{\Gamma(Nz + 1 - (\lambda_i(N) + \theta(N-i)))}},
\end{gathered}
\end{equation}
and the logarithms appearing in the definition of $H(z; \theta)$ are defined on the complex plane cut along $(-\infty, 0]$.
The contour $\CC^-$ in this case can be the positively oriented contour consisting of the segment $[M - r\ii, M + r\ii]$ and the rays $[M - r\ii, -\infty - r\ii)$, $[M + r\ii, -\infty + r\ii)$, for any $r > 0$ and any $M > \max\{\theta,\ \theta + \frac{\lambda_1(N) - 1}{N}\}$.
Note that the contour $\CC^-$ is in the domain of definition for the functions $\ln{z}$, $\ln{(z - \theta)}$, and moreover $\CC^-$ encloses all singularities of the integrand. Since $\{\lambda(N)\}_{N \geq 1}$ is a VK sequence, then $\lim_{N\rightarrow\infty}{\lambda_1(N)/N}$ exists and so $\{\lambda_1(N)/N\}_{N \geq 1}$ is a bounded sequence.
Therefore we can choose one single contour $\CC^-$ for which the integral representation $(\ref{defsasymptotics})$ holds for all $N\geq 1$.

Once we have the integral formula in the form $(\ref{rewrittenformula1})$, there are several techniques to obtain asymptotic expansions as $N\rightarrow\infty$, such as the saddle-point method; see the classical text \cite{C}.

The plan is the following.
First we show that, for the study of asymptotics as $N \rightarrow \infty$, we can replace contour $\CC^-$ in formula $(\ref{rewrittenformula1})$ by a finite contour going through the critical point of $w(z)$.
Second, we prove $\lim_{N\rightarrow\infty}{\psi(z; \lambda(N), \theta)} = \sqrt{\frac{z - \theta}{z}}\cdot\Psi\left( \frac{z}{z - \theta}; \omega, \theta \right)$, uniformly on compact subsets, and moreover argue that we can replace $\psi(z; \lambda(N), \theta)$ by $\sqrt{\frac{z - \theta}{z}}\cdot\Psi\left( \frac{z}{z - \theta}; \omega, \theta \right)$ in the integral without modifying the asymptotics as $N\rightarrow\infty$. Third, after the modifications, we apply the saddle-point method to finish the proof.
The first and second items in this plan are carried out in Subsection $\ref{sec:deformationapp1}$ and the saddle-point method analysis is given in Subsection $\ref{saddleapp1}$.

\subsection{Deformation of the integral}\label{sec:deformationapp1}

We make use of Cauchy's theorem to deform $\CC^-$ so that it passes through the \textit{critical points} of $w(z)$.
Recall that the contour $\CC^-$ consists of the segment $[M + r\ii, \ M - r\ii]$ and the horizontal lines $[M + r\ii, -\infty+r\ii)$, $[M - r\ii, -\infty - r\ii)$, for some $M > \lambda_1$ and some $r > 0$.
The critical points $z_0$ of $w(z)$ are the solutions to the equation $w'(z) = 0$.
For $w(z)$, as defined in $(\ref{defsasymptotics})$ above, the equation is $y - \ln{z} + \ln{(z-\theta)} = 0$, and has the unique solution
\begin{equation}\label{criticalpt}
z_0 = \frac{\theta}{1 - e^{-y}}.
\end{equation}
Observe that $z_0$ depends on $y$ and $\theta$, but for simplicity we write $z_0$ instead of $z_0(y, \theta)$. We claim that there exists $\delta > 0$ such that for any compact subset $K_{\delta} \subset V_{\delta} := \{v\in\C : 0 < \Re v < \delta, \ -\pi < \Im v < \pi\}$ and any $y\in K_{\delta}$, we can find a contour $\gamma(s)$, $s\in\R$, with the following features:
\begin{enumerate}
	\item The contour $\gamma(s)$, $s\in\R$, is a piecewise smooth contour in $\C$, with $\gamma(0) = z_0$.
	\item There exists $\epsilon_0 > 0$ so that $\gamma(s) = z_0 + se^{\ii\phi}\ \forall s\in[-\epsilon_0, \epsilon_0]$ and $\phi$ is such that $w''(z_0)e^{2\ii\phi} < 0$. In the terminology of \cite{C}, this means that $\gamma$ is locally in the \textit{critical direction}.
	\item The contour $\gamma$ is positively oriented and contains $(-\infty, L]$ in its interior, where we denote
\begin{equation*}
L  = L(\theta) \myeq \sup_{N\geq 1}{\left|\frac{\lambda_1(N)}{N}\right|} + \theta + 1.
\end{equation*}
	\item $\Re w(\gamma(0)) = \sup_{t\in\R}{\Re w(\gamma(t))}$ and $\Re w(\gamma(0)) = \Re w(z_0) > \Re w(t)$ for all $t\neq 0$.
	\item $|\Im z|$ is constant, for $z\in\gamma$ with $|\Re z|$ is large enough. For convenience, let us require that $|\Im z| = r$ for $z\in\gamma$ with $|\Re z|$ large enough, and $r > 0$ is the positive real number used in the definition of $\CC^-$.
\end{enumerate}

Condition (5) ensures that $\int_{\gamma}{\exp(Nw(z))\psi(z; \lambda(N), \theta)dz}$ converges absolutely and also that one needs only to modify a compact subset of $\CC^-$ in order to obtain $\gamma$. Condition (4) will allow us to estimate the value of the whole integral by the contribution of a small neighborhood of the critical point $z_0$. In fact, away from $z_0$, the real part of $w(z)$ decreases implying that the integrand decreases exponentially. Condition (3) tells us that we can replace $\CC^-$ by $\gamma$ into the integral representation $(\ref{rewrittenformula1})$ because no new poles are picked up in a deformation from $\CC^-$ into $\gamma$. The conditions (1), (2) are useful in applying the saddle-point method in the next section.

The proof of existence of a contour $\gamma$ satisfying the conditions above was given in \cite[Section 4.2]{GP}, except that condition (3) was not mentioned and condition (2) was expressed differently. Let us briefly review the proof in \cite{GP} by considering these additional points. Our first task is to look at the level line $\{ z\in\C : \Re w(z) = \Re w(z_0) \}$. From the Taylor expansion of $w(z)$ near $z_0$, there are four directions to the level line $\Re w(z) = \Re w(z_0)$.
Moreover, for large values of $R > 0$, the level line $\Re w(z) = \Re (yz - H(z; \theta)) = \Re w(z_0)$ intersects the circle $\TT_R := \{z\in\C : |z| = R\}$ at exactly two points (which are near the two points of intersection of the line $\Re yz = \Re w(z_0)$ with $\TT_R$).
This means that the level line $\Re w(z) = \Re w(z_0)$ has a loop.
Due to the maximum principle, the loop must contain some points in the interval $[0, \theta]$, which are the only points on which $\Re w(z)$ is not a harmonic function.
When $-\pi < \Im y < 0$, $\Im y = 0$ and $0 < \Im y < \pi$, we have that $z_0 = \theta/(1 - e^{-y})$ has positive, zero and negative imaginary part, respectively; moreover the straight lines $\Re yz = \Re w(z_0)$ have negative, infinite and positive slope, respectively.
Therefore the level lines $\Re w(z) = \Re w(z_0)$ look qualitatively different in these cases, for example, see Figures $\ref{fig:sub1}$ and $\ref{fig:sub2}$ for the cases $\Im y < 0$ and $\Im y > 0$.

\begin{figure}
\centering
\begin{subfigure}{.45\textwidth}
  \centering
  \includegraphics[width=.4\linewidth]{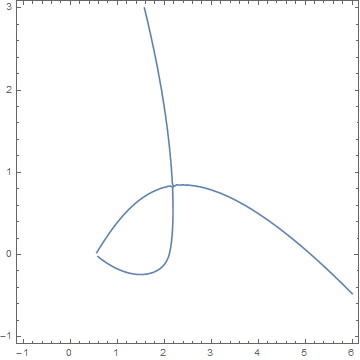}
  \caption{$y = 1 - \ii$ ($\theta = 2$)}
  \label{fig:sub1}
\end{subfigure}
\begin{subfigure}{.45\textwidth}
  \centering
  \includegraphics[width=.4\linewidth]{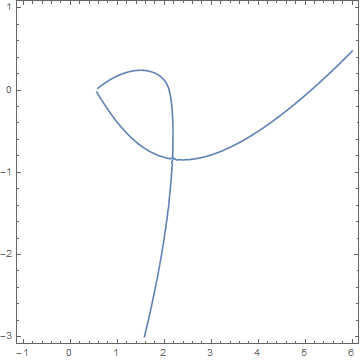}
  \caption{$y = 1 + \ii$ ($\theta = 2$)}
  \label{fig:sub2}
\end{subfigure}
\caption{Level lines of $\Re(yz - H(z; 2)) = \Re(yz_0 - H(z_0; 2))$ for two values of $y$}
\end{figure}

The complex plane is divided into three regions by the level line $\{z\in\C : \Re w(z) = \Re w(z_0)\}$, one is a bounded region and two are unbounded regions, one of which contains $+\infty$ and the other one contains $-\infty$.
The bounded region and the unbounded region that contains $+\infty$ are disjoint and their union is $\{z\in\C : \Re w(z) > \Re w(z_0)\}$. On the other hand, the unbounded region that contains $-\infty$ is $\{z\in\C : \Re w(z) < \Re w(z_0)\}$.
Thus it is possible to deform and reparametrize a compact portion of the contour $\CC^-$ so that the new contour is fully contained in the region $\{z\in\C : \Re w(z) \leq \Re w(z_0)\}$ and passes through $z_0$ at time $0$.
Thus it is not difficult to find a contour that satisfies conditions (1), (4) and (5) above. We need, however, to prove that it can be deformed to satisfy (2) and (3).

The contour can be modified near $z_0$ to be in the critical direction, thus satisfying condition (2), the only issue being that there may not exist $\epsilon_0 > 0$ such that $[z_0 - \epsilon_0e^{\ii\phi}, z_0 + \epsilon_0e^{\ii\phi}] \cap (-\infty, L] = \emptyset$.
We argue this is always the case, however. First, a simple calculation shows that for sufficiently small $\epsilon_1, \delta >0$, and any $y\in\C$ with $0 < \Re y < \delta$, $|\Im y| \leq \epsilon_1$, the intersection of $x$-axis with the line $\{z = z_0 + te^{\ii\phi} : t\in\R\}$ is a positive real point larger than $L$, and thus for any $\epsilon_0 > 0$ we will have $[z_0 - \epsilon_0e^{\ii\phi}, z_0 + \epsilon_0e^{\ii\phi}] \cap (-\infty, L] = \emptyset$.
Second, for $y\in\C$ with $0 < \Re y < \delta$, $|\Im y| \in (\epsilon_1, \pi - \epsilon_1)$, we easily find $|\Im z_0| > (\theta e^{-\delta}\sin \epsilon_1)/5$; thus if we set $\epsilon_0 \myeq (\theta e^{-\delta}\sin \epsilon_1)/10$, we again guarantee $[z_0 - \epsilon_0e^{\ii\phi}, z_0 + \epsilon_0e^{\ii\phi}] \cap (-\infty, L] = \emptyset$, if $z_0 = z_0(y, \theta)$ and $|\Im y| \in (\epsilon_1, \pi - \epsilon_1)$.
Thus let $\delta > 0$ be chosen as above.
For any compact subset $K_{\delta} \subset V_{\delta}$, we can find $\epsilon_1 > 0$ small enough so that $K_{\delta} \subset V_{\delta} \cap \{z\in\C : |\Im z| < \pi - \epsilon_1\}$ and then by setting $\epsilon_0 = (\theta e^{-\delta}\sin \epsilon_1)/10$, we are assured that we can deform the contour so as to satisfy condition (2), without violating (3).

As for condition (3), we need to prove it is always possible to deform the contour to make it contain $(-\infty, L]$ in its interior, but we need to guarantee that condition (4) remains in place.
The only way that would prevent a contour to go around $(-\infty, L]$ without violating (4) is if the rightmost point of intersection between the curve $\Re w(z) = \Re w(z_0)$ and the $x$-axis is to the left of $L$.
We can easily show that $\Re w(x)$ is an increasing function on $[(1 - e^{-\Re y})^{-1}, \infty)$, its value at $x = (1 - e^{-\Re y})^{-1}$ is at most equal to $\Re w(z_0) = \theta(\Re y + 1 - \ln{\theta} + \ln{|1 - e^{-y}|})$ and $\Re w(x)$ goes to infinity with $x$. Therefore the rightmost point of intersection between the curve $\Re w(z) = \Re w(z_0)$ and the $x$-axis is to the right of $(1 - e^{-\Re y})^{-1}$. Thus if $y$ is such that $(1 - e^{-\Re y})^{-1} > L$, it is impossible that the rightmost point of intersection between the curve $\Re w(z) = \Re w(z_0)$ and the $x$-axis is to the left of $L$. By decreasing the value of $\delta > 0$ to a smaller positive value, we can guarantee $(1 - e^{-\Re y})^{-1} > L$, for all $y\in V_{\delta}$, and therefore we can make the contour $\gamma$ so that it satisfies (3), as we wished.

We can actually modify the contour $\gamma$ to strengthen condition (4).
Since $w(z)$ is smooth near $z_0$, we can reduce $\delta > 0$, and also $\epsilon_0 > 0$ if necessary, so that
both of the rightmost points of intersection between the level lines $\Re w(z) = \Re w(z_0 \pm \epsilon_0)$ and the $x$-axis are to the right of $L$.
The implication is that we can deform $\gamma$ in such a way that $\Re w(\gamma(t))$ attains a maximum at $t = 0$, then it decreases as $t$ goes from $0$ to $\pm\epsilon_0$ and then it takes values strictly less than $\min\{\Re\gamma(\epsilon_0), \Re\gamma(-\epsilon_0)\}$, for $|t| > \epsilon_0$.
We thus can postulate a new condition for contour $\gamma$:

\begin{enumerate}
	\setcounter{enumi}{5}
	\item $\sup_{|t| > \epsilon_0}{\Re w(\gamma(t))} \leq \min\{\Re w(\gamma(t - \epsilon_0 e^{\ii \phi})), \Re w(\gamma(t + \epsilon_0 e^{\ii \phi}))\} = \inf_{|t|\leq \epsilon_0}{\Re w(\gamma(t))}$.
\end{enumerate}

We have argued thus far that the contour $\CC^-$ in the integral representation $(\ref{rewrittenformula1})$ of $P_{\lambda(N)}(e^y; N, \theta)$ can be replaced by a contour $\gamma$, satisfying conditions (1)-(6) above. Next we make some estimates on this expression and reduce it to a form on which we can apply the saddle-point method.

The first claim is that, for large enough $M > 0$, the contribution of the portion $\gamma_{<-M} := \gamma\cap\{z\in\C : \Re z < -M\}$ is negligible, as $N\rightarrow\infty$.
We only need very crude estimates. Let $a^+ > 0 > a^-$ be reals such that $a^-N \leq \lambda_N(N) \leq \dots \leq \lambda_1(N) \leq a^+N$, for all $N \geq 1$.
Observe they must exist because $\{\lambda(N)\}_{N \geq 1}$ is a VK sequence of signatures.
We have $|Nz - s| < |Nz - s - \theta|$, for all $z\in\CC^-$, as long as $\Re (Nz - s) < 0$.
It follows that the absolute value of $\Gamma(Nz - s - \theta)/\Gamma(Nz - s)$ is less than the absolute value of $\Gamma(Nz +1 - s - \theta)/\Gamma(Nz + 1 - s)$, as long as $\Re (Nz - s) < 0$.
Now consider any $z\in\CC^-$ such that $\Re z - a^- < 0$ and let $K = K(z)\in\Z_{\geq 0}$ be the nonnegative integer such that $-1\leq \Re z + K - a^- < 0$.
Then we have that the absolute value of $\Gamma(Nz - \lambda_i(N) - \theta)/\Gamma(Nz - \lambda_i(N))$ is at most equal to the absolute value of $\Gamma(N(z + K + a^-) - \theta)/\Gamma(N(z + K + a^-))$.
Note that, if $z\in\CC^-$ and $\Re z$ is small enough, say $\Re z < a^-$ is enough, then by definition of $\CC^-$ the imaginary part $|\Im z|$ is constant and equal to some $r > 0$. Therefore $z + K + a^-$ belongs to the disconnected, but compact, set $\{z\in\C : -1 \leq \Re z \leq 0, \ |\Im z| = r\}$.
By well known asymptotics of the Gamma function, it follows that there exists a constant $c > 0$ such that the absolute value of $\Gamma(N(z + K + a^-) - \theta)/\Gamma(N(z + K + a^-))$ is at most $(Nc)^{-\theta}$.
The conclusion from the estimate above is that
\begin{equation*}
\begin{gathered}
\sup_{N \geq 1}\sup_{\substack{z\in\CC^- \\ \Re z < a^-}}{ \left|  N^{N\theta}\prod_{i=1}^N{\frac{\Gamma(Nz+1-(\lambda_i(N) + \theta(N - i + 1)))}{\Gamma(Nz+1-(\lambda_i(N) + \theta(N - i)))}} \right| } \leq N^{N\theta}\prod_{i=1}^N{(Nc)^{-\theta}} = c^{-N\theta},\\
\Longrightarrow |\psi(z; \lambda(N), \theta)| \leq c^{-N\theta}\cdot\exp(N\Re H(z; \theta)) = \exp(N(\Re H(z; \theta) - \theta \ln{c})).
\end{gathered}
\end{equation*}

Write $\int_{\gamma_{<-M}}{\exp(Nw(z))\psi(z)dz} = \exp(Nw(z_0))\int_{\gamma_{<-M}}{\exp(N(w(z) - w(z_0)))\psi(z)dz}$ and observe that the prefactor $\frac{\Gamma(\theta N)N^{1 - \theta N}}{(e^y - 1)^{\theta N - 1}2\pi\ii}$, in $(\ref{rewrittenformula1})$, times $\exp(Nw(z_0)) = (e^y - 1)^{N\theta}e^{N\theta}\theta^{-N\theta}$ has absolute value of order $O(N^{1/2})$, because of Stirling's formula.
From the estimates above, we have $\int_{\gamma_{<-M}}{|\exp(N(w(z) - w(z_0)))\psi(z)|dz}\leq \int_{\gamma_{<-M}}{\exp(N(\Re w(z) - \Re w(z_0) + \Re H(z; \theta) - \theta \ln{c}))dz}$.
But $\Re w(z) - \Re w(z_0) + \Re H(z; \theta) - \theta \ln{c}$ decreases as $z \in \CC^-$ ranges from $z = -M \pm r\ii$ to $-\infty \pm r\ii$, and it is always negative if $M > 0$ is large enough so that $\Re w(z) - \Re w(z_0) + \Re H(z; \theta) - \theta \ln{c} < 0$ for all $z = -M \pm r\ii$. It follows that the absolute value of $\int_{\gamma_{<-M}}{\exp(N(w(z) - w(z_0)))\psi(z; \lambda(N), \theta)dz}$ is exponentially small, as $N\rightarrow\infty$, and in particular it is of order $o(N^{-1/2})$. Thus we conclude that to calculate the limit of $(\ref{rewrittenformula1})$ as $N$ tends to infinity, we can replace $\CC^-$ with the compact contour $\gamma_{\geq -M} = \gamma\cap\{z\in\C : \Re z \geq -M\}$, and pick up an additive error of order $o(1)$.

The next claim is that we can replace $\psi(z; \lambda(N), \theta)$ in the integral formula $(\ref{rewrittenformula1})$, but with contour $\CC^-$ replaced by $\gamma_{\geq -M}$, by $\sqrt{\frac{z - \theta}{z}}\cdot\widetilde{\Psi}(z; \omega, \theta)$ (the function $\widetilde{\Psi}$ is defined in the Lemma $\ref{limitphi}$ below) and the additive error as a result is of order $o(1)$, as $N$ tends to infinity. This is a consequence of the following lemma, whose proof we postpone till the end of this section, and which implies the uniform convergence $\lim_{N\rightarrow\infty}{\psi(z; \lambda(N), \theta)} = \sqrt{\frac{z - \theta}{z}}\cdot\widetilde{\Psi}(z; \omega, \theta)$ on the compact subset $\gamma_{\geq -M}$.

\begin{lem}\label{limitphi}
The following limit
\begin{equation*}
\lim_{N\rightarrow\infty}{\psi(z; \lambda(N), \theta)} = \sqrt{\frac{z - \theta}{z}}\times\widetilde{\Psi}(z; \omega, \theta),
\end{equation*}
\begin{equation*}
\widetilde{\Psi}(z; \omega, \theta) \myeq \exp\left(\frac{\theta\gamma^+}{z - \theta} - \frac{\theta\gamma^-}{z} \right)\prod_{i=1}^{\infty}{\frac{(1 + \theta \beta_i^+/(z - \theta))(1 - \theta\beta_i^-/z)}{(1 - \alpha_i^+/(z - \theta))^{\theta}(1 + \alpha_i^-/z)^{\theta}}},
\end{equation*}
holds pointwise for all $z$ in the complex plane cut along $(-\infty, \theta]$. The square root in the limiting expression, as well as the powers $x^{\theta} = \exp(\theta\ln{x})$ in the denominator of $\widetilde{\Psi}$, take their principal values.
\end{lem}

The final claim is that contour $\gamma_{\geq -M}$ can be replaced by the neighborhood $\gamma(\epsilon_0) \myeq \gamma \cap \{z\in\C : |z - z_0| \leq \epsilon_0\} = [z_0 - \epsilon_0 e^{\ii \phi}, z_0 + \epsilon_0 e^{\ii \phi}]$ with no change in the asymptotics of $(\ref{rewrittenformula1})$, as $N\rightarrow\infty$.
This is a consequence of condition (6) of contour $\gamma$, which implies that $\sup_{|t| \geq \epsilon_0}{\Re w(\gamma(t))} \leq \inf_{|t| \leq \epsilon_0}{\Re w(\gamma(t))}$ and therefore the contribution of $\gamma_{\geq -M} \setminus \gamma(\epsilon_0)$ is
\begin{equation*}
\frac{\Gamma(\theta N)N^{1-\theta N}}{(e^y - 1)^{\theta N - 1}2\pi\ii}\exp(Nw(z_0))\int_{\gamma_{\geq -M} \setminus \gamma(\epsilon_0)}{\exp(N (w(z) - w(z_0)) )\sqrt{\frac{z - \theta}{z}}\widetilde{\Psi}(z; \omega, \theta)dz}.
\end{equation*}
Since $\exp(Nw(z_0)) = (e^y - 1)^{N\theta}e^{N\theta}\theta^{-N\theta}$, then the factor before the integral is of order $O(N^{1/2})$.
On the other hand, the contribution of the integral above can be bounded by a constant factor times $\exp(N(\Re w(z_0 \pm \epsilon_0 e^{\ii \phi}) - \Re w(z_0) )) \leq const\times \exp( N w''(z_0) (\epsilon_0 e^{\ii\phi})^2/2 ) =  const\times \exp(N w''(z_0)e^{2\ii\phi} \epsilon_0^2/2)$, which is exponentially decreasing on $N$ by condition (2), and in particular of order $o(N^{-1/2})$.
Therefore the contribution of $\gamma_{\geq M} \setminus \gamma(\epsilon_0)$ is of order $o(1)$, and our claim is proved. 

In summary, we have just given reasoning for the following successive estimations
\begin{equation*}
\begin{gathered}
\frac{\Gamma(\theta N)N^{1-\theta N}}{(e^y - 1)^{\theta N - 1}2\pi\ii}\int_{\CC^-}{\exp(Nw(z))\psi(z; \lambda(N), \theta) dz}\\
= \frac{\Gamma(\theta N)N^{1-\theta N}}{(e^y - 1)^{\theta N - 1}2\pi\ii}\int_{\gamma}{\exp(Nw(z))\psi(z; \lambda(N), \theta) dz}\\
= \frac{\Gamma(\theta N)N^{1-\theta N}}{(e^y - 1)^{\theta N - 1}2\pi\ii}\int_{\gamma_{\geq -M}}{\exp(Nw(z))\psi(z; \lambda(N), \theta) dz} + o(1)\\
= \frac{\Gamma(\theta N)N^{1-\theta N}}{(e^y - 1)^{\theta N - 1}2\pi\ii}\int_{\gamma_{\geq -M}}{\exp(Nw(z)) \sqrt{\frac{z - \theta}{z}} \widetilde{\Psi}(z; \omega, \theta) dz}\cdot (1 + o(1)) + o(1)\\
= \frac{\Gamma(\theta N)N^{1-\theta N}}{(e^y - 1)^{\theta N - 1}2\pi\ii}\int_{\gamma(\epsilon_0)}{\exp(Nw(z)) \sqrt{\frac{z - \theta}{z}} \widetilde{\Psi}(z; \omega, \theta) dz} \cdot (1 + o(1)) + o(1).
\end{gathered}
\end{equation*}

The asymptotic expansion of the factor before the integral can be calculated by Stirling's formula.
Thus we only need to find the limit, as $N\rightarrow\infty$, of the following expression and it will be carried out in the next subsection:
\begin{equation}\label{toapplysaddle}
\begin{gathered}
\frac{\Gamma(\theta N)N^{1-\theta N}}{(e^y - 1)^{\theta N - 1}2\pi\ii}\int_{\gamma(\epsilon_0)}{\exp(Nw(z)) \sqrt{\frac{z - \theta}{z}} \widetilde{\Psi}(z; \omega, \theta) dz}\\
= \left(1 + O\left(\frac{1}{N}\right)\right)\times\frac{N^{1/2}\theta^{\theta N - 1/2}}{(e^y - 1)^{\theta N - 1}e^{\theta N}\sqrt{2\pi}\ii} \int_{\gamma(\epsilon_0)}{\exp(Nw(z))\sqrt{\frac{z - \theta}{z}}\widetilde{\Psi}(z; \omega, \theta)dz}.
\end{gathered}
\end{equation}

\begin{proof}[Proof of Lemma $\ref{limitphi}$]
Let us define
\begin{eqnarray*}
\phi_1(z; N, \theta) &\myeq& \exp(NH(z; \theta)) N ^{\theta N} \frac{\Gamma(N(z - \theta) + 1)}{\Gamma(Nz + 1)}\\
\phi_2(z; \lambda(N), \theta) &\myeq& \prod_{i=1}^N{\frac{\Gamma(Nz + 1 - (\lambda_i(N) + \theta(N - i + 1)))\Gamma(Nz + 1 - \theta(N - i))}{\Gamma(Nz + 1 - (\lambda_i(N) + \theta(N - i)))\Gamma(Nz + 1 - \theta(N - i + 1))}}
\end{eqnarray*}
and observe that $\psi(z; \lambda(N), \theta) = \phi_1(z; N, \theta)\phi_2(z; \lambda(N), \theta)$. We find the limits of both $\phi_1$ and $\phi_2$ as $N\rightarrow\infty$ .

From Stirling's formula, $\Gamma(w) = \sqrt{2\pi} (w - 1)^{w - 1/2}e^{1 - w}\left( 1 + O(1/w) \right)$, for $|w|\rightarrow\infty$, uniform on $|\arg w| \leq \pi - \epsilon < \pi$. It follows that
\begin{eqnarray*}
\phi_1 = \exp(N(z\ln{z} - (z - \theta)\ln{(z - \theta)} - \theta)) N^{\theta N} e^{\theta N}\frac{(N(z - \theta))^{N(z-\theta) + 1/2}}{(Nz)^{Nz + 1/2}}\times\left( 1 + O(1/N) \right).
\end{eqnarray*}
After obvious simplifications, we have $\phi_1 = \sqrt{(z - \theta)/z}\times(1 + O(1/N))$, as $N\rightarrow\infty$, and therefore
\begin{equation}\label{phi1limit}
\lim_{N\rightarrow\infty}{\phi_1(z; N, \theta)} = \sqrt{\frac{z - \theta}{z}}
\end{equation}
for any $z\in\C \setminus (-\infty, \theta]$. Now we look at $\phi_2$, which can be rewritten as
\begin{equation}\label{phi2}
\begin{gathered}
\phi_2(z; \lambda(N), \theta) = \prod_{i=1}^{\ell(\lambda^+(N))}{\frac{\Gamma(Nz + 1 - \theta(N - i + 1) - \lambda_i^+(N))\Gamma(Nz + 1 - \theta(N - i))}{\Gamma(Nz + 1 - \theta(N - i + 1))\Gamma(Nz + 1 - \theta(N - i) - \lambda_i^+(N))}}\\
\times\prod_{i=1}^{\ell(\lambda^-(N))}{\frac{\Gamma(Nz + 1 - \theta i + \lambda_i^-(N))\Gamma(Nz + 1 - \theta i + \theta)}{\Gamma(Nz + 1 - \theta i)\Gamma(Nz + 1 - \theta i + \lambda_i^-(N) + \theta)}}
\end{gathered}
\end{equation}
We call $\phi_2^{(1)}(z; \lambda(N), \theta)$ and $\phi_2^{(2)}(z; \lambda(N), \theta)$ to the first and second products of $(\ref{phi2})$, respectively. We claim
\begin{equation}\label{phi2limit}
\lim_{N\rightarrow\infty}{\phi_2^{(1)}(z; \lambda(N), \theta)} = \exp\left( \frac{\theta \gamma^+}{z - \theta} \right)\times\prod_{i=1}^{\infty}{\frac{1 + \theta\beta_i^+/(z - \theta)}{(1 - \alpha_i^+/(z - \theta))^{\theta}}}.
\end{equation}
Observe that the right side above is part of the expression in the definition of $\widetilde{\Psi}(z; \omega, \theta)$. Similarly, $\phi_2^{(2)}(z; \lambda(N), \theta)$ will converge to the remaining product in the definition of $\widetilde{\Psi}(z; \omega, \theta)$, and the proof is analogous. Therefore we shall only prove $(\ref{phi2limit})$, leaving the second part as an exercise to the reader.

From $\Gamma(t + 1) = t\Gamma(t)$, we can write the expression $\phi_2^{(1)}(z; \lambda(N), \theta)$ as
\begin{eqnarray*}
\phi_2^{(1)}(z; \lambda(N), \theta) &=& \prod_{i=1}^{\ell(\lambda^+(N))}{\frac{(Nz - \theta(N - i))\cdots (Nz + 1 - \lambda_i^+(N) - \theta(N - i))}{(Nz - \theta(N - i + 1))\cdots (Nz + 1 - \lambda_i^+(N) - \theta(N - i + 1))}}\\
&=&\prod_{s\in\lambda^+(N)}{\frac{N(z - \theta) + \theta(l'(s) + 1) - a'(s)}{N(z - \theta) + \theta l'(s) - a'(s)}}.
\end{eqnarray*}
Each term in the last product is indexed by a square in the Young diagram of $\lambda$. Recall that the $k$-th hook in a Young diagram $\lambda^+(N)$ is the set of its squares with coordinates $(i, j)$ such that $\min\{i, j\} = k$. Let $d(\lambda^+(N))$ be the diagonal length of the partition $\lambda^+(N)$, or equivalently the number of distinct hooks in the Young diagram of $\lambda^+(N)$. If we group the terms in the product above by the hook to which they belong, we can express the product as
\begin{equation}\label{afterhooks}
\begin{gathered}
\phi_2^{(1)}(z; \lambda(N), \theta) = \prod_{i=1}^{d(\lambda^+(N))} \left\{\frac{N(z - \theta) + \theta \lambda^+(N)'_i - (i-1)}{N(z - \theta) + \theta(i - 1) - (i-1)}\right.\times\\
\left.\times\frac{\Gamma(N(z - \theta) + \theta i - i + 1)\Gamma(N(z - \theta) + \theta(i-1) - \lambda_i^+(N) + 1)}{\Gamma(N(z - \theta) + \theta(i - 1) - i +1)\Gamma(N(z - \theta) + \theta i - \lambda_i^+(N) + 1)}\right\}.
\end{gathered}
\end{equation}
For any fixed $i = 1, 2, \ldots$, by definition of a VK sequence and its boundary point, we have
\begin{equation}\label{finite:limits}
\begin{gathered}
\lim_{N\rightarrow\infty}{\frac{N(z - \theta) + \theta \lambda^+(N)'_i - (i-1)}{N(z - \theta) + \theta(i - 1) - (i-1)}} = 1 + \theta \beta_i^+/(z - \theta),\\
\lim_{N\rightarrow\infty}{\frac{\Gamma(N(z - \theta) + \theta i - i + 1)\Gamma(N(z - \theta) + \theta(i-1) - \lambda_i^+(N) + 1)}{\Gamma(N(z - \theta) + \theta(i - 1) - i +1)\Gamma(N(z - \theta) + \theta i - \lambda_i^+(N) + 1)}} = (1 - \alpha_i^+/(z - \theta))^{-\theta},
\end{gathered}
\end{equation}
so if $d(\lambda^+(N))$ remained bounded as $N$ goes to infinity, then $\gamma^+ = 0$, and the limits above prove the desired $(\ref{phi2limit})$. In general, $d(\lambda^+(N))$ grows to infinity with $N$, but since $|\lambda^+(N)| \geq d(\lambda^+(N))^2$ (the square with main diagonal of size $d(\lambda^+(N))$ is contained in the Young diagram of $\lambda^+(N)$) and $\lim_{N\rightarrow\infty}{\frac{|\lambda^+(N)|}{N}}$ exists, then $d(\lambda^+(N)) = O(\sqrt{N})$ and therefore $d(\lambda^+(N))$ does not grow too fast. In this general case, we need an additional argument to prove $(\ref{phi2limit})$, which shows how the factor $\exp(\theta\gamma^+/(z-\theta))$ appears.

We take the logarithm of $\phi_2^{(1)}(z; \lambda(N), \theta)$ and show that it converges to the logarithm of the right-hand side of $(\ref{phi2limit})$. For that, we express $\ln{\phi_2^{(1)}(z; \lambda(N), \theta)}$ as a power series on $(z - \theta)^{-1}$ and show that the coefficients converge to the corresponding coefficients of the power series of
\begin{equation}\label{taylorcoeff}
\begin{gathered}
 \frac{\theta \gamma^+}{z - \theta} + \sum_{i=1}^{\infty}{\left(\ln{\left(1 + \frac{\theta\beta_i^+}{(z - \theta)}\right)} - \theta\ln{\left(1 - \frac{\alpha_i^+}{(z - \theta)}\right)}\right)}\\
= \frac{\theta}{z - \theta}\left(\gamma^+ + \sum_{i=1}^{\infty}{(\alpha_i^+ + \beta_i^+)}\right) + \sum_{j=2}^{\infty}{\frac{(-1)^{j-1}\theta^j}{j(z - \theta)^j}\sum_{i=1}^{\infty}{(\beta_i^+)^j}} + \sum_{j=2}^{\infty}{\frac{\theta}{j(z - \theta)^j}\sum_{i=1}^{\infty}{(\alpha_i^+)^j}}.
\end{gathered}
\end{equation}

Note that we assumed that all absolute values of $\theta\beta_i^+/(z - \theta)$ and $\alpha_i^+/(z - \theta)$, $i \geq 1$, are less than $1$, in order to use the expansion $\ln{(1 + x)} = x - x^2/2 + \ldots$. However, there is no loss in generality, because $\sum_{i = 1}^{\infty}{\alpha_i^+} < \infty$ implies $|\alpha_i^+/(z - \theta)| < 1$ for large enough $i > N_0$, and similarly $|\theta\beta_i^+/(z - \theta)| < 1$ for large enough $i > N_0$.
Thus we can take the logarithms of the products starting from such large $N_0$ and prove convergence of the product $(\ref{afterhooks})$ with $\prod_{i=1}^{d(\lambda^+(N))}$ replaced by $\prod_{i=N_0}^{d(\lambda^+(N))}$ to the product $(\ref{phi2limit})$ with $\prod_{i=1}^{\infty}$ replaced by $\prod_{i = N_0}^{\infty}$, and then deal with the few remaining terms by using the limits $(\ref{finite:limits})$.
So we shall assume without loss of generality that $|\alpha_i^+/(z - \theta)|, |\theta \beta_i^+/(z - \theta)| < 1$ for all $i = 1, 2, \ldots$. When we take logarithms of the terms in $\phi_2^{(1)}(z; \lambda(N), \theta)$, we shall also make use of similar Taylor expansions and accordingly we need to assume $|\theta (\lambda_i'(N))^+/(z - \theta)| < 1$ for all $i = 1, 2, \ldots$, as well as other similar relations; we always assume they hold, since there is no loss of generality, as we just argued.

Now we find the Taylor series coefficients of $\phi_2^{(1)}(z; \lambda(N), \theta)$, in the product form $(\ref{afterhooks})$, and show they converge to the Taylor series coefficients in $(\ref{taylorcoeff})$ above. For conciseness, we prove the limits of the coefficients of $(z - \theta)^{-1}$ and $(z - \theta)^{-2}$, and only sketch how to extend the analysis for higher powers $(z - \theta)^{-k}$. For the first product of $(\ref{afterhooks})$, we have
\begin{equation}\label{taylorseries1}
\begin{gathered}
\ln\left(\prod_{i=1}^{d(\lambda^+(N))}{\frac{N(z - \theta) + \theta \lambda^+(N)'_i - (i-1)}{N(z - \theta) + \theta(i - 1) - (i-1)}}\right) = \\
\sum_{i=1}^{d(\lambda^+(N))}{\ln\left( 1 + \frac{\theta \lambda^+(N)'_i - i + 1}{N(z - \theta)} \right) - \ln\left( 1 + \frac{\theta(i - 1) - i + 1}{N(z - \theta)} \right)} = \\
\frac{\theta}{z - \theta}\sum_{i=1}^{d(\lambda^+(N))}{\frac{b_i^+(\lambda(N)) + \frac{1}{2}}{N}} + \sum_{j=2}^{\infty}\frac{(-1)^{j-1}}{jN^j(z - \theta)^j}\sum_{i=1}^{d(\lambda^+(N))} \left( \theta \lambda^+(N)'_i + 1 - i \right)^j - \left( \theta(i - 1) + 1 - i \right)^j.
\end{gathered}
\end{equation}
For the second product of $(\ref{afterhooks})$, we use the asymptotic expansion of the log-Gamma function
\begin{equation*}
\ln{\Gamma(z)} = \left( \frac{\ln{2\pi}}{2} - z + (z - \frac{1}{2})\ln{z} \right)\left( 1 + O(1/z) \right), \hspace{.2in} |z| \rightarrow \infty, \ |\arg(z)| \leq \pi - \delta < \pi,
\end{equation*}
to obtain a Taylor series (on $(z - \theta)^{-1}$) for it. The first two terms of the Taylor expansion give
\begin{equation}\label{taylorseries2}
\begin{gathered}
\ln\left(\prod_{i=1}^{d(\lambda^+(N))}{\frac{\Gamma(N(z - \theta) + \theta i - i + 1)\Gamma(N(z - \theta) + \theta(i-1) - \lambda_i^+(N) + 1)}{\Gamma(N(z - \theta) + \theta(i - 1) - i +1)\Gamma(N(z - \theta) + \theta i - \lambda_i^+(N) + 1)}}\right)\\
= \frac{\theta}{(z - \theta)}\sum_{i=1}^{d(\lambda^+(N))}{\frac{a_i^+(\lambda(N)) - \frac{1}{2}}{N}}\\
+ \frac{\theta}{(z - \theta)^2}\sum_{i=1}^{d(\lambda^+(N))}{\frac{(2a_i^+(\lambda(N)) - 1)(2a_i^+(\lambda(N)) + 2\theta + 4i - 3 - 4\theta i)}{8N^2}} + O((z - \theta)^{-3}).
\end{gathered}
\end{equation}
Thus by combining $(\ref{taylorseries1})$ and $(\ref{taylorseries2})$, we have that the expansion of $\ln{\phi_2^{(1)}(z; \lambda(N), \theta)}$ in powers of $(z - \theta)^{-1}$, up to the second term, is
\begin{equation}\label{taylorcompare}
\frac{\theta}{z - \theta}\sum_{i=1}^{d(\lambda^+(N))}{ \frac{a_i^+(N) + b_i^+(N)}{N} } + \frac{1}{2(z - \theta)^2}\sum_{i=1}^{d(\lambda^+(N))}{\left( \frac{\theta a_i^+(\lambda(N))^2 - \theta^2 b_i^+(\lambda(N))^2}{N^2} + \textrm{ `other terms' } \right)}
\end{equation}
where `other terms' in $(\ref{taylorcompare})$ stands for a linear combination of sums
\begin{equation*}
\sum_{i=1}^{d(\lambda^+(N))}{ \frac{ia_i^+(\lambda(N))}{N^2} }, \ \sum_{i=1}^{d(\lambda^+(N))}{ \frac{a_i^+(\lambda(N))}{N^2} }, \ \sum_{i=1}^{d(\lambda^+(N))}{ \frac{i}{N^2} }, \textrm{ and } \sum_{i=1}^{d(\lambda^+(N))}{\frac{1}{N^2}}
\end{equation*}
(and similar ones with $b_i^+(\lambda(N))$ replacing $a_i^+(\lambda(N))$).

The coefficient of $(z - \theta)^{-1}$ in $(\ref{taylorcompare})$ converges, as $N\rightarrow\infty$, to the coefficient of $(z - \theta)^{-1}$ in $(\ref{taylorcoeff})$ because of the initial assumption that $\{\lambda(N)\}_{N\geq 1}$ is a VK sequence with boundary point $\omega = (\alpha^{\pm}, \beta^{\pm}, \gamma^{\pm})$.
Heuristically, the coefficient of $(z - \theta)^{-2}$ in $(\ref{taylorcompare})$ converges to the coefficient of $(z - \theta)^{-2}$ in $(\ref{taylorcoeff})$ because of $(\ref{claimfatou})$ below, but one also needs to prove that the sums bundled inside `other terms' converge to $0$, as $N$ tends to infinity.
Such convergence follows from the bound $d(\lambda^+(N)) = O(\sqrt{N})$.
In fact, we prove such convergence for the larger of those sums, namely $\sum_{i=1}^{d(\lambda^+(N))}{\frac{ia_i^+(\lambda(N))}{N^2}}$. We can easily bound
\begin{equation*}
0 \leq \sum_{i=1}^{d(\lambda^+(N))}{\frac{ia_i^+(\lambda(N))}{N^2}} \leq \frac{d(\lambda^+(\lambda(N)))}{N^2} \sum_{i=1}^{d(\lambda^+(N))}{a_i^+(\lambda(N))} \leq \frac{d(\lambda^+(\lambda(N)))}{N^2}|\lambda^+(N)|,
\end{equation*}
and we know that $\lim_{N\rightarrow\infty}{\frac{|\lambda^+(N)|}{N}} < \infty$, while $d(\lambda^+(\lambda(N))) = O(\sqrt{N})$, so $\frac{d(\lambda^+(\lambda(N)))}{N^2}|\lambda^+(N)|$ converges to $0$, as $N\rightarrow\infty$, as we wished.

For higher powers of $(z - \theta)^{-k}$, $k > 2$, the analysis is the similar by using $d(\lambda^+(N)) = O(\sqrt{N})$ and $(\ref{claimfatou})$.
The convergence of Taylor series coefficients is enough for the (pointwise) convergence that we desire because of the dominated convergence theorem and estimates above.
The proof of Lemma $\ref{limitphi}$ will be finished once we prove $(\ref{claimfatou})$ below.

\begin{equation}\label{claimfatou}
\begin{gathered}
\lim_{N\rightarrow\infty}{\sum_{i=1}^{d(\lambda^{\pm}(N))}{ \left( \frac{a_i^{\pm}(\lambda(N))^k}{N^k} \right) }} = \sum_{i=1}^{\infty}{(\alpha_i^{\pm})^k}, \ \forall k \geq 2;\\
\lim_{N\rightarrow\infty}{\sum_{i=1}^{d(\lambda^{\pm}(N))}{ \left( \frac{b_i^{\pm}(\lambda(N))^k}{N^k} \right) }} = \sum_{i=1}^{\infty}{(\beta_i^{\pm})^k}, \ \forall k \geq 2.
\end{gathered}
\end{equation}

Let us prove only $\lim_{N\rightarrow\infty}{\sum_{i=1}^{d(\lambda^+(N))}{ \left( \frac{a_i^+(\lambda(N))^k}{N^k} \right) }} = \sum_{i=1}^{\infty}{(\alpha_i^+)^k}$, since the other three limit relations are very similar.

We know $\lim_{N\rightarrow\infty}{\frac{a_i^+(\lambda(N))^k}{N^k}} = (\alpha_i^+)^k$, for all $i = 1, 2, \ldots$, so Fatou's lemma yields
\begin{equation}\label{fatou1}
\liminf_{N\rightarrow\infty}{\sum_{i=1}^{d(\lambda^+(N))}{ \left( \frac{a_i^+(\lambda(N))^k }{N^k} \right) }} \geq \sum_{i=1}^{\infty}{ (\alpha_i^+)^k }
\end{equation}

On the other hand, let $1 > \epsilon > 0$ be an arbitrary real number. We know $\sum_{i=1}^{\infty}{\alpha_i^+} \leq \sum_{i=1}^{\infty}{(\alpha_i^+ + \beta_i^+)} < \infty$, so there exists $N_1\in\N$ such that $\alpha_{N_1}^+ < \epsilon$. Since $\lim_{N\rightarrow\infty}{\frac{a^+_{N_1}(\lambda(N))}{N}} = \alpha_{N_1}^+$, there exists $N_0\in\N$ such that $\frac{a^+_{N_1}(\lambda(N))}{N} < \epsilon$ for all $N\geq N_0$, and therefore $\frac{a^+_i(\lambda(N))}{N} < \epsilon$ for all $N \geq N_0$, $i\geq N_1$. It follows that $\frac{a_i^+(\lambda(N))}{N}\epsilon^{k-1} - (\frac{a_i^+(\lambda(N))}{N})^k > 0$, for almost all $i$, and $N$ large enough. Fatou's lemma then yields
\begin{equation}\label{fatou2}
\begin{gathered}
\sum_{i=1}^{\infty}{\left( \alpha_i^+ \epsilon^{k-1} - (\alpha_i^+)^k \right)} \leq
\liminf_{N\rightarrow\infty}{\sum_{i=1}^{d(\lambda^+(N))}{\left( \frac{a_i^+(\lambda(N))}{N}\epsilon^{k-1} - (\frac{a_i^+(\lambda(N))}{N})^k \right) } }\\
\leq \liminf_{N\rightarrow\infty}{ \sum_{i=1}^{d(\lambda^+(N))}{\left( \frac{|\lambda^+(N)|}{N}\epsilon^{k-1} - (\frac{a_i^+(\lambda(N))}{N})^k \right) } }\\
= \delta^+ \epsilon^{k-1} - \limsup_{N\rightarrow\infty}{ \sum_{i=1}^{d(\lambda^+(N))}{ (\frac{a_i^+(\lambda(N))}{N})^k } }\\
= (\sum_{i=1}^{\infty}{(\alpha_i^+ + \beta_i^+)} + \gamma^+)\epsilon^{k-1} - \limsup_{N\rightarrow\infty}{\sum_{i=1}^{d(\lambda^+(N))}{ \left( \frac{a_i^+(\lambda(N))^k}{N^k} \right) }}
\end{gathered}
\end{equation}
By combining $(\ref{fatou1})$ and $(\ref{fatou2})$, we obtain
\begin{equation*}
\sum_{i=1}^{\infty}{(\alpha_i^+)^k}
\leq \liminf_{N\rightarrow\infty}{\sum_i{\frac{a_i^+(\lambda(N))^k}{N^k}}}
\leq \limsup_{N\rightarrow\infty}{\sum_i{\frac{a_i^+(\lambda(N))^k}{N^k}}}
\leq \sum_{i=1}^{\infty}{(\alpha_i^+)^k} + \epsilon^{k-1} \left( \gamma^+ + \sum_{i=1}^{\infty}{\beta_i^+} \right).
\end{equation*}
Since $\epsilon \in (0, 1)$ is arbitrary, $\liminf{\sum_{i}{\frac{a_i^+(\lambda(N))^k}{N^k}}} = \limsup{\sum_{i}{\frac{a_i^+(\lambda(N))^k}{N^k}}} = \sum_{i=1}^{\infty}{(\alpha_i^+)^k}$. The three other statements in $(\ref{claimfatou})$ are proved similarly.
Lemma $\ref{limitphi}$ is now proved.
\end{proof}

\subsection{Saddle-point method}\label{saddleapp1}

In the previous subsection, we worked towards proving Proposition $\ref{toproveclaim2}$, and we were left with studying the asymptotics of the integral in the second line of $(\ref{toapplysaddle})$, as $N\rightarrow\infty$.
Recall that the finite contour $\gamma(\epsilon_0)$ is given by $[-\epsilon_0, \epsilon_0] \rightarrow \C$, $s \mapsto z_0 + se^{\ii \phi}$ and $\phi$ is such that $w''(z_0)e^{2\ii\phi} < 0$, i.e., $\gamma$ is in the \textit{critical direction} near $z_0$.
We divide the task by studying separately the contributions of a $N^{-\epsilon}$-neighborhood of $z_0$ (for some $\epsilon > 0$) and the contribution of $[z_0 - \epsilon_0e^{\ii\phi}, z_0 + \epsilon_0e^{\ii\phi}] \setminus B(z_0, N^{-\epsilon})$.

Fix any positive real number $\frac{1}{2} > \epsilon > \frac{1}{3}$.

We first consider the small $N^{-\epsilon}$-neighborhood of $z_0$: the integral we estimate first is
\begin{equation}\label{integralsaddle}
\left(1 + O\left(\frac{1}{N}\right)\right)\times\frac{N^{1/2}\theta^{\theta N - 1/2}}{(e^y - 1)^{\theta N - 1}e^{\theta N}\sqrt{2\pi}\ii}\times\int_{\gamma(N^{-\epsilon})}{\exp(Nw(z))\sqrt{\frac{z - \theta}{z}}\widetilde{\Psi}(z; \omega, \theta)dz},
\end{equation}
where $\gamma(N^{-\epsilon}) \myeq \gamma \cap \{z\in\C : |z - z_0| \leq N^{-\epsilon}\} = [z_0 - N^{-\epsilon}e^{\ii \phi}, z_0 + N^{-\epsilon}e^{\ii \phi}]$. We work this out by closely following the saddle-point method, as outlined in \cite{C}.

When $N$ large enough, then $N^{-\epsilon}$ is very small, and $w(z) = yz - z\ln{z} + (z-\theta)\ln{(z - \theta)} + \theta$ is an analytic function on $B(z_0, 2N^{-\epsilon})$, so we have a Taylor series expansion
\begin{equation*}
w(z) = w(z_0) + a_2(z - z_0)^2 + O(N^{-3\epsilon}), \hspace{.2in} a_2 = \frac{w''(z_0)}{2}.
\end{equation*}
It follows that
\begin{equation}\label{estimate1app1}
\exp(Nw(z)) = \exp\left( Nw(z_0) + Na_2(z - z_0)^2 \right) \times (1 + O(N^{1 - 3\epsilon})).
\end{equation}

By a similar reasoning, we obtain
\begin{equation}\label{estimate2app1}
\begin{gathered}
\sqrt{\frac{z - \theta}{z}}\widetilde{\Psi}(z; \omega, \theta) = \sqrt{\frac{z_0 - \theta}{z_0}} \widetilde{\Psi}(z_0; \omega, \theta) + O(N^{-\epsilon}),\\
= \sqrt{\frac{z_0 - \theta}{z_0}} \widetilde{\Psi}(z_0; \omega, \theta) + o(N^{1-3\epsilon}), \ |z - z_0| \leq N^{-\epsilon}, \textrm{ as }N\rightarrow\infty.
\end{gathered}
\end{equation}
The latter equality holds by the condition that $1/3 < \epsilon < 1/2$.

We can easily calculate $w''(z) = \frac{\theta}{z(z - \theta)}$, so $a_2 = \frac{w''(z_0)}{2} = \frac{(e^y - 1)^2}{2\theta e^y}$. Let $A > 0$ and $\alpha\in[0, 2\pi]$ be defined by
\begin{equation*}
a_2 = \frac{(e^y - 1)^2}{2\theta e^y} =: Ae^{\ii \alpha}.
\end{equation*}
A parametrization of the contour $\gamma(N^{-\epsilon})$ is $z = z_0 + re^{\ii \phi}$, $r\in [N^{-\epsilon}, N^{\epsilon}]$.
With respect to this parametrization, $Na_2 (z - z_0)^2 = NAr^2e^{(\alpha + 2\phi)\ii}$, and thus $\phi = \pm\frac{\pi}{2} - \frac{\alpha}{2}$, because the contour is in the critical direction and therefore $\phi = \frac{\pi}{2} - \frac{\alpha}{2}$ because the contour is also positively oriented.
It follows that $Na_2(z - z_0)^2 = -NAr^2$ and $\exp(Na_2(z - z_0)^2) = \exp(- ANr^2)$.

On the other hand, $\exp(N w(z_0)) = (e^y - 1)^{N\theta}e^{N\theta}\theta^{-N\theta}$, $\sqrt{\frac{z_0 - \theta}{z_0}} = e^{-y/2}$, and from the definitions of $\Psi$ and $\widetilde{\Psi}$, we have $\widetilde{\Psi}(z_0; \omega, \theta) = \widetilde{\Psi}\left(\frac{\theta}{1 - e^{-y}}; \omega, \theta\right) = \Psi(e^y; \omega, \theta)$. Therefore, by combining the estimates $(\ref{estimate1app1})$ and $(\ref{estimate2app1})$, we have
\begin{equation}\label{estimate3app1}
\begin{gathered}
e^{Nw(z)}\sqrt{\frac{z - \theta}{z}}\widetilde{\Psi}(z; \omega, \theta) = (e^y - 1)^{N\theta}e^{N\theta}\theta^{-N\theta}e^{-\frac{y}{2}}\Psi(e^y; \omega, \theta)\exp(-ANr^2)\times (1 + O(N^{1 - 3\epsilon}))
\end{gathered}
\end{equation}
for all $|z - z_0| \leq N^{-\epsilon}$, and in particular for all $z\in\gamma(N^{-\epsilon})$. With estimate $(\ref{estimate3app1})$, the integral formula $(\ref{integralsaddle})$ is asymptotically equal to
\begin{equation}\label{secondsaddle}
\begin{gathered}
\frac{(e^y - 1)N^{1/2}e^{-y/2}\theta^{-1/2}}{\sqrt{2\pi}\ii}\Psi(e^y; \omega, \theta)\times\int_{-N^{-\epsilon}}^{N^{-\epsilon}}{\exp\left(-ANr^2 + \frac{(\pi - \alpha)}{2}\ii\right)dr}\times\left( 1 + O(N^{1 - 3\epsilon}) \right)\\
= \frac{(e^y - 1)N^{1/2}e^{-y/2}\theta^{-1/2}}{\sqrt{2\pi}\ii}\Psi(e^y; \omega, \theta)\times\frac{e^{(\pi - \alpha)\ii /2}}{\sqrt{AN}}\times\int_{-\sqrt{AN^{1 - 2\epsilon}}}^{\sqrt{AN^{1 - 2\epsilon}}}{\exp(-u^2)du}\times\left( 1 + O(N^{1 - 3\epsilon}) \right),
\end{gathered}
\end{equation}
where for the latter we made the change of variables $u = r\sqrt{AN}$.

The estimate
\begin{equation*}
\int_S^{\infty}{e^{-u^2}du} \leq \int_S^{\infty}{e^{-u^2}\left( \frac{2u^2 + 1}{u^2} \right)du} = \left.\left( -\frac{e^{-u^2}}{u} \right)\right|_{S}^{\infty} = e^{-S^2}/S
\end{equation*}
applied to $S = \sqrt{AN^{1 - 2\epsilon}}$ shows $\int_{\sqrt{AN^{1 - 2\epsilon}}}^{\infty}{e^{-u^2}du} = O(N^{\epsilon - \frac{1}{2}}e^{-AN^{1 - 2\epsilon}}) = o(N^{1 - 3\epsilon})$. By an analogous reasoning, $\int_{-\infty}^{-\sqrt{AN^{1 - 2\epsilon}}}{e^{-u^2}du} = o(N^{1 - 3\epsilon})$ and so we can replace the integral $\int_{-\sqrt{AN^{1 - 2\epsilon}}}^{\sqrt{AN^{1 - 2\epsilon}}}$ by $\int_{-\infty}^{\infty}$ in $(\ref{secondsaddle})$, and we can evaluate $\int_{-\infty}^{\infty}{e^{-u^2}du} = \sqrt{\pi}$. Finally we observe that
\begin{equation*}
\frac{e^{(\pi - \alpha)\ii/2}}{\sqrt{AN}} = \left( -\frac{1}{ANe^{\alpha \ii}} \right)^{1/2} = \left( -\frac{1}{Na_2} \right)^{1/2} = \frac{\ii}{e^y - 1}\left( \frac{2\theta e^y}{N} \right)^{1/2}.
\end{equation*}
Therefore $(\ref{secondsaddle})$ equals
\begin{equation*}
\begin{gathered}
\frac{(e^y - 1)N^{1/2}e^{-y/2}\theta^{-1/2}}{\sqrt{2\pi}\ii}\Psi(e^y; \omega, \theta)\frac{\ii}{e^y - 1}\left( \frac{2\theta e^y}{N} \right)^{1/2}\sqrt{\pi}\times (1 + O(N^{1 - 3\epsilon}))\\
= \Psi(e^y; \omega, \theta)\times(1 + O(N^{1 - 3\epsilon})).
\end{gathered}
\end{equation*}
As a result, the contribution of the neighborhood $\gamma(N^{-\epsilon})$, is $\Psi(e^y; \omega, \theta)\times(1 + O(N^{1 - 3\epsilon})) = \Psi(e^y; \omega, \theta)\times(1 + o(1))$.

Finally, we show that the contribution of $\gamma(\epsilon_0) \setminus \gamma(N^{-\epsilon}) = [z_0 - \epsilon_0e^{\ii\phi}, z_0 + \epsilon_0e^{\ii\phi}] \setminus B(z_0, N^{-\epsilon})$ to the asymptotics of the second line of $(\ref{toapplysaddle})$ is of order $o(1)$, thus concluding the proof of the theorem. Indeed the contribution of $\gamma(\epsilon_0) \setminus \gamma(N^{-\epsilon})$ is
\begin{equation}\label{contributionfar}
\frac{N^{1/2}\theta^{\theta N - 1/2}}{(e^y - 1)^{\theta N - 1}e^{\theta N}\sqrt{2\pi}\ii}\exp(Nw(z_0))\int_{\gamma(\epsilon_0) \setminus \gamma(N^{-\epsilon})} {\exp(N(w(z) - w(z_0)))\sqrt{\frac{z - \theta}{z}}\widetilde{\Psi}(z; \omega, \theta)dz}
\end{equation}
Since $\exp(Nw(z_0)) = (e^y - 1)^{N\theta}e^{N\theta}\theta^{-N\theta}$, then the factor $\frac{N^{1/2}\theta^{\theta N - 1/2}}{(e^y - 1)^{\theta N - 1}e^{\theta N}\sqrt{2\pi}\ii}\exp(Nw(z_0))$ is of order $O(N^{1/2})$.
On the other hand, the factor $\sqrt{\frac{z - \theta}{z}}\widetilde{\Psi}(z; \omega, \theta)$ is holomorphic on the $\epsilon_0$-neighborhood of $z_0$ and also $\Re w(z)$ decreases as $z$ ranges from $z_0$ to $z_0 \pm \epsilon_0 e^{\ii \phi}$, by the definition of $\phi$.
It follows that the absolute value of the integral in $(\ref{contributionfar})$ is upper bounded by a constant times $\exp(N(\Re w(z_0 \pm N^{-\epsilon}e^{\ii\phi}) - \Re w(z_0) )) \leq const\times \exp( - const\cdot N^{1 - 2\epsilon}) = o(N^{-1/2})$.
Therefore the expression $(\ref{contributionfar})$ is of order $o(1)$, as $N$ tends to infinity, proving the contribution outside an $N^{-\epsilon}$-neighborhood of $z_0$ is zero. Hence Proposition $\ref{toproveclaim2}$ is finally proved.

The conclusion of the present section is summarized in the following corollary. It is a consequence of Proposition $\ref{toproveclaim2}$, the analogous statement for $V_{\delta}'  := \{v\in\C : -\delta < \Re v < 0, \ 0 < \Im v < 2\pi\}$ which was left to the reader, and also Lemma $\ref{claimconvergence}$.

\begin{cor}\label{cor:m1}
Theorem $\ref{thm:application1}$ holds for $m = 1$. In other words, if $\{\lambda(N)\}_{N \geq 1}$, $\lambda(N) \in \GT_N$, is a VK sequence of signatures with boundary point $\omega\in\Omega$, then there exists $\delta_0 > 0$ small enough such that
\begin{equation*}
\lim_{N\rightarrow\infty}{J_{\lambda(N)}(z; N, \theta)} = \Psi(z; \omega, \theta)
\end{equation*}
uniformly on $\TT_{\delta_0}$. The function $\Psi$ was defined in $(\ref{Psidef})$.
\end{cor}

\section{General number $m$ of variables: proof of Theorem $\ref{thm:application1}$}\label{sec:proofthm}

In this section, let $\{\lambda(N)\}_{N \geq 1}$ be a VK sequence of signatures with boundary point $\omega = (\alpha^{\pm}, \beta^{\pm}, \gamma^{\pm})$. Because of Section $\ref{sec:casem1}$, see Corollary $\ref{cor:m1}$, we have that Theorem $\ref{thm:application1}$ holds for $m = 1$, i.e., there exists some small $0 < \delta_0 < 1$ such that $\lim_{N\rightarrow\infty}{J_{\lambda(N)}(z; N, \theta)} = \Psi(z; \omega, \theta)$ holds uniformly on $\TT_{\delta_0}$. Let us fix such $\delta_0 \in (0, 1)$.

\subsection{Boundedness of Jack characters}

Let $m\in\N$, $m \geq 2$, be arbitrary.

\begin{lem}\label{lem:boundednessjacks}
Let $\delta \in (0, 1)$ be any constant such that
\begin{equation*}
1 - \delta_0 < \max\{(1 - \delta)^m, \ (1 + \delta)^{-m}\} < 1 < \min\{(1 + \delta)^m, \ (1 - \delta)^{-m}\} < 1 + \delta_0.
\end{equation*}
Then
\begin{equation*}
\sup_{N \geq m}{\sup_{(z_1, \dots, z_m)\in\TT^m_{\delta}}{|J_{\lambda(N)}(z_1, \ldots, z_m; N, \theta)|}} < \infty.
\end{equation*}
\end{lem}

\begin{proof}
{\bf Step 1.}
We begin with the following inequality: for any $k_1, \ldots, k_m\in\Z$ and positive reals $x_1, \ldots, x_m > 0$, we have
\begin{equation}\label{lemmaclaim1}
\begin{gathered}
x_1^{k_1}\cdots x_m^{k_m} \leq \max(x_1, 1/x_1)^{\max_i{|k_i|}}\cdots \max(x_m, 1/x_m)^{\max_i{|k_i|}}\\
\leq \sum_{\epsilon_1, \ldots, \epsilon_m \in\{\pm 1\}}{(x_1\cdots x_m)^{\max_i{|k_i|}}}
\leq \sum_{\epsilon_1, \ldots, \epsilon_m \in\{\pm 1\}} \sum_{j=1}^m{ (x_1\cdots x_m)^{k_j}}.
\end{gathered}
\end{equation}

{\bf Step 2.} From the fact that all branching coefficients $\psi_{\mu/\nu}(\theta)$ and the value $J_{\lambda(N)}(1^N; \theta)$ are nonnegative, see Theorems $\ref{evaluationjacks}$, $\ref{branchingjacks}$, and the triangle inequality we have
\begin{equation*}
\left| J_{\lambda(N)}(z_1, \ldots, z_m; N, \theta) \right| \leq J_{\lambda(N)}(|z_1|, \ldots, |z_m|; N, \theta).
\end{equation*}
Consequently the lemma is reduced to proving the bound
\begin{equation*}
\sup_{N \geq m}\sup_{(x_1, \ldots, x_m)\in I_{\delta}^m}{J_{\lambda(N)}(x_1, \ldots, x_m; N, \theta)} < \infty,
\end{equation*}
where $I_{\delta} \myeq [1 - \delta, 1 + \delta]\subset\R$.\\

{\bf Step 3.} Write
\begin{equation}\label{eqn:expansionJacks}
J_{\lambda(N)}(x_1, \ldots, x_m; N, \theta) = \sum_{k = (k_1, \ldots, k_m)\in\Z^m}{c_{\lambda(N), k}(\theta)\cdot x_1^{k_1}\cdots x_m^{k_m}},
\end{equation}
so each $c_{\lambda(N), k} \geq 0$ because of Theorems $\ref{evaluationjacks}$, $\ref{branchingjacks}$. Note that the coefficients $c_{\lambda(N), k}$ do not depend on $x_1, \ldots, x_m$, in particular we can set $x_1 = \ldots = x_{j-1} = x_{j+1} = \ldots = x_m = 1$ and get the Jack character of one variable
\begin{equation*}
J_{\lambda(N)}(x_j; N, \theta) = \sum_{k = (k_1, \ldots, k_m)\in\Z^m}{c_{\lambda(N), k}(\theta)\cdot x_j^{k_j}},
\end{equation*}
for any $1\leq j\leq m$. Next we apply $(\ref{lemmaclaim1})$ to the right side of $(\ref{eqn:expansionJacks})$: for any $x_1, x_2, \ldots, x_m\in I_{\delta}$, we obtain
\begin{equation*}
\begin{gathered}
J_{\lambda(N)}(x_1, \ldots, x_m; N, \theta) \leq \sum_{\epsilon_1, \ldots, \epsilon_m\in\{\pm 1\}}{ \sum_{j=1}^m{ \sum_{k = (k_1, \ldots, k_m)\in\Z^m}{c_{\lambda(N), k}(\theta)\cdot (x_1^{\epsilon_1}\cdots x_m^{\epsilon_m})^{k_j}} } }\\
= \sum_{\epsilon_1, \ldots, \epsilon_m\in\{\pm 1\}}{ \sum_{j=1}^m{ J_{\lambda(N)}(x_1^{\epsilon_1}\cdots x_m^{\epsilon_m}; N, \theta) } }\\
= m \cdot \sum_{\epsilon_1, \ldots, \epsilon_m\in\{\pm 1\}}{J_{\lambda(N)}(x_1^{\epsilon_1}x_2^{\epsilon_2}\cdots x_m^{\epsilon_m}; N, \theta)}.
\end{gathered}
\end{equation*}
Observe that if $x_1, \ldots, x_m\in I_{\delta}$ and $\epsilon_1, \ldots, \epsilon_m\in\{\pm 1\}$, then
\begin{equation}\label{smallbound}
1 - \delta_0 < \min\{(1-\delta)^m, (1 + \delta)^{-m}\} < x_1^{\epsilon}\cdots x_m^{\epsilon_m} < \max\{(1 + \delta)^m, (1 - \delta)^{-m}\} < 1 + \delta_0.
\end{equation}
Let $\delta(min) := \min\{(1-\delta)^m, (1 + \delta)^{-m}\}$, $\delta(max) := \max\{(1 + \delta)^m, (1 - \delta)^{-m}\}$, so $[\delta(min), \delta(max)] \subset (1 - \delta_0, 1 + \delta_0)$. Combining $(\ref{smallbound})$ with the estimates before gives
\begin{equation*}
\sup_{N \geq m}\sup_{(x_1, \ldots, x_m)\in I_{\delta}}{J_{\lambda(N)}(x_1, \ldots, x_m; N, \theta)} \leq m \cdot 2^m\cdot\sup_{N \geq m}\sup_{x\in [\delta(min), \delta(max)]}{J_{\lambda(N)}(x; N, \theta)}
\end{equation*}
and the latter is finite because, by choice of $\delta_0$, the limit $\lim_{N\rightarrow\infty}{J_{\lambda(N)}(x; N, \theta)} = \Psi(x; \omega, \theta)$ holds uniformly on $\TT_{\delta_0} \supset [\delta(min), \delta(max)]$.
\end{proof}

\begin{cor}\label{cor:subsequential}
Let $\delta > 0$ be any real number such that
\begin{equation*}
1 - \delta_0 < \max\{(1 - \delta)^m, \ (1 + \delta)^{-m}\} < 1 < \min\{(1 + \delta)^m, \ (1 - \delta)^{-m}\} < 1 + \delta_0.
\end{equation*}
For any sequence $\{N_j\}_{j\geq 1} \subset \{m, m+1, m+2, \dots\}$, there exists a subsequence $\{N_{k(n)}\}_{n \geq 1} \subseteq \{N_j\}_{j\geq 1}$ such that
\begin{equation*}
\left\{J_{\lambda(N_{k(n)})}(z_1, \ldots, z_m; N_{k(n)}, \theta)\right\}_{n\geq 1}
\end{equation*}
converges uniformly on $\TT_{\delta}^m$, as $n$ goes to infinity, to a holomorphic function on $\TT_{\delta}^m$.
\end{cor}
\begin{proof}
From Lemma $\ref{lem:boundednessjacks}$, it follows that the sequence $\left\{ J_{\lambda(N)}(z_1, \ldots, z_m; N, \theta) : N \geq m \right\}$ is uniformly bounded on $\TT_{\delta}^m$.
The result then follows as an application of Montel's theorem.
\end{proof}

\subsection{Pointwise convergence via Pieri integral formula}\label{sec:convergencepieri}

For each $m\in\N$, choose any real number $\delta_m \in (0, 1)$ such that 
\begin{equation*}
1 - \delta_0 < \max\{(1 - \delta_m)^m, \ (1 + \delta_m)^{-m}\} < 1 < \min\{(1 + \delta_m)^m, \ (1 - \delta_m)^{-m}\} < 1 + \delta_0.
\end{equation*}
Let us also make the choice of $\delta_1, \delta_2, \delta_3, \dots$ be such that
\begin{equation*}
\delta_0 > \delta_1 > \delta_2 > \dots > 0.
\end{equation*}
The choice of $\delta_m$ lets us use Lemma $\ref{lem:boundednessjacks}$ above. The next statement and its proof are key to our approach for the study of asymptotics of Jack characters. We believe the approach is robust enough to carry over to other limit regimes.

\begin{prop}\label{prop:pointwiseconvergence}
Let $x_1, x_2, \ldots, x_m$ be real numbers which satisfy
\begin{equation*}
\begin{gathered}
1 > x_m > \ldots > x_2 > x_1 > (1 - \delta_m)^{1/2},\\
x_3 > x_1/x_2, \ x_4 > x_2/x_3,\cdots,\ x_m > x_{m-2} / x_{m-1}.
\end{gathered}
\end{equation*}
Then the following pointwise limit holds
\begin{equation*}
\lim_{N\rightarrow\infty}{J_{\lambda(N)}(x_1, \ldots, x_m; N, \theta)} = \prod_{i = 1}^m{\Psi(x_i; \omega, \theta)}.
\end{equation*}
\end{prop}

\begin{rem}\label{rem:applyPieri}
The inequalities in Proposition $\ref{prop:pointwiseconvergence}$ imply, for any $j = 3, 4, \dots, m$, that
\begin{equation*}
x_j > \max_{i = 1, \dots, j-2}{(x_i/x_{i+1})}.
\end{equation*}
Also for any $j = 1, 2, \dots, m$, our choice of the constants $\delta_1, \dots, \delta_m$ implies
\begin{equation*}
1 > x_j > \dots > x_1 > (1 - \delta_j)^{1/2}.
\end{equation*}
\end{rem}

\begin{rem}\label{rem:deltaremark}
Let $\delta \in (0, 1)$ be any real number such that
\begin{equation}\label{eqn:conditiondelta}
1 - \delta_0 < \max\{(1 - \delta)^m, \ (1 + \delta)^{-m}\} < 1 < \min\{(1 + \delta)^m, \ (1 - \delta)^{-m}\} < 1 + \delta_0.
\end{equation}
Also let $m\in\N$ be arbitrary. Then there exist real numbers $\delta_1, \dots, \delta_{m-1}, \delta_m = \delta$ satisfying the conditions at the beginning of Section $\ref{sec:convergencepieri}$. This implies that the conclusion of Proposition $\ref{prop:pointwiseconvergence}$ is satisfied if we replace $\delta_m$ by any $\delta \in (0, 1)$ satisfying $(\ref{eqn:conditiondelta})$ above. The reason for our wording of Proposition $\ref{prop:pointwiseconvergence}$ is so that the proof, which uses induction on $m$, is more transparent.
\end{rem}

\begin{proof}
The proof is by induction on $m$. The base case $m = 1$ follows from the much stronger analysis of Section $\ref{sec:casem1}$, which proves Theorem $\ref{thm:application1}$ for $m = 1$.

For the induction step, let us assume the statement holds for $m\in\N$ and let us next prove it for $(m+1)$. Let $x_1, x_2, \ldots, x_m, x$ be any real numbers satisfying
\begin{equation}\label{eqns:inequalities}
\begin{gathered}
1 > x > x_m > \ldots > x_2 > x_1 > (1 - \delta_{m+1})^{1/2},\\
x_3 > x_1/x_2,\ \cdots,\ x_m > x_{m-2}/x_{m-1},\ x > x_{m-1} / x_m.
\end{gathered}
\end{equation}
From the inductive hypothesis, the following pointwise limits hold
\begin{equation*}
\lim_{N\rightarrow\infty}{J_{\lambda(N)}(x_1, \ldots, x_m; N, \theta)} = \prod_{i = 1}^m{\Psi(x_i; \omega, \theta)}, \hspace{.2in}
\lim_{N\rightarrow\infty}{J_{\lambda(N)}(x; N, \theta)} = \Psi(x; \omega, \theta).
\end{equation*}
Next we apply the Pieri integral formula, Theorem $\ref{thm:pieri}$. We can do so because of our assumptions on the variables $x_1, \dots, x_m, x$, see also Remark $\ref{rem:applyPieri}$. The result is
\begin{equation}\label{pieribeforelimit}
\begin{gathered}
J_{\lambda(N)}(x_1, \ldots, x_m; N, \theta)J_{\lambda(N)}(x; N, \theta) = \frac{\Gamma(N\theta)x^{m\theta}}{\Gamma((N-m)\theta)\Gamma(\theta)^m(1-x)^{m\theta}\prod_{i=1}^m{(1-x_i)^{\theta}}}\\
\times\int\dots\int { G(x_1, \ldots, x_m, x; w_1, \ldots, w_m; \theta)\left( F(x_1, \ldots, x_m, x; w_1, \ldots, w_m) \right)^{\theta(N-m)-1}\prod_{i=1}^m{(1 - w_i)^{\theta-1}} }\\
{ J_{\lambda(N)}\left(x_1w_1, \ldots, x_mw_m, \frac{x}{w_1\cdots w_m}; N, \theta\right) \frac{dw_1}{w_1}\dots\frac{dw_m}{w_m} },
\end{gathered}
\end{equation}
where the functions $F, G$ are defined in $(\ref{defn:FGfns})$ and the domain of integration is the compact subset $\U_x = \{(w_1, \ldots, w_m) \in \R^m : 1 \geq w_1, w_2, \dots, w_m \geq 0, \ w_1w_2\cdots w_m \geq  x\}$.

Let us study the asymptotics of $(\ref{pieribeforelimit})$ as $N \rightarrow \infty$. The left side has a limit, as $N\rightarrow\infty$, because of the inductive hypothesis and that limit is $\Psi(x; \omega, \theta)\prod_{i=1}^m{\Psi(x_i; \omega, \theta)}$.

Let us move on to the right side of $(\ref{pieribeforelimit})$. For the first line, we use $\Gamma(x+a)/\Gamma(x+b) = x^{a-b}(1 + O(x^{-1}))$, $a, b \in \C$, $x\rightarrow\infty$. It easily follows that
\begin{equation}\label{eqn:firstline}
\textrm{first line of right side of }(\ref{pieribeforelimit}) = \frac{(N\theta)^{m \theta} x^{m \theta}}{\Gamma(\theta)^m(1 - x)^{m\theta}\prod_{i=1}^m{(1 - x_i)^{\theta}}}\cdot \left( 1 + O(N^{-1}) \right), \ N\rightarrow\infty.
\end{equation}
Now let us study the integral by dividing it into a union of two domains of integration. Let $1 > \epsilon > \frac{1}{2}$ be arbitrary and consider $N$ large enough so that $(1 - N^{-\epsilon})^m > x$. Then consider the small subdomain $\U(N, \epsilon) \subset \U_x$ around $(w_1, \ldots, w_m) = (1^m)$:
\begin{equation*}
\U(N, \epsilon) \myeq \{ (w_1, \ldots, w_m)\in\R^m : 1 \geq w_1, w_2, \dots, w_m \geq 1 - N^{-\epsilon} \}.
\end{equation*}
Let $H_N(x_1, \ldots, x_m, x; w_1, \ldots, w_m; \theta)$ be the integrand shown in the second and third lines of $(\ref{pieribeforelimit})$.
(We include $(w_1 \cdots w_m)^{-1}$, which appears in the third line of $(\ref{pieribeforelimit})$ in the form $\frac{dw_1}{w_1}\dots\frac{dw_m}{w_m}$, as a factor of $H_N$.)
We make the following two claims.

\begin{equation*}
\begin{gathered}
\textbf{Claim 1.} \hspace{.1in} \int\cdots\int{|H_N(x_1, \ldots, x_m, x; w_1, \ldots, w_m; \theta)|dw_1 \cdots dw_m} = o(N^{-m \theta}),\\ \textrm{ where the domain of integration is }\U_x \setminus \U(N, \epsilon).\\
\textbf{Claim 2.} \hspace{.1in} \int\cdots\int{H_N(x_1, \ldots, x_m, x; w_1, \ldots, w_m; \theta) dw_1 \cdots dw_m}\\
= (Nx\theta)^{-m \theta}\Gamma(\theta)^m(1 - x)^{m\theta}\prod_{i=1}^m{(1 - x_i)^{\theta}} \left( J_{\lambda(N)}(x_1, \ldots, x_m, x; N, \theta) + O(N^{-\epsilon}) \right) \cdot (1 + O(N^{1 - 2\epsilon})),\\
\textrm{ where the domain of integration is }\U(N, \epsilon).
\end{gathered}
\end{equation*}

Let us complete the inductive step from claims 1 and 2, and after that we give the proofs to the claims. Indeed, from both claims above and $(\ref{eqn:firstline})$, the right side of $(\ref{pieribeforelimit})$ equals
\begin{equation}\label{eqn:rightsidejacks}
\begin{gathered}
\left( J_{\lambda(N)}(x_1, \ldots, x_m, x; N, \theta) + O(N^{-\epsilon}) \right)\cdot (1 + O(N^{1 - 2\epsilon})) + o(1)\\
= J_{\lambda(N)}(x_1, \ldots, x_m, x; N, \theta) + J_{\lambda(N)}(x_1, \ldots, x_m, x; N, \theta)\cdot O(N^{1 - 2\epsilon}) + o(1), \ N \rightarrow\infty.
\end{gathered}
\end{equation}
From Lemma $\ref{lem:boundednessjacks}$, we have the bound
\begin{equation*}
\sup_{N \geq m+1}{|J_{\lambda(N)}(x_1, \ldots, x_m, x; N, \theta)|} \leq \sup_{N\geq m+1}\sup_{(z_1, \ldots, z_m, z)\in\TT_{\delta_0}^{m+1}}{\left| J_{\lambda(N)}(z_1, \ldots, z_m, z; N, \theta) \right|} < \infty,
\end{equation*}
and in particular $J_{\lambda(N)}(x_1, \ldots, x_m, x; N, \theta) \cdot O(N^{1 - 2\epsilon}) = o(1)$ as $N$ tends to infinity. 
Then, by $(\ref{eqn:rightsidejacks})$, the right side of $(\ref{pieribeforelimit})$ equals $J_{\lambda(N)}(x_1, \ldots, x_m, x; N, \theta) + o(1)$, as $N\rightarrow\infty$.

Since the left side of $(\ref{pieribeforelimit})$ converges to $\Psi(x; \omega, \theta)\prod_{i=1}^m{\Psi(x_i; \omega, \theta)}$ as $N$ goes to infinity, then the same must be true for the right side of $(\ref{pieribeforelimit})$. Therefore $\lim_{N\rightarrow\infty}{J_{\lambda(N)}(x_1, \ldots, x_m, x; N, \theta)}$ exists and equals $\Psi(x; \omega, \theta)\prod_{i=1}^m{\Psi(x_i; \omega, \theta)}$.

We finish with the proofs of the claims above. For simplicity, denote the vectors $\bfx = (x_1, \ldots, x_m, x)$, $\bfw = (w_1, \ldots, w_m)$, and write $\mathring{F}(\bfw) \myeq F(\bfx, \bfw)$, $\mathring{G}(\bfw) \myeq G(\bfx, \bfw; \theta)$, thus focusing on the variable $\bfw$ and treating $\bfx$ as a sequence of constant values.\\

\textit{Proof of Claim 1.}

We can simplify the claim quite a bit. In fact, observe that the function $\mathring{G}(\bfw) = G(\bfx, \bfw; \theta)$ defined in $(\ref{defn:FGfns})$ does not depend on $N$ and attains a maximum positive and finite value as $\bfw$ ranges over points of the compact set $\U_x$.

Also note that for any $\bfw\in\U_x$, we have $1 > x_i \geq x_iw_i \geq x_ix > x_1^2 > 1 - \delta_{m+1}$ and $1 \geq x(w_1\cdots w_m)^{-1} \geq x > (1 - \delta_{m+1})^{1/2} > 1 - \delta_{m+1}$. Thus, together with Lemma $\ref{lem:boundednessjacks}$, we obtain
\begin{equation*}
\begin{gathered}
\sup_{N \geq m+1}\sup_{\bfw\in\U_x}{\left|J_{\lambda(N)}\left( x_1w_1, \dots, x_mw_m, \frac{x}{w_1\cdots w_m}; N, \theta \right)\right|}\\
\leq \sup_{N \geq m+1}{\sup_{(z_1, \dots, z_{m+1})\in\mathbb{D}_m}{|J_{\lambda(N)}(z_1, \dots, z_{m+1}; N, \theta)|}} < \infty,
\end{gathered}
\end{equation*}
where we denoted $\mathbb{D}_m := \TT_{\delta_{m+1}}^{m+1}$.

Taking these estimates into account, the first claim will be proved if we show
\begin{equation}\label{eqn:lefttoprove}
\int\dots\int{|F(\bfx; \bfw)|^{\theta (N-m)-1}\prod_{i=1}^m{|1 - w_i|^{\theta - 1}}dw_1\dots dw_m} = o(N^{-m\theta}),
\end{equation}
where the domain of integration is $\U_x \setminus \U(N, \epsilon)$.

It is not hard to see that $F(\bfx; \bfw)$ is increasing in each of the variables $w_i$. In fact, one can easily see that $F(\bfx; \bfw) \geq 0, \ \forall \bfw\in\U_x$, and $\left(\partial_{w_i}F(\bfx; \bfw)\right) (F(\bfx; \bfw))^{-1} \geq 0, \ \forall \bfw\in\U_x, \ \forall i = 1, \dots, m$. Note that $\bfw\in \U_x \setminus \U(N, \epsilon)$ implies $w_i < 1 - N^{-\epsilon}$ for some $i = 1, 2, \ldots, m$. In the case that $w_i < 1 - N^{-\epsilon}$, we have
\begin{equation*}
\begin{gathered}
|F(\bfx; \bfw)| = F(\bfx; \bfw) \leq F\left(\bfx; (1, \dots, 1, 1 - N^{-\epsilon}, 1 \dots, 1)\right)\\
= (1 - x)^{-1}(1 - x_i)^{-1}(1 - x_iw_i)\left.\left(1 - \frac{x}{w_i} \right)\right|_{w_i = 1 - N^{-\epsilon}}\\
= ((1 - x_i)(1 - x))^{-1}(1 + x_ix  - x_i (1 - N^{-\epsilon}) - x(1 - N^{-\epsilon})^{-1})\\
\leq ((1 - x_i)(1 - x))^{-1}(1 + x_ix  - x_i (1 - N^{-\epsilon}) - x(1 + N^{-\epsilon})) = 1 - N^{-\epsilon}\frac{x - x_i}{(1 - x_i)(1 - x)}.
\end{gathered}
\end{equation*}

Thus the estimate above implies that $(\ref{eqn:lefttoprove})$ is implied if we prove
\begin{equation}\label{eqn:lefttoprove2}
\begin{gathered}
\left( 1 - N^{-\epsilon}\frac{x - x_i}{(1 - x_i)(1 - x)} \right)^{\theta (N - m) - 1} = o(N^{-m\theta}),\\
\int_x^1{(1 - w)^{\theta - 1}dw} = O(1).
\end{gathered}
\end{equation}

The second estimate in $(\ref{eqn:lefttoprove2})$ is quite clear and it is implied by the fact that $\int_0^1{x^{\theta - 1}dx}$ is a convergent integral, for any $\theta\in\C$ with $\Re\theta > 0$.

It remains to prove the first estimate of $(\ref{eqn:lefttoprove2})$. Begin with the Taylor series expansion
\begin{equation*}
\ln{\left( 1 - N^{-\epsilon}\frac{x - x_i}{(1 - x_i)(1 - x)} \right)} = - N^{-\epsilon}\frac{x - x_i}{(1 - x_i)(1 - x)} + O(N^{-2\epsilon}) = -N^{-\epsilon}c + O(N^{-2\epsilon}),
\end{equation*}
where we denoted $c = \frac{x - x_i}{(1 - x_i)(1 - x)}$ a positive constant independent of $N$. Then

\begin{equation}\label{lefttoprove2done}
\begin{gathered}
\left( 1 - N^{-\epsilon}\frac{x - x_i}{(1 - x_i)(1 - x)} \right)^{\theta (N - m) - 1}
= \exp\left((\theta (N - m) - 1) (- N^{-\epsilon}c + O(N^{-2\epsilon})) \right)\\
= \exp\left( -\theta c N^{1 - \epsilon} + O(N^{-\epsilon}) + O(N^{1 - 2\epsilon}) \right) = \exp\left(-\theta c N^{1-\epsilon} \right) \times (1 + o(1)),
\end{gathered}
\end{equation}
where we used that $1 > \epsilon > \frac{1}{2}$ implies $O(N^{-\epsilon}) + O(N^{1 - 2\epsilon}) = o(1)$. The estimate $(\ref{lefttoprove2done})$ shows the first line of $(\ref{eqn:lefttoprove2})$, in fact, we achieved a much stronger bound.\\

\textit{Proof of Claim 2.}

Let us do an approximation of the factor $\mathring{F}(\bfw)^{\theta N}$. A simple application of Taylor's series expansion around $\bfw = (1^m)$ gives
\begin{equation}\label{step1:claim2}
\ln\left(\mathring{F}(\bfw)\right) = \ln{\left(\mathring{F}(1^m)\right)} + \sum_{i=1}^m{\left.\frac{\partial_{w_i}\mathring{F}(\bfw)}{\mathring{F}(\bfw)}\right|_{\bfw = 1^m}(w_i - 1)} + O\left(\sum_{1\leq i,j \leq m}{|(w_i - 1)(w_j - 1)|} \right).
\end{equation}
From the formula $(\ref{defn:FGfns})$ of $F$, it is clear that
\begin{equation*}
\mathring{F}(1^m) = 1, \hspace{.1in} \frac{\partial_{w_i}\mathring{F}(\bfw)}{\mathring{F}(\bfw)} = -\frac{x_i}{1 - x_iw_i} + \frac{x}{w_i (w_1\cdots w_m - x)} \Rightarrow \left.\frac{\partial_{w_i}\mathring{F}(\bfw)}{\mathring{F}(\bfw)}\right|_{\bfw = 1^m} = \frac{x - x_i}{(1 - x_i)(1 - x)}.
\end{equation*}
Note also that $\sup_{\bfw \in \U(N, \epsilon)}{N|(w_i - 1)(w_j - 1)|} = O(N^{1 - 2\epsilon})$.
If we parametrize points in $\U(N, \epsilon)$ by
\begin{equation}\label{eqn:bfyw}
w_i = 1 - \frac{(1 - x_i)(1 - x)}{\theta N(x - x_i)} y_i, \ \textrm{ with } 0\leq y_i \leq \frac{\theta (x - x_i)N^{1-\epsilon}}{(1 - x_i)(1 - x)}, \textrm{ for all } i = 1, 2, \dots, m,
\end{equation}
we can use $(\ref{step1:claim2})$ to deduce
\begin{equation}\label{eqn:firstestimate}
\mathring{F}(\bfw)^{\theta N} = \exp\left( \theta N \ln{\mathring{F}(\bfw)} \right) =
\exp\left( -\sum_{i=1}^m{y_i} + O(N^{1 - 2\epsilon})\right)
= \exp\left( -\sum_{i=1}^m{y_i}\right) \times (1 + O(N^{1 - 2\epsilon})),
\end{equation}
where $O(N^{1 - 2\epsilon})$ is uniform for $\bfw = (w_1, \ldots, w_m) \in \U(N, \epsilon)$, and $\bfy = (y_1, \ldots, y_m)$ is defined from $\bfw$ as in $(\ref{eqn:bfyw})$. By similar Taylor expansions, we obtain
\begin{equation}\label{eqn:secondestimate}
\begin{gathered}
\mathring{G}(\bfw) = \mathring{G}(1^m) + O( N^{-\epsilon} )
= x^{-m\theta}\prod_{i=1}^m{(x - x_i)^{\theta}} + O(N^{-\epsilon}) = x^{-m\theta}\prod_{i=1}^m{(x - x_i)^{\theta}}  (1 + O(N^{-\epsilon})),\\
\frac{\mathring{F}(\bfw)^{-\theta m - 1}}{w_1\cdots w_m} = 1 + O(N^{-\epsilon}),
\end{gathered}
\end{equation}
where both bounds $O(N^{-\epsilon})$ are uniform for $\bfw \in \U(N, \epsilon)$. From the parametrization $(\ref{eqn:bfyw})$ of $\U(N, \epsilon)$, we have
\begin{equation}\label{eqn:thirdestimate}
\begin{gathered}
\prod_{i=1}^m{(1 - w_i)^{\theta - 1}} = (\theta N)^{m(1 - \theta)}\cdot\prod_{i=1}^m{ \frac{(1 - x_i)^{\theta - 1}(1 - x)^{\theta - 1}}{(x - x_i)^{\theta - 1}}y_i^{\theta - 1} },\\
\prod_{i=1}^m{dw_i} = (-1)^m (\theta N)^{-m}\cdot\prod_{i=1}^m{\frac{(1 - x_i)(1 - x)}{x - x_i}}\prod_{i=1}^m{dy_i}.
\end{gathered}
\end{equation}

Lastly we need to estimate the Jack polynomial term of the integrand. Let us use

\begin{equation}\label{eqn:firstboundjacks}
\begin{gathered}
J_{\lambda(N)}\left(x_1w_1, \dots,  x_mw_m, \frac{x}{w_1\cdots w_m}; N, \theta\right) = J_{\lambda(N)}(x_1, \ldots, x_m, x; N, \theta)\\
+ \sum_{n=1}^m{R_n(\bfx, \bfw; \lambda(N), \theta)(w_n - 1)},\\
R_n(\bfx, \bfw; \lambda(N), \theta) := \int_0^1{\frac{\partial}{\partial w_n}J_{\lambda(N)}\left(\dots, x_i(1 + t(w_i - 1)), \dots , \frac{x}{\prod_{i=1}^m{(1 + t(w_i - 1))}}; N, \theta\right)dt}.
\end{gathered}
\end{equation}
Observe that for any $t\in [0, 1]$, $\bfw\in\U_x$, we have the inequalities
\begin{equation*}
\begin{gathered}
1 > x_i \geq x_i(1 + t(w_i - 1))  \geq x_iw_i \geq x_ix > x_1^2 > 1 - \delta_{m+1},\ \forall i = 1, \dots, m,\\
1 \geq \frac{x}{w_1\cdots w_m} \geq \frac{x}{\prod_{i=1}^m{(1 + t(w_i - 1))}} \geq x > (1 - \delta_{m+1})^{1/2} > 1 - \delta_{m+1}.
\end{gathered}
\end{equation*}
By Lemma $\ref{lem:boundednessjacks}$ and our choice of $\delta_{m+1}$, we have that $\{J_{\lambda(N)}(z_1, \dots, z_m, z_{m+1}; N, \theta) : N \geq m+1\}$ is uniformly bounded on $(z_1, \dots, z_{m+1})\in\TT^{m+1}_{\delta_{m+1}}$. Thus the sequences of partial derivatives $\{\frac{\partial}{\partial z_n}J_{\lambda(N)}(z_1, \dots, z_m, z_{m+1}; N, \theta) : N \geq m+1\}$, $1 \leq n \leq m+1$, are also uniformly bounded on $\TT^{m+1}_{\delta_{m+1}}$. Consequently, from the definition of the functions $R_n(\bfx; \lambda(N), \theta)$ above, we obtain
\begin{equation*}
\sup_{N \geq m+1}{|R_n(\bfx, \bfw; \lambda(N), \theta)|} = O(1).
\end{equation*}
From $(\ref{eqn:firstboundjacks})$, we therefore have the bound
\begin{equation}\label{eqn:fourthestimate}
J_{\lambda(N)}\left(x_1w_1, \cdots, x_mw_m, \frac{x}{w_1\cdots w_m}; N, \theta\right) = J_{\lambda(N)}(x_1, \ldots, x_m, x; N, \theta) + O(N^{-\epsilon}),
\end{equation}
uniformly for $\bfw \in \U(N, \epsilon)$.

By combining the estimates $(\ref{eqn:firstestimate})$, $(\ref{eqn:secondestimate})$, $(\ref{eqn:thirdestimate})$ and $(\ref{eqn:fourthestimate})$, and the fact that $1 > \epsilon > \frac{1}{2}$ implies $0 > 1 - 2\epsilon > -\epsilon$ and $N^{1 - 2\epsilon} > N^{-\epsilon}$, then
\begin{equation}\label{lastestimate}
\begin{gathered}
\int\cdots\int{H_N(x_1, \ldots, x_m, x; w_1, \ldots, w_m; \theta) dw_1 \cdots dw_m} = (Nx\theta)^{-m \theta}(1 - x)^{m\theta}\prod_{i=1}^m{(1 - x_i)^{\theta}}\\
\times \left(J_{\lambda(N)}(x_1, \ldots, x_m, x; N, \theta) + O(N^{-\epsilon}) \right)\prod_{i=1}^m{\int_0^{N^{1 - \epsilon}\frac{\theta(x-x_i)}{((1 - x_i)(1 - x))}}{e^{-y_i}y_i^{\theta - 1}dy_i}}\cdot (1 + O(N^{1 - 2\epsilon}))\\
= (Nx\theta)^{-m \theta}(1 - x)^{m\theta}\prod_{i=1}^m{(1 - x_i)^{\theta}}(J_{\lambda(N)}(x_1, \ldots, x_m, x; N, \theta) + O(N^{-\epsilon})) \cdot \Gamma(\theta)^m\cdot (1 + O(N^{1 - 2\epsilon})),
\end{gathered}
\end{equation}
in the equality from the second to last line, we used that for any constant $c>0$ and $N$ large enough, we have $\int_{N^{1-\epsilon}c}^{\infty}{e^{-y}y^{\theta - 1}dy} \leq \int_{N^{1-\epsilon}c}^{\infty}{e^{-y/2}dy} = 2e^{-N^{1-\epsilon}c/2} = O(N^{1 - 2\epsilon})$ and so $\int_0^{N^{1-\epsilon}c}{e^{-y}y^{\theta - 1}dy} = \int_0^{\infty}{e^{-y}y^{\theta - 1}dy} + O(N^{1-2\epsilon}) = \Gamma(\theta) + O(N^{1 - 2\epsilon}) = \Gamma(\theta) (1 + O(N^{1-2\epsilon}))$. Thus we have proved the second claim.

We make one final observation. The factor $(-1)^m$ in $(\ref{eqn:thirdestimate})$ was removed in $(\ref{lastestimate})$ because $y_i$ goes from $N^{1 - \epsilon}\frac{\theta (x - x_i)}{(1 - x_i)(1 - x)}$ to $0$ (instead of going from $0$ to $N^{1 - \epsilon}\frac{\theta (x - x_i)}{(1 - x_i)(1 - x)}$, as we wrote in the latter integral) in the parametrization $(\ref{eqn:bfyw})$.
\end{proof}

\subsection{Proof of Theorem $\ref{thm:application1}$}

The work is already done, we simply put the pieces together. Let $m\in\N$ be arbitrary. Let $\delta > 0$ be small enough so that
\begin{equation*}
1 - \delta_0 < \max\{(1 - \delta)^m, \ (1 + \delta)^{-m}\} < 1 < \min\{(1 + \delta)^m, \ (1 - \delta)^{-m}\} < 1 + \delta_0.
\end{equation*}
We want to prove the limit
\begin{equation}\label{eqn:finaldesiredthm}
\lim_{N\rightarrow\infty}{J_{\lambda(N)}(z_1, \ldots, z_m; N, \theta)} = \prod_{i=1}^m{\Psi(z_i; \omega, \theta)}
\end{equation}
holds uniformly on $\TT_{\delta}^m$. By Corollary $\ref{cor:subsequential}$, any sequence $\{N_j\}_{j \geq 1}$ of positive integers has a subsequence $\{N_{k(n)}\}_{n\geq 1}$ such that $\lim_{n\rightarrow\infty}{J_{\lambda(N_{k(n)})}\left(z_1, \ldots, z_m; N_{k(n)}, \theta\right)}$ has a uniform limit on $\TT_{\delta}^m$, which is a holomorphic function that we call $\Phi(z_1, \ldots, z_m)$. By Proposition $\ref{prop:pointwiseconvergence}$, see also Remark $\ref{rem:deltaremark}$, the function $\Phi(z_1, \ldots, z_m)$ must equal $\prod_{i=1}^m{\Psi(z_i; \omega, \theta)}$ in the subset
\begin{equation*}
\begin{gathered}
D(m) = \{(z_1, \ldots, z_m) \in \R^m : 1 > z_m > z_{m-1} > \ldots > z_1 > (1 - \delta)^{1/2}, \\
z_3 > z_1/z_2, \cdots, z_m > z_{m-2}/z_{m-1}\}.
\end{gathered}
\end{equation*}
Clearly $D(m)\subset \TT_{\delta}^m$ is nonempty and has limit points.

Since $\Psi(z; \omega, \theta)$ is holomorphic on $\TT_{\delta_0}$ and $(1 + \delta)^m < 1 + \delta_0 < (1 + \delta_0)^m \Rightarrow \delta < \delta_0$, then it is also holomorphic on $\TT_{\delta}$. Thus $\prod_{i=1}^m{\Psi(z_i; \omega, \theta)}$ is holomorphic on $\TT_{\delta}^m$, just like $\Phi(z_1, \dots, z_m)$. Therefore, by analytic continuation, the equality $\Phi(z_1, \dots, z_m) = \prod_{i=1}^m{\Psi(z_i; \omega, \theta)}$ must hold on all $\TT_{\delta}^m$.

The analysis above shows that subsequential (uniform on $\TT_{\delta}^m$) limits of $\{J_{\lambda(N)}(z_1, \ldots, z_m; N, \theta)\}_{N \geq 1}$ exist and must always equal $\prod_{i=1}^m{\Psi(z_i; \omega, \theta)}$. Hence the desired uniform limit $(\ref{eqn:finaldesiredthm})$ follows.$\hfill\square$


\begin{thebibliography}{9}

\bibitem[AAR]{AAR}
G. E. Andrews, R. Askey, and R. Roy. Special Functions, Encyclopedia of Mathematics and Its Applications, 71, Cambridge University Press, Cambridge, 1999.

\bibitem[BGG]{BGG}
A. Borodin, V. Gorin, and A. Guionnet. Gaussian asymptotics of discrete $\beta$-ensembles. Publications mathématiques de l'IHÉS (2015): 1-78.

\bibitem[BO12]{BOl}
A. Borodin and G. Olshanski. The boundary of the Gelfand–Tsetlin graph: A new approach. Advances in Mathematics 230.4 (2012): 1738-1779.

\bibitem[BuG15]{BuG1}
Alexey Bufetov and Vadim Gorin. Representations of classical Lie groups and quantized free convolution. Geometric and Functional Analysis 25.3 (2015): 763-814.

\bibitem[BuG16]{BuG2}
Alexey Bufetov and Vadim Gorin. Fluctuations of particle systems determined by Schur generating functions. Preprint, arXiv:1604.01110 (2016).

\bibitem[Co]{C}
E. T. Copson. Asymptotic expansions. No. 55. Cambridge University Press, 2004.

\bibitem[Cu]{C1}
C. Cuenca. Asymptotic Formulas for Macdonald Polynomials and the Boundary of the $(q, t)$-Gelfand-Tsetlin Graph. Preprint, arXiv:1704.02429 (2017).

\bibitem[DF13]{DF13}
M. Doł\k{e}ga, and V. F\'{e}ray. On Kerov polynomials for Jack characters, DMTCS Proceedings (FPSAC 2013), AS:539-550 (2013).

\bibitem[DF16]{DF16}
M. Doł\k{e}ga, and V. F\'{e}ray. Gaussian fluctuations of Young diagrams and structure constants of Jack characters. Duke Mathematical Journal 165.7 (2016): 1193-1282.

\bibitem[DFS]{DFS}
M. Doł\k{e}ga, V. F\'{e}ray, and P. Sniady. Jack polynomials and orientability generating series of maps. S\'{e}minaire Lotharingien de Combinatoire 70 (2014): B70j.

\bibitem[F]{F}
Peter J. Forrester. Log-gases and random matrices (LMS-34). Princeton University Press, 2010.

\bibitem[G12]{G}
V. Gorin. The q-Gelfand-Tsetlin graph, Gibbs measures and q-Toeplitz matrices. Advances in Mathematics, 229 (2012), no. 1, 201-266.

\bibitem[G14]{G0}
V. Gorin. From alternating sign matrices to the gaussian unitary ensemble. Communications in Mathematical Physics 332.1 (2014): 437-447.

\bibitem[GP]{GP}
V. Gorin, and G. Panova. Asymptotics of symmetric polynomials with applications to statistical mechanics and representation theory. The Annals of Probability 43.6 (2015): 3052-3132.

\bibitem[GS]{GS}
V. Gorin, and M. Shkolnikov. Multilevel Dyson Brownian motions via Jack polynomials. Probability Theory and Related Fields 163.3-4 (2015): 413-463.

\bibitem[He]{He}
 G. Heckman. Root systems and hypergeometric functions. II. Compositio Math., 64(3): 353–373, 1987.

\bibitem[HO]{HO}
G. Heckman and E. Opdam. Root systems and hypergeometric functions. I. Compositio Math., 64(3): 329–352, 1987.


\bibitem[L]{L}
M. Lassalle. Jack polynomials and free cumulants. Advances in Mathematics 222.6 (2009): 2227-2269.

\bibitem[M99]{M}
I. Macdonald. Symmetric Functions and Hall Polynomials. Oxford University Press, Oxford, second edition, 1999.

\bibitem[Ne]{Ne}
A. Yu. Neretin. Rayleigh triangles and non-matrix interpolation of matrix beta integrals. Sbornik: Mathematics 194.4 (2003): 515.

\bibitem[Ok]{Ok1}
A. Okounkov. Binomial formula for Macdonald polynomials and applications. Mathematical Research Letters, Vol. 4, Issue 4 (1997), 533-553.

\bibitem[OkOl97]{OO0}
A. Okounkov, and G. Olshanski. Shifted Jack polynomials, binomial formula, and applications. Mathematical Research Letters 4 (1997): 69-78.

\bibitem[OkOl98]{OO1}
A. Okounkov and G. Olshanski. Asymptotics of Jack Polynomials as the Number of Variables goes to Infinity. International Mathematics Research Notices 1998, no. 13, 641-682.

\bibitem[Ol90]{Ol}
G. Olshanski, Unitary representations of infinite-dimensional pairs (G, K) and the formalism of R. Howe, Representation of Lie Groups and Related Topics (A. Vershik and D. Zhelobenko, eds.), Advanced Studies in Contemporary Math. 7, Gordon and Breach Science Publishers, New York etc., 1990, pp. 269–463.

\bibitem[Ol91]{Ol1}
G. Olshanski. On semigroups related ro infinite-dimensional groups. In: Topics in Representation Theory (A. A. Kirillov, ed.) Advances in Soviet Math., vol. 2, Amer. Math. Soc., 1991, pp. 67-101.

\bibitem[OV]{OV}
G. Olshanski, and A. Vershik. Ergodic unitarily invariant measures on the space of infinite Hermitian matrices. Translations of the American Mathematical Society-Series 2 175 (1996): 137-176.

\bibitem[Op88a]{Op1}
 E. Opdam. Root systems and hypergeometric functions. III. Compositio Math., 67(1):21–49, 1988.

\bibitem[Op88b]{Op2}
E. Opdam. Root systems and hypergeometric functions. IV. Compositio Math., 67(2):191–209, 1988.

\bibitem[Op93]{Op}
E. M. Opdam. Dunkl operators, Bessel functions and the discriminant of a finite Coxeter group. Compositio Mathematica 85.3 (1993): 333-373.

\bibitem[Ox]{Ox}
The Oxford Handbook of Random Matrix Theory, edited by G. Akemann, J. Baik, P. Di Francesco, Oxford University Press, 2011.

\bibitem[Pa]{Pa}
G. Panova. Lozenge tilings with free boundaries. Letters in Mathematical Physics 105.11 (2015): 1551-1586.

\bibitem[Pe]{Pe}
L. Petrov. The boundary of the Gelfand-Tsetlin graph: new proof of Borodin-Olshanski's formula, and its q-analogue. Moscow Mathematical Journal, Vol. 14, Issue 1 (2014): 121-160.

\bibitem[Sn]{Sn}
P. \'{S}niady. Structure coefficients for Jack characters: approximate factorization property. Preprint, arXiv:1603.04268 (2016).

\bibitem[St]{St}
R. P. Stanley. Some combinatorial properties of Jack symmetric functions. Advances in Mathematics 77.1 (1989): 76-115.

\bibitem[VK]{VK}
A. M. Vershik and S. V. Kerov. Characters and factor-representations of the infinite unitary group. Sov. Math. Dokl. 26 (1982): 570–574.

\end{thebibliography}
\end{document}